\documentclass[11pt,a4paper]{article}

\usepackage{a4wide}
\usepackage{graphicx}
\usepackage{pict2e}
\usepackage{latexsym}
\usepackage{epsfig}
\usepackage{amssymb}
\usepackage{amstext}
\usepackage{amsgen}
\usepackage{amsxtra}
\usepackage{amsgen}
\usepackage{amsthm}
\usepackage{color}
\usepackage{tikz}

\usepackage{subcaption}

\def\eps{\varepsilon}

\def\RR{\mathbb{R}}

\newtheorem{thm}{Theorem}[section]
\newtheorem{prop}[thm]{Proposition}
\newtheorem{lemma}[thm]{Lemma}

\newtheorem{rem}[thm]{Remark}

\theoremstyle{definition}

\newtheorem{definition}[thm]{Definition}

\numberwithin{equation}{section}

\newcommand{\R}{\mathbb{R}}

\newcommand{\A}{{\cal A}}

\newcommand{\ur}{\underline{r}}
\newcommand{\ub}{\underline{b}}
\DeclareMathOperator*{\argmin}{arg\,min}

\begin{document}

\title{On a cross-diffusion model for multiple species with nonlocal interaction and size exclusion}
\author{Judith Berendsen\thanks{Institut f\"ur Numerische und Angewandte Mathematik, Westf\"alische Wilhelms-Universit\"at (WWU) M\"unster. Einsteinstr. 62, D 48149 M\"unster, Germany. e-mail: \{judith.berendsen,martin.burger,jan.pietschmann\}@wwu.de  } \and Martin Burger$^\mathrm{*}$ \and Jan-Frederik Pietschmann$^\mathrm{*}$ }
\maketitle

\begin{abstract}
The aim of this paper is to study a PDE model for two diffusing species interacting by local size exclusion and global attraction. This leads to a nonlinear degenerate cross-diffusion system, for which we provide a global existence result. The analysis is motivated by the formulation of the system as a formal gradient flow for an appropriate energy functional consisting of entropic terms as well as quadratic nonlocal terms. Key ingredients are entropy dissipation methods as well as the recently developed boundedness by entropy principle.
Moreover, we investigate phase separation effects inherent in the cross-diffusion model by an analytical and numerical study of minimizers of the energy functional and their asymptotics to a previously studied case as the diffusivity tends to zero. Finally we briefly discuss coarsening dynamics in the system, which can be observed in numerical results and is motivated by rewriting the PDEs as a system of nonlocal Cahn-Hilliard equations.


\end{abstract}

\section{Introduction}
Mathematical models with local repulsion and global attraction received strong attention in the last decades, in particular motivated by applications in biology ranging from bacterial chemotaxis (cf. e.g. \cite{burger2006,calvez2006volume,hillen2009user,painter2002volume}) to macroscopic motion of animal groups (cf. e.g. \cite{boi2000modelling,Burger2007,chen2014minimal,mogilner1999non}) as well as applications in other fields of science (cf. e.g. \cite{stanczy2011evolution,herz2016including}). The macroscopic modelling of the density evolution leads to partial differential equations with nonlinear diffusion and an additional nonlocal term. The majority of such models can be formulated as metric gradient flows for the density $\rho$ with some energy functional consisting of a local and a nonlocal interaction term
\begin{equation}
	E(\rho) = \int_\Omega e(\rho) - \rho (K*\rho) ~dx, 
\end{equation}
where $e$ is a convex functional, $K$ a nonnegative interaction kernel, and $\Omega \subset \RR^N$. Throughout the paper, $\Omega$ may always be unbounded and we will explicitly remark when its boundedness is needed. Various results on the analysis of energy minimizers respectively stationary states (cf. \cite{Balague2013,bedrossian2011global,Burger2008,Burger2013,Burger2014,Canizo2015,Carrillo2016,carrillo2016nonlinear,Choksi2015,Fellner2010,kaib2016stationary,Simione2014,Simione2015}) and the gradient flow dynamics of the form (cf. \cite{bedrossian2011local,bertozzi2009existence,bertozzi2011lp,Burger2008,Carrillo2011,chayes2013aggregation,dyson2010existence,zm2014}) 
\begin{equation}
	\partial_t \rho = \nabla \cdot ( M(\rho) \nabla E'(\rho))
\end{equation}
have been achieved in the last years, which led to a good understanding of such models and phenomena. 

Much less is known however in the case of multi-species systems, which received most attention only recently (cf. e.g. \cite{bakhtacontrol,bruna2012excluded, bruna2012diffusion,Burger2010,Griepentrog2004,painter2009continuous,Schlake2011,simpson2009multi}). With different species, the modelling leads to nonlinear degenerate cross-diffusion systems for the densities of all species, again with some nonlocal terms. The majority of work was concerned with the derivation of models including formal and computational studies, rigorous results are so far mainly available without the nonlocal interaction terms. First rigorous studies of stationary problems (cf. \cite{burger2012nonlinear,Burger2016b,Cicalese2015}) show interesting phase separation phenomena, whose dynamics seems rather unexplored so far. In this paper we hence study a nonlocal cross-diffusion model for two species called red (density $r$) and blue (density $b$) for simplicity, which can be derived from a lattice-based microscopic model with size exclusion (cf. \cite{Burger2010,Schlake2011}):  
\begin{align}
	\partial_t r &= \nabla \cdot( \eps (1-\rho) \nabla r + \eps r \nabla \rho + r(1-\rho)\left[ \nabla (c_{11}K\ast r - K\ast b) + \nabla V\right]), \label{eq:crossdiffusion1}\\
	\partial_t b &= \nabla \cdot( D \eps (1-\rho) \nabla b + D \eps b \nabla \rho + D b(1-\rho)\left[\nabla (c_{22}K\ast b - K\ast r) + \nabla V\right]),\label{eq:crossdiffusion2} 
\end{align}
either on the whole space or in a bounded domain supplemented with no-flux boundary conditions. The positive parameter $\eps$ regulates the strength of the diffusion relative to the nonlocal convection terms.
Here $K$ is again the interaction kernel, and the constants $c_{ii}< 0 $, $i=1,2$ measure the strength of self-interaction, while the strength of the cross-interaction is scaled to unity (in accordance with the notation of \cite{Cicalese2015}). The time scaling is chosen such that $r$ has a unit diffusion coefficient, and $D$ is the potentially different diffusion coefficient of $b$. The function $\rho$ is the nonnegative total density
\begin{equation}
	\rho = r+b,
\end{equation}
naturally bounded from above by one. 
This system (for $V \equiv 0$) is a gradient flow for the energy functional
\begin{align}\label{eq:Feps}
 F^\eps (r,\,b)=\eps F^E(r,b) + F^0(r,b).
\end{align}
consisting of the nonlocal interaction
\begin{align}\label{eq:F0}
 F^0 (r,\,b)=\int_{\Omega} c_{11}r(K*r) -r(K*b)-b(K*r)+c_{22}b(K*b)\;dx 
\end{align}
and the entropy term
\begin{align}\label{eq:FE}
 F^E (r,\,b)=\int_{\Omega} r\log r + b\log b + (1-\rho)\log (1-\rho)\;dx.
\end{align}
For $\Omega$ unbounded we need a confining potential $V$ and the energy is modified to
\begin{align}\label{eq:FC}
 E^\eps (r,\,b)=F^\eps (r,\,b)+F^C (r,\,b), \qquad F^C (r,\,b)=\int_{\Omega} (r+b)V\;dx.
\end{align}
We shall from now on always use the letter $F$ to denote the energy without confining potential and $E$ if the potential is present.
This energy is to be considered on the set of bounded densities with given mass
\begin{align}\label{eq:A}
&\A=\{(r,\,b) \in L^1(\Omega, \mathbb{R}^+)^2: \\\nonumber
&\int_{\mathbb{R}^N}r\;dx=m_r,\int_{\mathbb{R}^N}b\;dx=m_b, \rho=r+b \leq 1 \text{ for a.e. }x\in \Omega  \},
\end{align}
which can be shown to be invariant under the dynamics of \eqref{eq:crossdiffusion1}, \eqref{eq:crossdiffusion2}.
The special case of minimizing $F^0$ on $\A$ was recently investigated by Cicalese et. al. \cite{Cicalese2015} and it appears obvious that $F^\eps$ is the natural entropic version and \eqref{eq:crossdiffusion1}, \eqref{eq:crossdiffusion2} the natural dynamic model leading to such a minimization problem in the large time asymptotics.
\subsection{Connection to Cahn-Hilliard}
In order to see the inherent phase separation in \eqref{eq:crossdiffusion1}, \eqref{eq:crossdiffusion2} we rewrite it as a system of nonlocal Cahn-Hilliard equations. For this sake let $K$ be nonnegative and integrable with 
\begin{equation}\label{eq:k}
	k = \int_{\R^N} K(x)~dx, 
\end{equation}
which allows to define the nonlocal Laplacian $\Delta_K$ as a negative semidefinite operator via
\begin{equation}
	- \Delta_K u = u - \frac{1}k K \ast u.
\end{equation}
Then  \eqref{eq:crossdiffusion1}, \eqref{eq:crossdiffusion2} becomes
\begin{align}
&	\partial_t r = \nabla \cdot ( k r (1- \rho) \nabla( c_{11} \Delta_K r - \Delta_K b + \partial_r W(r,b))) \\
&	\partial_t b = \nabla \cdot ( D k b (1- \rho) \nabla( c_{22} \Delta_K b - \Delta_K r + \partial_b W(r,b)))
\end{align}
with the multi-well potential
\begin{equation} \label{eq:Wdefinition}
	W(r,b) = \eps ( r \log r + b \log b + (1-\rho) \log(1-\rho)) + \frac{c_{11}}2 r^2 - r b + \frac{c_{22}}2 b^2 - \frac{c_{11}}2 r - \frac{c_{22}}2 b.
\end{equation}
Note that for $\eps=0$ the potential has three global minimizers $(r,b)$ in the corners of the unit triangle, i.e. $(0,1)$, $(1,0)$ and $(0,0)$. The first two correspond to segregated states and the third one to void, hence those are the main structures to be expected as energy minimizers and in the long time asympotics. For small $\eps > 0$ the minima are shifted to the interior of the unit triangle and we thus expect less pronounced segregation. For large $\eps$ the energy becomes convex and hence we expect mixing instead of segregation. As in other Cahn-Hilliard equations we expect dynamics at two time scales:  a short one where clusters - in this case separated in red and blue - appear, and a large one where coarsening dynamics of the clusters appear. In the single species case of the model  \eqref{eq:crossdiffusion1}, \eqref{eq:crossdiffusion2}, e.g. obtained for $b \equiv 0$, the coarsening dynamics has been discussed in detail in \cite{burger2008asymptotic,dolak2005keller}, coarsening rates for a related nonlocal model were derived in \cite{slepcev2008coarsening}. Related work for systems of local Cahn-Hilliard equations (partly with constant mobility) was carried out in \cite{bronsard1998multi,elliott1997diffusional,garcke1998anisotropic,garcke2000singular,Griepentrog2004}, supplemented by numerical simulations in \cite{barrett2001fully,nurnberg2009numerical}.
\subsection{Organization of the paper}
The aim of this work is to analyse the entropic regularization, with parameter $\eps$, of the functional $F^0$. We takle this problem from two different perspectives: First we consider the regularized energy from a variational point of view. Second, we study the system of non-local cross-diffusion equations which naturally occur as the gradient flow with respect to this energy.\\
To this end, we will start by discussing the energy minimization of $F^\eps$ respectively $E^\eps$, we briefly comment on the results in \cite{Cicalese2015} and minor extensions, moreover we verify the $\Gamma$-convergence of $E^\eps$ to $E^0$. We also prove the existence of minimizers of $E^\eps$ and their convergence. This is supplemented by a numerical study of the variational problem based on energy minimization via splitting methods.  We further show the existence of weak solutions for the transient model \eqref{eq:crossdiffusion1}, \eqref{eq:crossdiffusion2}, extending the results of \cite{Burger2010} to nonlocal interactions and simplifying the line of the proof via appropriate time discretization with regularizing terms in primal and dual (entropy) variables. Moreover, we provide a regularity result in the case of equal diffusivities and verify consistency of the stationary problem with the energy minimization. Finally we study the dynamics in the particularly relevant case of $\eps \rightarrow 0$ by formal and numerical methods. In the single species case metastable coarsening dynamics of clusters are already studied in detail (cf. \cite{dolak2005keller,burger2008asymptotic}) and we obtain similar behaviour in the multi-species setting.

\section{Minimizers of the Energy Functionals}\label{sec:min}
We start by stating assumptions on the interaction kernel $K$ and the potential $V$. Furthermore, we define
\begin{definition}[Coulomb kernel]\label{def:coulomb} The Coulomb kernel on $\RR^N$ is given by
\begin{align*}
 K(x) := \left\{\begin{array}{ll}
                 -\frac{1}{2}|x| & \text{for } N=1,\\
                 -\frac{1}{2\pi}\log |x| & \text{for } N=2,\\
                 \frac{1}{(N-2)\omega_N}\frac{1}{|x|^{N-2}} & \text{for } N\ge 3,
                \end{array}\right.
\end{align*}
where $\omega_N$ is the $N$-dimensional measure of the unit ball.
\end{definition}
While the Coulomb kernel (or Newton-Potential) is a prominent example for our model, we can extend our arguments to a more general class.
\begin{definition}[Admissible Kernel]\label{def:admissible} For $\Omega\subset \RR^N$, possibly unbounded, we say that a kernel $K$ is admissible if the following conditions are satisfied
\begin{itemize}
 \item[\textbf{(K1)}] $K \in W^{1,1}_{loc}(\Omega)$.
 \item[\textbf{(K2)}] K is radially symmetric, i.e. $K(x)=k(|x|)$ and $k$ is non-increasing.
 \item[\textbf{(K3)}] As $x\to 0$ and $x\to\infty$, $K$ behaves at most as singular as the Coulomb kernel.
\end{itemize}
\end{definition}
Since $V$ should act as a confining potential, we impose the following growth condition at infinity.
\begin{flushleft}
 \textbf{(V1)} Let $V\in W^{1,\infty}_{loc}(\RR^N)$ be nonnegative almost everywhere and such that 
 $$
 \liminf_{|x|\to\infty} \frac{V(x)}{|x|^2}>0.
 $$
\end{flushleft}
%
\subsection{Existence of Minimizers}
Let us first summarize the results of Cicalese et. al. concerning the existence of a minimizer of solely the functional $F^0$ on the whole space $\RR^N$.
\begin{thm}[\cite{Cicalese2015}]\label{thm:novaga}
Let $c_{11}, \, c_{22} \le 0$ and $K \in L_{loc}^1(\RR^N)$ and $\A$ as defined in \eqref{eq:A}. Then there exists a minimizer $(r^0,\,b^0) \in \mathcal{A}$ to the functional
\begin{align*}
F^0 (r,\,b)=\int_{\Omega} c_{11}r(K*r) -r(K*b)-b(K*r)+c_{22}b(K*b)\,dx. 
\end{align*}
More precisely let $(r_n, \,b_n)$ be a minimizing sequence. Then there is a unique sequence of translations $(\tau_n) \subset \mathbb{R}^N$ such that up to a subsequence
$$(r_n(\cdot - \tau_n),b_n(\cdot - \tau_n)) \to (r^0,\,b^0)$$
as $n \to \infty$.
\end{thm}
\begin{rem} Note that due to the translation invariance of $F^0$, the results of \cite{Cicalese2015} allow for arbitrary shifts of the minimizers. Adding a confining potential, i.e. replacing $F^0$ by $E^0 = F^0 + F^C$, however, does not change their result and removes this technicality.
\end{rem}
Next, we extend this existence result to the functionals $F^{\eps}$ and $E^{\eps}$, respectively. 
%
\begin{thm}{(Existence of Minimizer)} \label{thm:existence}
Let $c_{11}, \, c_{22} \le 0$ and $K$ be admissible in the sense of definition \ref{def:admissible}. Then there exists at least one minimizer $(\ur,\,\ub) \in \mathcal{A}$ to the functional
\begin{align*}
E^\eps (r,b)&= \eps F^E(r,b) + F^0(r,b) + F^C(r,b).
\end{align*} 
More precisely any minimizing sequence $(r_n, \,b_n)$ converges $L^1(\Omega)-$weakly to a minimizer $(\ur,\,\ub)$ of $E^\eps$ as $n \to \infty$.\\ If, in addition, $\Omega$ is bounded, the same result holds for the functional $F^\eps$, i.e. the energy without confining potential.
\end{thm}
 While the proof of Theorem \ref{thm:novaga} uses Lions' concentration compactness principle, our approach is based on Dunford-Pettis, making use of the confining potential when $\Omega$ is unbounded. 
\begin{thm}{(Dunford-Pettis theorem)}\label{lem:concentration}
Let $\mathcal{F}$ be a bounded set in $L^1(\Omega).$ Then $\mathcal{F}$ has
compact closure in the weak topology $\sigma(L^1(\Omega), L^{\infty}(\Omega))$ if and only if $\mathcal{F}$ is equi-integrable,
that is,
\begin{align*}
(i) &\quad \forall\, \eps > 0 \; \exists\, \delta > 0 \text{ such that } \int_{U} f \,dx  < \eps\; \forall U \subset \Omega, \text{ measurable with } |U| < \delta, \forall f\in \mathcal{F}   \\
(ii) &\quad \forall\, \eps > 0 \; \exists\, \omega \subset \Omega \text{ measurable with } |\omega|<\infty \text{ such that } \int_{\Omega \setminus \omega} f \, dx < \eps, \forall f\in \mathcal{F}.  
\end{align*}
\end{thm}
\begin{rem} We carry out the details of the proof for $K$ being the Coulomb kernel as in Definition \ref{def:coulomb}, only. 
Due to the properties of admissible kernels, the arguments will not change for general $K$.
\end{rem}
\begin{proof}[Proof of Theorem \ref{thm:existence}]
We follow the direct method of calculus of variations. First note that there exist functions $(r^*,b^*)\in \A$ such that $E^\eps(r^*,b^*) < \infty$, take for example
\begin{align*}
 r^* = \frac{c_1e^{-|x|^2}}{1+c_1e^{-|x|^2} + c_2e^{-|x|^2}},\quad b^* = \frac{c_2e^{-|x|^2}}{1+c_1e^{-|x|^2} + c_2e^{-|x|^2}},
\end{align*}
where the constants $c_1,c_2$ are chosen to fix the appropriate masses. 
Next we show that $E^\eps$ is bounded from below on $\A$. We focus on $\eps F^{E}+F^C$ first using a relative entropy argument. 
\begin{align*}
 \int_\Omega \left( \eps r\log r + rV\right) 
 &= \int_\Omega \eps r\log\left(\frac{r}{e^{-\frac{V}{2\eps}}}\right) +r \frac{V}{2}\; dx\\
 &= \int_\Omega \eps \frac{r}{e^{-\frac{V}{2\eps}}}\log\left(\frac{r}{e^{-\frac{V}{2\eps}}}\right)e^{-\frac{V}{2\eps}}+r \frac{V}{2}\;dx\\
 &= \int_\Omega e^{-\frac{V}{2\eps}}\;dy \int_\Omega \eps \frac{r}{e^{-\frac{V}{2\eps}}}\log\left(\frac{r}{e^{-\frac{V}{2\eps}}}\right)\frac{e^{-\frac{V}{2\eps}}}{\int_\Omega e^{-\frac{V}{2\eps}}\;dy}+ r \frac{V}{2}\;dx\\ 
 &\ge  \eps  m_r\log(\frac{m_r}{\int_{\Omega} e^{-\frac{V}{2\eps}}\,dx}),
\end{align*}
where we applied Jensen's inequality to the convex function $r\log r$ and used the integrability of $e^{-V}.$ Furthermore we used that $\int_\Omega r V \;dx \ge 0,$ because $V$ is assumed to be nonnegative and $r\ge 0$ almost everywhere. The same argument holds for $b\log b$. In addition we have
\begin{align}
 \int_\Omega \eps (1-\rho)\log (1-\rho)\;dx &= \int_\Omega \eps ((1-\rho)\log (1-\rho) - (1-\rho) + 1 - \rho)\;dx \\\nonumber
 &\ge -\eps(m_r+ m_b),
\end{align}
where we used that the function $z \log z - z + 1$ is non-negative on $[0,1]$.
%
%
To show that $F^0$ is also bounded from below note that for $N=1$, we know that $F^0$ is larger than zero since $r,\, b$ and the absolute value are positive functions. For $N=2$ we apply the logarithmic Hardy-Littlewood-Sobolev inequality \cite[lemma 5]{bedrossian2011local} using again the potential to bound the logarithmic terms. Finally, for $N \geq 3$ we use the Hardy-Littlewood-Sobolev type inequality stated in \cite[lemma 4]{bedrossian2011local} with $t=1$ and $p=q=2$. Thus, we are allowed to take minimizing sequence $(r_n,b_n)$ in $\A$ which, for $n$ large enough, is always smaller than or equal to $(r^*,b^*)$. This implies
\begin{align*}
 -\infty &< \int_\Omega \eps( r_n\log r_n + b_n\log b_n + (1-\rho_n)\log(1-\rho_n) + \frac{1}{2\eps}(r_n+b_n)V) \;dx + F^0(r_n,b_n) \\
 &+ \int_\Omega \frac{1}{2}(r_n+b_n)V\;dx  \le E^\eps(r^*,b^*) < \infty,
\end{align*}
and since $\int_\Omega r V \;dx \ge 0$, we also conclude the boundedness of $\int_\Omega (r_n+b_n)V\;dx$ which, by assumption \textbf{(V1)} implies a second moment bound on $r_n$ and $b_n$.\\
We are now in a position to apply Theorem \ref{lem:concentration} to conclude the existence of a weakly converging subsequence. We apply the theorem to $\mathcal{F}$ being the set $\A$ intersected with the set of all functions having bounded second moment, which is a bounded subset of $L^1(\Omega).$ In order to show condition (i) we exploit the fact that every element in $\A$ is bounded from above by one i.e. we can take $\eps = \delta.$. For the second condition we employ the second moment bounds. Consider a ball $B_R$ with radius $R$ in $\Omega.$ We obtain
\begin{align*}
 \int_{\Omega\setminus B_R} r_n \,dx = \int_{\Omega\setminus B_R} \frac{|x|^2}{|x|^2} r \,dx \leq \frac{C}{R^2} < \eps,
\end{align*}
for $R$ sufficiently large.
That means the minimizing sequence $(r_n,\, b_n)$ must have a subsequence $(r_{n_j},\, b_{n_j})_j$ weakly converging to some limit $(\ur,\ub)$ in $L^1(\Omega).$ The functional $F^0$ is weakly lower semicontinuous this is shown in \cite[Theorem 2.1]{Choksi2015} for $N=1$ and $N=2$ and in Lemma 3.3 in \cite{Choksi2015} for $N\geq 3$. Since $F^\eps$ is convex, this implies the lower semicontinuity of $E^\eps$ which allows us to conclude that $(\ur,\ub)$ are indeed minimizers of $E^\eps$ which also implies that they are contained in $\A$.\\
%
%
For $\Omega$ bounded, the boundedness from below for the entropic terms $F^\eps$ is trivial using the fact that $z\log z \in L^\infty(\Omega)$ for all $z\in [0,1]$, while the arguments for $F^0$ remain unchanged. Since the second moment bound is also a trivial consequence of the $L^\infty$ bounds in $\A$, all remaining arguments are also still valid.
\end{proof}

\subsection{Structure of Energy Minimizers}

The energy minimizers of $F^0$ on $\A$ could be characterized explicitly in some important cases in \cite{Cicalese2015} using radial symmetrization techniques. In particular they showed that there is no phase separation in the  case of cross-attraction being larger than the self-attractions, i.e. $c_{11}  > -1$ and $c_{22} > -1$. On the other hand they obtain a strong phase separation result if $c_{11} + c_{22} < -2$ and one of the self-attractions is weak ($c_{ii} > -1$). In this case the energy minimizer consists of the strongly attracting species having density one in a ball and the other one having density in a spherical shell around it. The case of strong self-attraction of both species, $c_{11}<-1$, $c_{22} < -1$ was left open in multiple dimension however, noticing that already in dimension one the minimizer is not spherically symmetric. We supplement these results by a generic phase separation result, whose proof closely follows a related result in \cite{Burger2016b} as well as numerical simulations for $N=2$ in Section \ref{sec:nummin}:

\begin{prop} 
Let $(\underline{r},\underline{b}) \in \A$ be a minimizer of $F^0$ on this admissible set, for $K$ being a radially strictly decreasing kernel and $c_{11} + c_{22} < -2$. 
Then the intersection of the supports of $\underline{r}$ and $\underline{b}$ on $\Omega \subset \RR^N$ possibly unbounded has zero Lebesgue measure.
\end{prop}
\begin{proof}
Assume by contradiction that there exists an set $D$ of positive Lebesgue measure and $\delta_0 \in (0,\frac{1}2)$ such that $\underline{r} > \delta_0$ and $\underline{b} > \delta_0$ on $D$. Let $d$ be the diameter and $R < \frac{d}6$ sufficiently small. Then there exist two balls $B_R(x_1)$ and $B_R(x_2)$ such that $|x_1 - x_2| > 4R$ and the intersection of both balls with $D$ has nonzero Lebesgue measure. Define $S_i=B_R(x_i) \cap D$. For $|\delta| < \min\{{\delta_0}{|S_1|},{\delta_0}{|S_2|})$ define variations $r^\delta$ and $b^\delta$ that coincide with $\underline{r}$ respectively $\underline{b}$ outside $S_1 \cup S_2$ and 
$$ r^\delta = \left\{\begin{array}{ll} \underline{r} + \frac{\delta}{|S_1|} & \text{in } S_1 \\ \underline{r} -  \frac{\delta}{|S_2|} & \text{in } S_2 \end{array} \right. \qquad   b^\delta = \left\{\begin{array}{ll} \underline{b} -  \frac{\delta}{|S_1|} & \text{in } S_1 \\ \underline{b} +  \frac{\delta}{|S_2|}& \text{in } S_2 \end{array} \right. $$ 
It is straightforward to see that $(r^\delta,b^\delta) \in \A$ and 
\begin{align*} 
F^0(r^\delta,b^\delta)&=F^0(\underline{r},\underline{b})  \\ 
&\quad + \delta \ell(\underline{r},\underline{b}) +\delta^2 (c_{11} + c_{22} + 2)\Big( \frac{1}{|S_1|^2}\int_{S_1} \int_{S_1} K(x-y)~dx~dy \\
&\quad + \frac{1}{|S_2|^2} \int_{S_2} \int_{S_2} K(x-y)~dx~dy-\frac{2}{|S_1|~|S_2|}\int_{S_1} \int_{S_2} K(x-y)~dx~dy\Big)
\end{align*} 
for some bounded linear functional $\ell$. From $F^0(r^\delta,b^\delta) \geq F^0(\underline{r},\underline{b})$ we conclude in the limit $\delta \rightarrow 0$ we have $\ell(\underline{r},\underline{b}) = 0$. The strict radial decrease of the kernel $K$ implies
$$ \frac{1}{|S_i|^2} \int_{S_i} \int_{S_i} K(x-y)~dx~dy    \geq  \inf_{|z|\leq 2R} K(z) $$ 
for $i \in\{1,2\}$, and
$$ \frac{2}{|S_1|~|S_2|}\int_{S_1} \int_{S_2} K(x-y)~dx~dy < 2 \sup_{|z|\geq 2R} K(z) \leq 2 \inf_{|z|\leq 2R} K(z), $$ 
hence the above difference of integrals is positive. With  $c_{11} + c_{22} + 2 < 0$ we then conclude for finite $\delta$ that 
$$ F^0(r^\delta,b^\delta) < F^0(\underline{r},\underline{b}) , $$
which is a contradiction.
\end{proof}

The entropic terms in the energy $F^\eps$ counter this separation effect on the minimizers. We can show this complete change of the situation on bounded domains. 
\begin{lemma} \label{lem:delta} Let $\Omega\subset \RR^N$ be open and bounded. For fixed $\eps > 0$ and $K$ admissible, let $(\ur,\ub) \in \mathcal{A}$ be minimizers of $F^\eps$. Then there exists a positive constant $\delta$ which depends on $\eps$, $C_K:=\int_\Omega \int_\Omega K(x-y)\;dxdy$ only such that 
 \begin{align}
   \delta \le \ur,\ub,\quad \underline{\rho} \le 1 - \delta.
 \end{align}
\end{lemma}
\begin{proof}
We argue by contradiction and thus assume that $\mathrm{ess} \inf_{x\in\Omega} \ur = 0$. Then there exists a set $\mathcal{M}_1$ with positive measure and a constant $\delta_{max}$, such that for  $0<\delta\le\delta_{max}$
 \begin{align}
  \ur \le \delta \text{ for a.e. } x \in \mathcal{M}_1
 \end{align}
Since $m_r >0$ there also exists, for $\delta$ sufficiently small, a set $\mathcal{M}_2$ such that 
\begin{align}
 \ur \ge \sqrt{\delta}\text{ for a.e. } x \in \mathcal{M}_2.
\end{align}
Without loss of generality, we assume that both sets have the same Lebesgue measure, i.e. $|\mathcal{M}_1| = |\mathcal{M}_2| =: C_{\mathcal{M}}$ (this is always possible since if one of the two sets has measure larger than the other, we just chose a smaller subset such that the measures are equal). 
We distinguish between the two cases $\underline{\rho} < \frac{1}{2}$ and $\underline{\rho} \geq \frac{1}{2}$ on $\Omega.$ 
\\First assume that $\underline{\rho} \geq \frac{1}{2}$ on $\Omega.$
Then, we define the functions $(\tilde r,\tilde b)$ by
\begin{align*}
 \tilde{r} &= \ur + \delta,\; \tilde b = \ub - \delta\text{ on }\mathcal{M}_1,\\
 \tilde{r} &= \ur - \delta,\; \tilde b = \ub + \delta\text{ on }\mathcal{M}_2,\\
  \tilde{r}&= \ur,\;\tilde b = \ub\text{ on }\Omega \setminus (\mathcal{M}_1 \cup \mathcal{M}_2)
\end{align*}
This definition implies that $\tilde r$ and $\tilde b$ have the same mass as $\ur$ and $\ub$ on $\Omega$ and also that $1-\tilde \rho =  1-\underline{\rho}$ almost everywhere on $\Omega.$ We also see that on $\mathcal{M}_1$
$$0 < \tilde r \leq 2\delta <1 \text{ and } 0< \frac{1}{2}-2\delta_{max}<\tilde{b}<1-\delta<1.$$ 
Analogously on $\mathcal{M}_2$ we see that
$$0 < \sqrt{\delta}-\delta \leq \ur-\delta = \tilde{r} <1-\delta \text{ and } 0< \ub < \tilde{b}+\delta = \rho +\delta - \ur \leq \frac{1}{2} + \delta - \sqrt{\delta}<1.$$
On $\Omega \setminus (\mathcal{M}_1 \cup \mathcal{M}_2)$ the box constraints for $(\tilde{r},\tilde{b})$ trivially hold, because $(\ur,\ub) \in \mathcal{A}.$
Therefore $(\tilde r,\, \tilde b)\in \mathcal{A}$ holds. Our goal is to show that $F^\eps(\tilde r,\tilde b) < F^\eps(\ur,\ub)$ which is a contradiction to the minimality of $(\ur,\ub)$. To this end, we examine each term in $F^\eps$ separately. First note that by the mean value theorem, we have for any $x$ such that $(x\pm \delta)\in[0,1]$
\begin{align*}
 (x + \delta)\log(x + \delta) &= x\log(x) + (\log \xi + 1)\delta \\
 (x - \delta)\log(x - \delta) &= x\log(x) - (\log \xi' + 1)\delta 
\end{align*}
for some $\xi \in (x, \, x+\delta)$ and $\xi' \in (x-\delta, \,x)$, respectively.
Thus, since on $\ur \le \delta$ on $\mathcal{M}_1$, we have 
\begin{align}\label{M1r+}
 \int_{\mathcal{M}_1} \tilde r\log \tilde r\;dx &= \int_{\mathcal{M}_1}\ur\log\ur + (\log \xi + 1)\delta \;dx \le  \int_{\mathcal{M}_1}\ur\log\ur + (\log (\ur+\delta) + 1)\delta \;dx,\nonumber \\
 &\le \int_{\mathcal{M}_1}\ur\log\ur + (\log (2\delta) + 1)\delta \;dx = \int_{\mathcal{M}_1}\ur\log\ur\;dx + C_{\mathcal{M}}(\log (2\delta) + 1)\delta ,\\
 \int_{\mathcal{M}_1} \tilde b\log \tilde b\;dx &= \int_{\mathcal{M}_1}\ub\log\ub -(\log \xi' + 1)\delta \;dx,  \nonumber\\
 &\le \int_{\mathcal{M}_1}\ub\log\ub \;dx - C_{\mathcal{M}} (\log (\frac{1}{2} - \delta_{max}) +1) \delta \nonumber
\end{align}
On $\mathcal{M}_2$, using $r\ge \sqrt{\delta}$ yields
\begin{align} \label{M2-}
 \int_{\mathcal{M}_2} \tilde r\log \tilde r\;dx &= \int_{\mathcal{M}_2}\ur\log\ur -(\log \xi' + 1)\delta \;dx ,\nonumber\\
 &\le \int_{\mathcal{M}_2}\ur\log\ur\;dx -C_{\mathcal{M}}(\log (\sqrt{\delta}-\delta) + 1)\delta \\
 \int_{\mathcal{M}_2} \tilde b\log \tilde b\;dx  &= \int_{\mathcal{M}_2}\ub\log\ub  + (\log \xi + 1)\delta \;dx, \nonumber\\
 &\le \int_{\mathcal{M}_2}\ub\log\ub  \;dx + C_{\mathcal{M}}(\log (1+\delta) + 1)\delta \nonumber,
\end{align} 
we thus obtain 
\begin{align} \label{Summary}
 &\quad \eps\int_\Omega \tilde r \log \tilde r + \tilde b \log \tilde b + (1-\tilde \rho)\log(1-\tilde \rho)\;dx \nonumber \\ 
 &\le \eps\int_\Omega \ur \log \ur + \ub \log \ub + (1-\underline{ \rho})\log(1-\underline{\rho})\;dx\\ \nonumber
 &\quad + \delta C_{\mathcal{M}}(\log(2\delta) -\log (\frac{1}{2} - \delta_{max})-\log (\sqrt{\delta}-\delta)+\log (1+\delta)).
\end{align}
Now consider $\rho < \frac{1}{2}$ on $\Omega.$ We only have to adjust the definition of $\tilde b$ on $\mathcal{M}_1,$ because $\ub$ can be smaller than $\delta$ in this case. Instead we use   
\begin{align*}
 \tilde{r} &= \ur + \delta,\; \tilde b = \ub \text{ on }\mathcal{M}_1,\\
 \tilde{r} &= \ur - \delta,\; \tilde b = \ub \text{ on }\mathcal{M}_2,\\
  \tilde{r}&= \ur,\;\tilde b = \ub\text{ on }\Omega \setminus (\mathcal{M}_1 \cup \mathcal{M}_2)
\end{align*}
The mass constraints and box constraints of $\mathcal{A}$ for $(\tilde{r},\tilde{b})$ on $\Omega$ follow analogously. Yet the total density changes to $\tilde{\rho}=\underline{\rho} + \delta$ on $\mathcal{M}_1$, respectively $\tilde{\rho}=\underline{\rho} - \delta$ on $\mathcal{M}_2$ and $\tilde{\rho}=\underline{\rho}$ on $\Omega \setminus (\mathcal{M}_1 \cup \mathcal{M}_2).$
We begin by considering the $F^E$ on $\mathcal{M}_1.$ The estimate for $\tilde{r}$ is the same as \eqref{M1r+} and $\tilde{b}=\ub.$ Furthermore $\tilde{\rho}= \underline{\rho} + \delta$ yields
\begin{align*}
 \int_{\mathcal{M}_1} (1-\tilde \rho) \log (1-\tilde \rho) \;dx &= \int_{\mathcal{M}_1}(1-\underline{\rho})\log(1-\underline{\rho} )- (\log \xi' + 1)\delta \;dx,\\
 &\le  \int_{\mathcal{M}_1}(1-\underline{\rho})\log(1-\underline{\rho} )\;dx -C_{\mathcal{M}}(\log(\frac{1}{2}-\delta_{max}) + 1)\delta .
\end{align*}
Now on $\mathcal{M}_2$ we look at \eqref{M2-} for the term $\tilde{r}$ and at \eqref{M1r+} for $\tilde{\rho}.$
Therefore \eqref{Summary} is not changed if $\underline{\rho}< \frac{1}{2}.$ 
Arguing similarly for the nonlocal terms $F^0$ using the box constraints on $\mathcal{A}$ yields 
\begin{align*}
 \int_{\mathcal{M}_1} c_{11}\tilde{r}(K* \tilde r) \;dx &\le c_{11}(\int_{\mathcal{M}_1}r(K*r) \;dx-2\delta C_K +\delta^2C_K) 
\end{align*}
and the analogue for $\tilde b$ on $\mathcal{M}_2$ as well as $\tilde \rho$ on $\mathcal{M}_1$ when $\underline{\rho} < \frac{1}{2}.$
The other terms follow with the necessary adjustments for the respective signs.
We obtain
\begin{align*}
 F^0(\tilde r, \tilde b) \le F^0(\ur, \ub) +(8- 2c_{11}-2c_{22})C_K\delta+8C_K \delta^2.
\end{align*}
In other words there exists a constant $C=C(C_{\mathcal{M}},C_K)$ such that
$$F^\eps(\tilde r,\tilde b) \le F^\eps(\ur, \ub) + \delta(C_{\mathcal{M}}\log(2\delta) + C)$$
For $\delta$ small enough, the constant term on the right-hand-side becomes negative which contradicts the fact that $(\ur, \ub)$ are minimizers. The cases $\mathrm{ess} \inf_{x\in\Omega} \ub = 0$ and $\mathrm{ess} \inf_{x\in\Omega} 1-\underline{\rho} = 0$ can be treated by the same argument and thus we conclude the assertion.
\end{proof}
The previous lemma allows us to easily construct variations that of the minimizers that stay in $\A$ and thus the derivation of the first variation.
\begin{prop}[First Variation] Let $(\ur, \, \ub)$ be a minimizer of $F^\eps$ in $\mathcal{A}$ on the open and bounded domain $\Omega\subset \RR^N$. Then for any $\phi \in L^{\infty}(\Omega)$ with $\int_{\Omega}\phi \;dx = 0$, 
we have
\begin{align*}
\eps \int_{\Omega} \phi (\log \ur - \log(1-\underline{\rho})) +2 \phi (c_{11}(K\ast \ur)-(K\ast \ub))\; dx = 0,
\end{align*}
and the analogue result for $\ub$. This implies that there exist constants $C_1,\,C_2$ that only depend on $\eps,\,\ur,\,\ub,\,c_{11},\,c_{22}$ and $\Omega$ such that, for a.e. $x\in\Omega$
\begin{align}\label{eq:firstvar}
\eps [\log \ur - \log(1-\underline{\rho})] +2(c_{11}(K\ast \ur)-(K\ast \ub)) &= C_1|\Omega|,\\\nonumber
\eps [\log \ub - \log(1-\underline{\rho})] +2(c_{22}(K\ast \ub)-(K\ast \ur)) &= C_2|\Omega|,
\end{align}
holds. 
\end{prop}
\begin{proof}
By Lemma \ref{lem:delta} we know that $(\ur + t\phi, \, \ub) \in \mathcal{A}$ for small enough $t.$ Since $(\ur, \, \ub)$ is a minimizer for $F^{\eps}$ the difference $F^{\eps}(\ur+t \phi,\, b)-F^{\eps}(\ur,\, \ub)$ must be nonnegative. This implies
\begin{align*}
&0 \le \lim_{t \to 0} \frac{F^{\eps}(\ur+t \phi,\, b)-F^{\eps}(\ur,\, \ub)}{t}  \\
&= \lim_{t \to 0} \eps \int_{\Omega}\frac{(\ur+t\phi)\log(\ur+t\phi)-\ur \log \ur }{t}+\frac{((1-\underline{\rho})+t\phi)\log((1-\underline{\rho})+t\phi)-(1-\underline{\rho})\log (1-\underline{\rho}) }{t} \; dx \\
&\quad +\int_{\Omega} c_{11}\ur(K*\phi)+c_{11}\phi(K*\ur) +t c_{11} \phi (K*\phi) - \phi(K* \ub)-\ub (K*\phi) \;dx\\
&=\int_{\Omega} \phi\left(\eps (\log \ur - \log(1-\underline{\rho})) +2c_{11}(K*\ur) - 2(K*\ub)\right)\;dx.
\end{align*}
We can repeat the argument for $-\phi$ to obtain that the integral is zero. 
By Lemma \ref{lem:delta} we know that the integrand
$$f(\ur, \ub)=\eps [\log \ur - \log(1-\underline{\rho})] +2(c_{11}(K\ast \ur)-(K\ast \ub))$$
is an element of $(L^{\infty}(\Omega))^2.$ Therefore we can take $\phi=f-C_1$ with $C_1=|\Omega|^{-1}\int_{\Omega}f dx$ and obtain
\begin{align*}
\int_{\Omega} |f-C_1|^2 \;dx=0.
\end{align*}
The same holds for the second integrand and the result follows.
\end{proof}
\begin{rem}[Regularity of minimizers]\label{rem:regmin} Let $(\ur, \, \ub)$ be a minimizer of $F^\eps$ in $\mathcal{A}$ on the open and bounded domain $\Omega\subset \RR^N$ as above.
Solving \eqref{eq:firstvar} for $\ur$ yields the nonlinear expression
\begin{align}\label{eq:rexpl}
 \ur &= \frac{e^{\frac{2}{\eps}(-c_{11}K\ast \ur+K \ast \ub+\frac{C_1}{2}|\Omega|)}}{1+e^{\frac{2}{\eps}(-c_{11}K\ast \ur+K\ast \ub+\frac{C_1}{2}|\Omega|)}+e^{\frac{2}{\eps}(-c_{22}K\ast \ub+K\ast \ur+\frac{C_2}{2}|\Omega|)}}.
\end{align}
Since $K \in W^{1,1}(\Omega)$ by definition \ref{def:admissible}, Young's inequality for convolutions yields $K \ast r,\, K \ast b \in L^{\infty}(\Omega)$ and $\nabla K\ast \ur,\,\nabla K\ast \ub \in [L^\infty(\Omega)]^N$ since $\ur,\,\ub \in \A$. When taking the gradient of \eqref{eq:rexpl}, only combinations of these terms, concatenated with smooth functions, occur and thus it is also bounded in $L^\infty$. This immediately implies $\ur \in W^{1,\infty}(\Omega)$ and the same assertion for $\ub$ as well.
\end{rem}

\subsection{Convergence as $\eps \rightarrow 0$}
We verify the approximation of $E^0$ with the entropic terms via a $\Gamma$-convergence result and the convergence of minimizers. For further studies of the detailed shapes of the minimizers we refer to the numerical examples in Section \ref{sec:numerics}.
\begin{thm}[$\Gamma$-Convergence] Let $\Omega\subset \RR^N$ be an open and bounded domain. 
 As $\eps \to 0$ and for all $(r,b)\in\A$, the functional
 \begin{align*}
 E^\eps(r,b)
 \end{align*}
$\Gamma$-converges to the functional 
\begin{align*}
 E^0(r,b) 
\end{align*}
with respect to the weak topology in $[L^2(\RR^N)]^2$. Furthermore, minimizers $(r^\eps,b^\eps)$ of $E^\eps$ converge to minimizers $(r_0,b_0)$ of $E^0$ as $\eps \to 0$.
\end{thm}
\begin{proof}
To prove $\Gamma$-convergence, we have to check a $\liminf-$inequality and the existence of a recovery sequence, see \cite[definition 1.5]{Braides2002} for details. To this end consider arbitrary sequences $r^\eps,\,b^\eps \in \mathcal{A}$ that weakly converge to some $(r,b)\in\mathcal{A}$ as $\eps \to 0$. Since on $\mathcal{A}$, the functional $F^E$ is uniformly bounded, is follows that $\eps F^E$ converges to zero. On the other hand, we know that the interaction terms $F^0$ are weakly lower semicontinuous and thus we have
\begin{align}\label{eq:gammasup}
 E^0(r,b)  &\le \liminf_{\eps \to 0} E^\eps(r_\eps,b_\eps).
\end{align}
Choosing for every $(r,b)\in\mathcal{A}$ the constant sequence as recovery sequence, and using once more the fact that $\eps F^E$ converges to zero as $\eps \to 0$ for all $(r,b)\in\mathcal{A}$ we have
\begin{align*}
 \lim_{\eps \to 0} E^\eps(r_\eps,b_\eps)  = E^0(r,b).
\end{align*}
This implies the $\Gamma$-convergence of $E^\eps$ to $E^0$. The fact that the set $\mathcal{A}$ is precompact with respect to the weak topology in $L^2$ trivially implies the coerciveness of $E^\eps$ and therefore, \cite[Theorem 1.21]{Braides2002}, we have the convergence of minimizers as $\eps \to 0$ with respect to the weak topology in $L^2$.
\end{proof}

\subsection{Large Scale Structures}

Finally we discuss the structures appearing as energy minimizers from a large scale point of view, i.e. we rescale space to a macroscopic variable $\tilde x  = \lambda x$ with $\lambda>0$ small and assume that the densities can be rewritten as functions $\tilde r$ and $\tilde b$ of the new variable $\tilde x$. As a consequence of the rescaling we find up to higher order terms in $\lambda$
\begin{align*}
\int_{\R^N} \int_{\R^N} K(x-y) r(x) b(y)~dx~dy &= \int_{\R^N} \int_{\R^N} K(\frac{\tilde x - \tilde y}\lambda) \tilde r(\tilde x) \tilde b(\tilde y) \lambda^{-2 N} ~d\tilde x~d\tilde y \\ 
&\approx \lambda^{-N} \int_{\R^N} k \tilde r(\tilde x) \tilde b(\tilde x) - C \lambda^2 \nabla \tilde r(\tilde x) \cdot \nabla \tilde b(\tilde x)  ~d\tilde x ,
\end{align*}
with $k$ defined in \eqref{eq:k}, $C$ being related to the second moment of $K$. With similar approximation of the self-interaction terms the energy $\lambda^{N} F^0$ equals up to second order 
$$ \tilde F^{\lambda}(\tilde r, \tilde b) = \lambda^2\int_{\R^N} (\frac{a_{11}}2 |\nabla \tilde r|^2 + 
\frac{a_{22} }2 |\nabla \tilde b|^2  + a_{12} \nabla \tilde r \cdot \nabla \tilde b) ~d\tilde x + k \int_{\R^N} W(\tilde r,\tilde b)~d\tilde x, $$
with $W$ as in \eqref{eq:Wdefinition} with $\eps = 0$ and
$$a_{11} = -c_{11}C, \quad  a_{22} = -c_{22}C, \quad a_{12} = C.$$ 
We see that for $c_{11},\, c_{22} > 1$ the symmetric $2 \times 2$ matrix $A$ with diagonal entries $a_{ii}$ und off-diagonal entries $a_{12}$ is symmetric. Hence, we may perform a change of variables to $(u,v) = A^{1/2} (\tilde r, \tilde b)$ and obtain the diagonalized problem 
$$ G^\lambda(u,v) = \frac{\lambda^2}2 \int_{\R^N} |\nabla u|^2 + |\nabla v|^2 - k W(A^{-1/2} (u,v))~d\tilde x $$
to be minimized on the transformed unit interval 
$$ {\cal B} = \{ (u,v) ~|~A^{-1/2} (u,v) \in {\cal A}\}. $$
The asymptotics of problems of this form has been investigated by Baldo \cite{baldo1990minimal}, who showed $\Gamma$-convergence to a multi-phase problem, the limiting functional being a sum of weighted perimeters between the pure phases. Translated to our setting we can conclude that  $\tilde F^\lambda$ $\Gamma$-converges to a functional of the form 
\begin{equation} \label{eq:partitioning}
	\tilde F^0(R,B) = d_{RB} \text{Per}(\partial R \cap \partial B) + d_{RV} \text{Per}(\partial R \setminus \partial B) + d_{BV} \text{Per}(\partial B \setminus \partial R)
\end{equation}
and weights $d_{RB}$, $d_{RV}$, $d_{BV}$ that can be derived from the behaviour of the potential $W$, i.e. the values $c_{11}$, $c_{22}$, on the unit triangle. Thus, we see that again that generic energy minimizers are segregated and minimizing structures are determined by the solutions of multiphase versions of the isoperimetric problems.

\section{Numerical Study of Minimizers}\label{sec:nummin}
In this Section we present a numerical algorithm to solve the optimization problem 
\begin{align}\label{eq:minprob}
 \argmin_{(r,b)\in\mathcal{A}}\; \eps F^E(r,b) + F^0(r,b)  + \int_{\Omega} (r+b)V\;dx ,
\end{align}
We perform our calculation on $\Omega\subset \mathbb{R}^N$, $N=1,2$, an open and bounded domain with smooth boundary $\partial\Omega$. In view of Lemma \ref{lem:delta}, we can then restrict ourselves to the set 
\begin{align}
 \mathcal{A}^\delta := \{(r,\,b) \in &L^1(\mathbb{R}^N, \mathbb{R})^2: \int_{\mathbb{R}^N}r\;dx=m_r,\int_{\mathbb{R}^N}b\;dx=m_b, \\
 &\delta \le r,b,\,\mbox{and}\; \rho=r+b \leq 1-\delta \text{ for a.e. }x\in \Omega  \},
\end{align}
for some $\delta >0$ sufficiently small.
Our numerical algorithm is based on a splitting approach. First, we introduce the sets
\begin{align}
 \mathcal{A}_{M} &:= \{(r,\,b) \in L^2(\Omega)^2: \int_{\Omega}r\;dx=m_r,\int_{\Omega}b\;dx=m_b \},\\
 \mathcal{A}_{B}^\delta &:= \{(r,\,b) \in L^2(\Omega)^2: \delta \le r,\,b,\;\rho=r+b \leq 1-\delta \text{ for a.e. }x\in \Omega \}.
\end{align}
with the corresponding projections
\begin{align}\label{eq:projam}
 P_{\mathcal{A}_M}: L^2(\Omega)^2 \to \mathcal{A}_M, \quad P_{\mathcal{A}_M}(r,b) = \left(\begin{array}{l}
                                                                r - \frac{1}{|\Omega|}(\int_\Omega r\;dx - m_r)\\
                                                                b - \frac{1}{|\Omega|}(\int_\Omega b\;dx - m_b)
                                                               \end{array}\right)
\end{align}
and
\begin{align}\label{eq:projab}
 P_{\mathcal{A}_B^\delta}: L^2(\Omega)^2 \to \mathcal{A}_B^\delta, \quad P_{\mathcal{A}_B^\delta}(r,b) = \left(\begin{array}{l}
                                                                \frac{1}{2}((1-\delta)-(\tilde b-\tilde r))\\
                                                                \frac{1}{2}((1-\delta)+(\tilde b-\tilde r))
                                                               \end{array}\right)
\end{align}
with $\tilde r = \min(\max(r,\delta),1-\delta)$ and $\tilde b = \min(\max(b,\delta),1-\delta)$. Obviously, we have $\mathcal{A}^\delta = \mathcal{A}^\delta_B \cap \mathcal{A}_M$.
Then we consider the problem
\begin{align*}
 \argmin_{(r_1,b_1)\in \mathcal{A}_B^\delta,\,(r_2,b_2)\in \mathcal{A}_M}\; &\eps F^E(r_1,b_1)  + F^0(r_2,b_2)\;dx + \int_{\Omega} (r_2+b_2)V\;dx,\\
 &\mbox{such that}\quad r_1 = r_2,\; b_1 = b_2
\end{align*}
The corresponding augmented Lagrangian \cite{Powel69,Hestenes1969}, with fixed parameter $\mu > 0$ is then given by 
\begin{align*}
 &\mathcal{L}_\mu(r_1,b_1,r_2,b_2,\lambda_r,\lambda_b) = \eps F^E(r_1,b_1) + F^0(r_2,b_2)\;dx + \int_{\Omega} (r_2+b_2)V\;dx \\
    &\quad + \int_\Omega \lambda_r(r_1-r_2)\;dx + \int_\Omega \lambda_b(b_1-b_2)\;dx  + \frac{\mu}{2}\left( \|r_1-r_2\|_{L^2(\Omega)}^2 + \|b_1-b_2\|_{L^2(\Omega)}^2\right)
\end{align*}
To solve this saddle point problem we use the following ADMM scheme \cite{Osher2016}
\begin{align}\label{eq:O1}
 (r_1,b_1)^{n+1} &:= \argmin_{(r,b)\in \mathcal{A}_B^\delta}  \mathcal{L}_\mu(r,b,r_2^n,b_2^n,\lambda_r^n,\lambda_b^n)\\\label{eq:O2}
 (r_2,b_2)^{n+1} &:= \argmin_{(r,b)\in \mathcal{A}_M}  \mathcal{L}_\mu(r_1^n,b_1^n,r,b,\lambda_r^n,\lambda_b^n)\\\label{eq:Lupdr}
 \lambda_r^{n+1} &:= \lambda_r^{n} + \mu(r_1^{n+1}-r_2^{n+1}),\\ \label{eq:Lupdb}
 \lambda_b^{n+1} &:= \lambda_b^{n} + \mu(b_1^{n+1}-b_2^{n+1})
\end{align}
In order to actually solve the two minimization problems \eqref{eq:O1} and \eqref{eq:O2}, we employ a projected gradient method. In fact, for \eqref{eq:O1}, the first order optimality condition (of the reduced problem) is given as, \cite[Thm 1.48]{Hinze2009}
\begin{align}
(\ur,\, \ub)\in \mathcal{A}_B^\delta,\;\text{ and } \langle \nabla \mathcal{L}_\mu(\ur,\ub,r_2^n,b_2^n,\lambda_r^n,\lambda_b^n), \binom{\ur}{\ub} - \binom{r}{b} \rangle\qquad \forall (r,b) \in \mathcal{A}_B^\delta,
\end{align}
with
\begin{align}
 \nabla \mathcal{L}_\mu(r,b,r_2^n,b_2^n,\lambda_r^n,\lambda_b^n) = \left(\begin{array}{l}
                                                                \eps[\log(r) - \log(1-r-b)] + V - \lambda_r^n - \mu(r-r_2^n)\\
                                                                \eps[\log(b) - \log(1-r-b)] + V - \lambda_b^n - \mu(b-b_2^n)
                                                               \end{array}\right),
\end{align}
which can be shown to be equivalent \cite[Lemma 1.12]{Hinze2009} to
\begin{align*}
 P_{\mathcal{A}_B^\delta}(\nabla \mathcal{L}_\mu(r,b,r_2^n,b_2^n,\lambda_r^n,\lambda_b^n)) = 0.
\end{align*}
For \eqref{eq:O2}, the corresponding condition gradient is given by
\begin{align}
 \nabla \mathcal{L}_\mu(r_1^n,b_1^n,r,b,\lambda_r^n,\lambda_b^n) = \left(\begin{array}{l}
                                                                c_{11}V_{r} - V_{b} + \lambda_r^n + \mu(r-r_1^n)\\
                                                                c_{22}V_{b} - V_{r} + \lambda_b^n + \mu(b-b_1^n)
                                                               \end{array}\right).
\end{align}
with
 $V_r = K \ast r$ and $V_b = K \ast b$.
%
To actually solve these problems using a projected steepest descent, we employ a P1 finite element scheme. We only describe the details for the case $N=2$ (with obvious modifications for $N=1$). To this end, we approximate $\Omega$ by a polygonal domain $\Omega_h,$ for which we introduce a conforming triangulation $T_h=\{T\}$. As usual, we call $h=\max_T h_T$ the mesh size. On $\Omega_h$, we introduce the space of continuous piecewise linear functions over $\Omega_h$, 
$$
V_h = \{ v_h \in C(\Omega_h): v_h|_T \in P_1(T) \quad \text{for all } T \in T_h\}.
$$ 
Note that by construction $V_h \subset H^1(\Omega_h)$. Since we restricted ourselves to $K$ being the Coulomb-Kernel, calculating the convolution $u=K\ast r$ is equivalent to solving the problem
\begin{align}
 -\Delta u &= r,\;\text{ in } \Omega\\
 u &= 0,\;\text{ on }\partial\Omega.
\end{align}
For all computations, we use a P1 finite element approximation. The use of first order finite elements has the important merit that the coefficient vector containts the node values and therefore, we can directly apply the projection \eqref{eq:projab} on the box constraints $\A_B$ (for \eqref{eq:projam} we have to employ the mass matrix in order to evaluate the integral). For extensions to higher order finite elements see \cite{Wurst2015}. All examples are implemented in MATLAB.\\
All one-dimensional calculation are done on the domain $\Omega_1 = [-1,1]$ while in two dimensions we used the circle with radius $r=2$, i.e. $\Omega_2 = B_2(0)$. 
%
\subsection{Examples in Spatial Dimension One}
We start by numerically verifying the results of Cicalese et. al. \cite{Cicalese2015}. The potential is given by
\begin{align}\label{eq:Vnum}
 V(x) = \left\{\begin{array}{cl}
	      (x-0.5)^2 &  x > 0.5,\\
	      (x+0.5)^2 &  x < 0.5,\\
	      0& \text{otherwise}.
	      \end{array}\right.
\end{align}
We initialize both $r$ and $b$ with random values in the interval $[0,0.49]$ which ensures $\rho \le 0.98 < 1$. We use the ADMM scheme described above with step size $\tau = 0.01$ and take $1000$ points to discretize space for the projected gradient iterations. We are able to numerically verify the different cases as depicted in \cite[Figure 2]{Cicalese2015} for the Coulomb kernel. The results are shown in Figure \ref{fig:epszero1d}. 
\begin{figure}\label{fig:nodiffadmm}
\begin{center}
\begin{subfigure}[t]{0.33\textwidth}
      \centering
      \includegraphics[width=\textwidth]{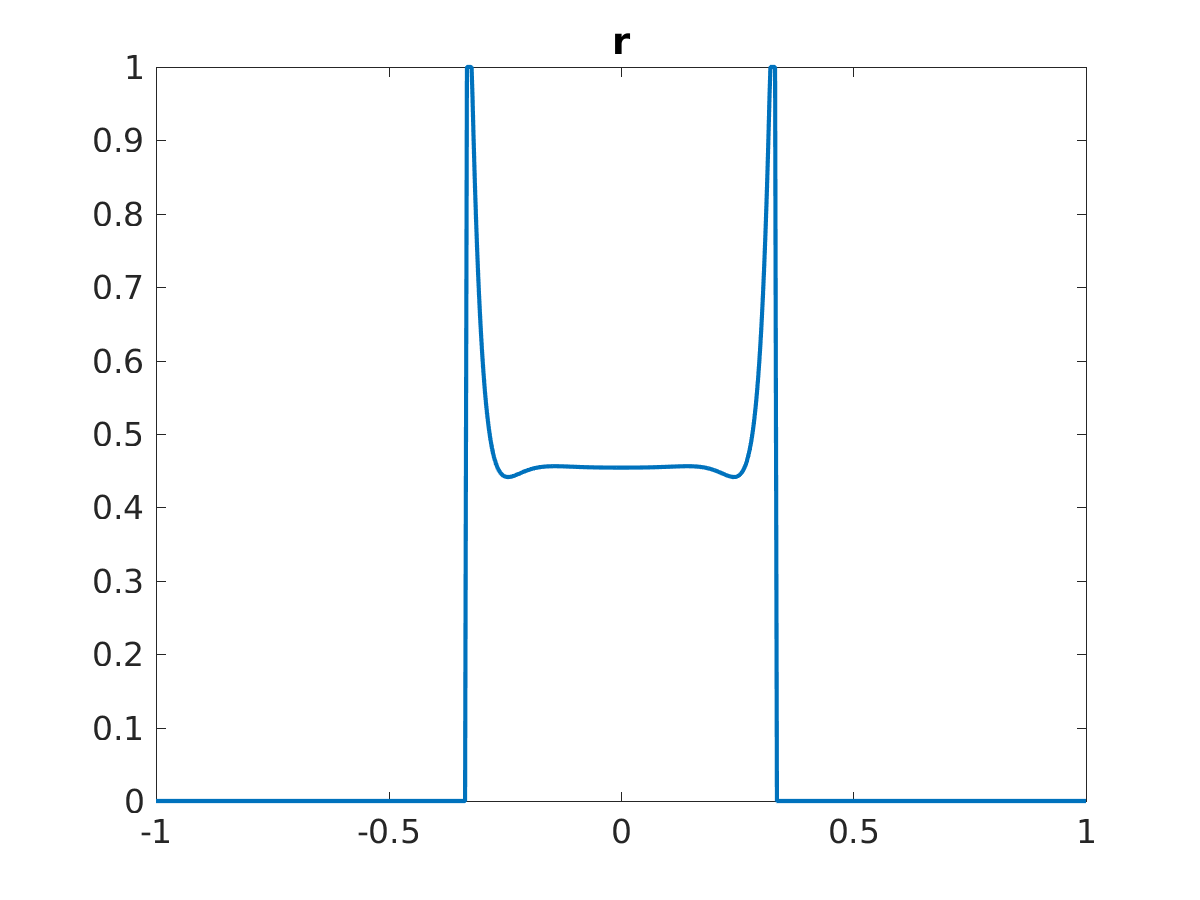}\\
      \includegraphics[width=\textwidth]{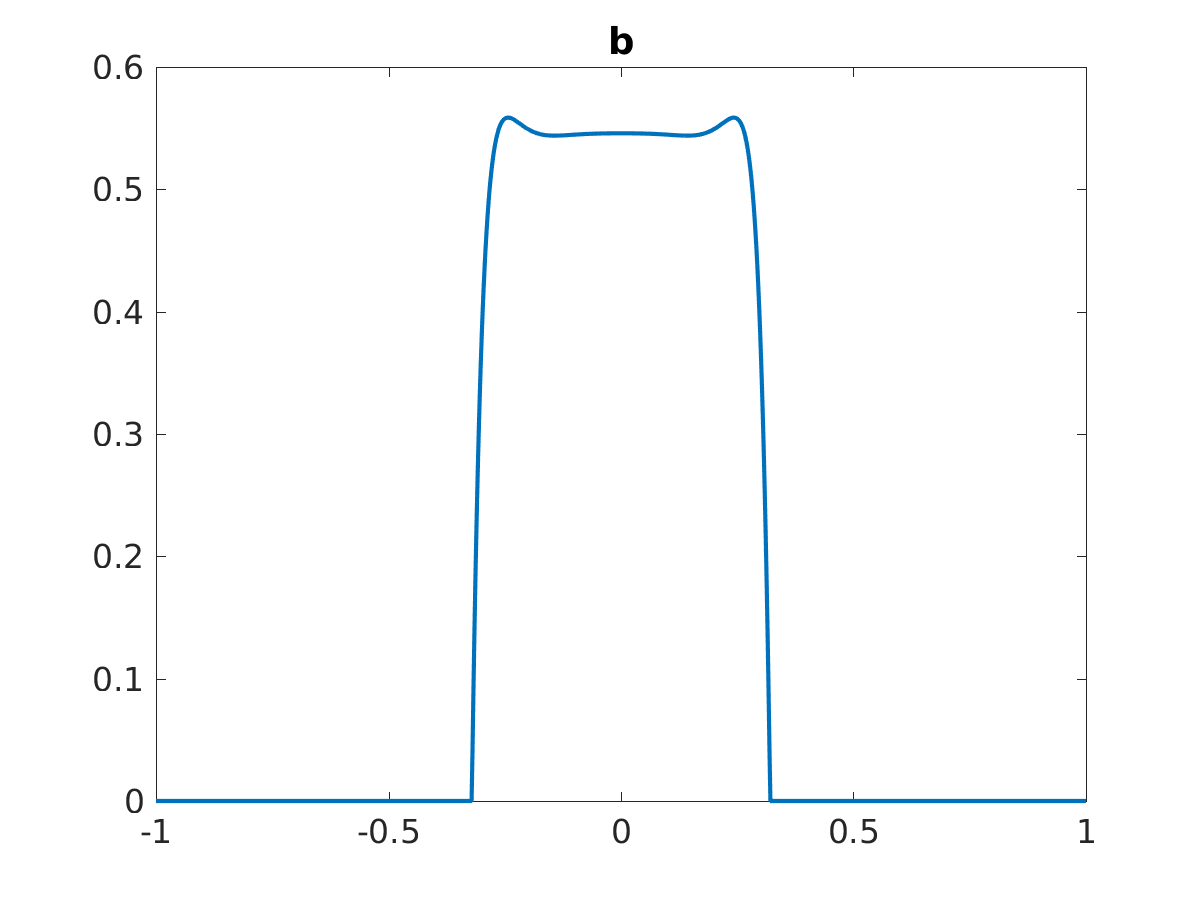}
      \caption{$(c_{11},c_{22}) = (-0.4,-0.5)$}
  \end{subfigure}%
  \begin{subfigure}[t]{0.33\textwidth}
      \centering
      \includegraphics[width=\textwidth]{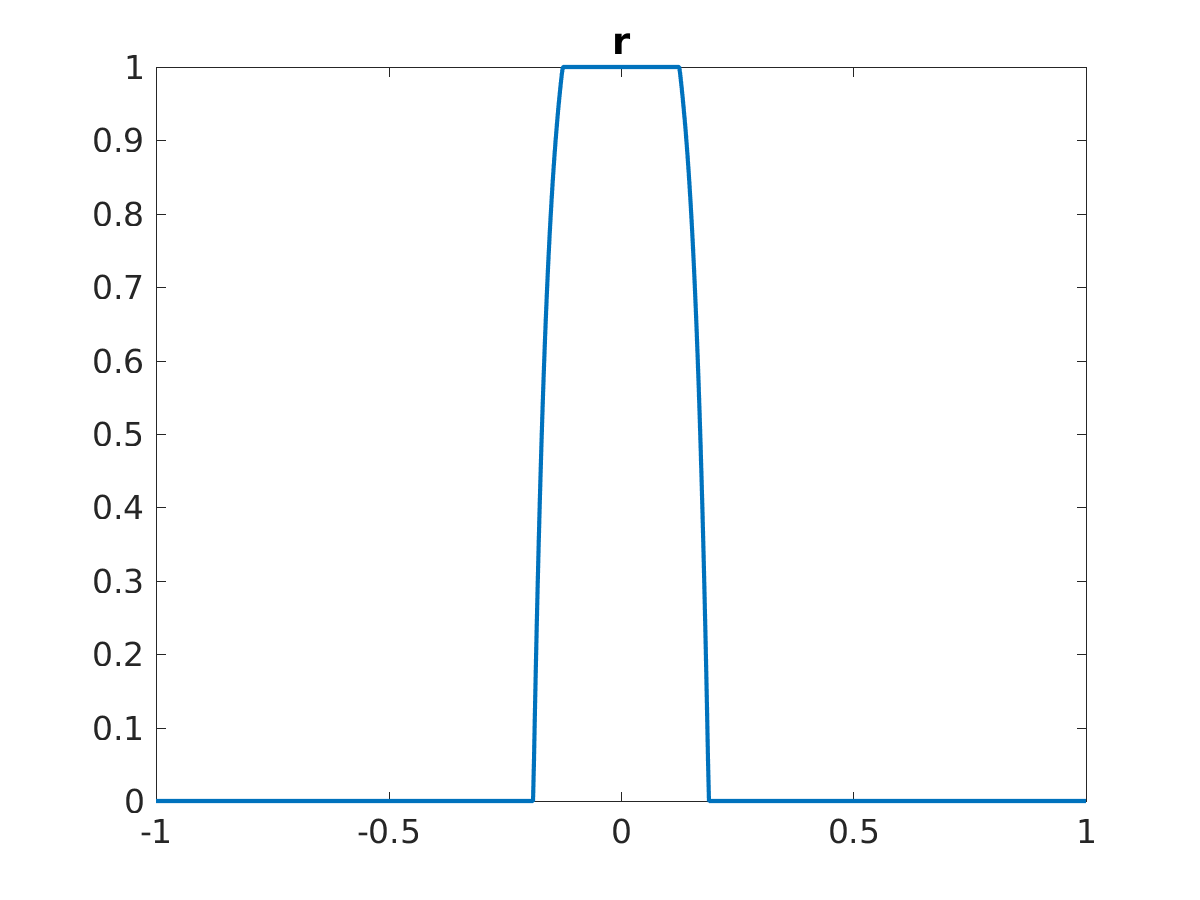}\\
      \includegraphics[width=\textwidth]{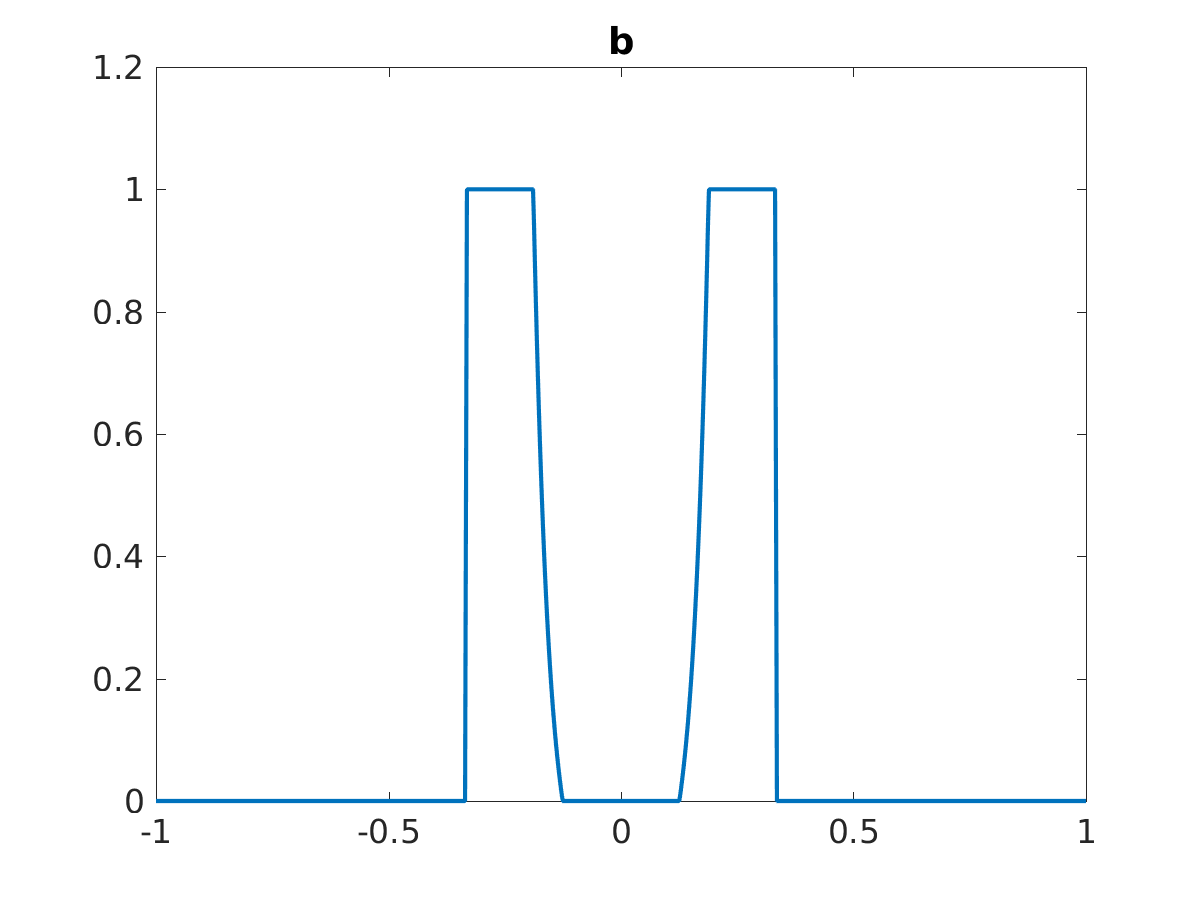}
      \caption{$(c_{11},c_{22}) = (-1.0,-0.5)$}
  \end{subfigure}%
  \begin{subfigure}[t]{0.32\textwidth}
      \centering
      \includegraphics[width=\textwidth]{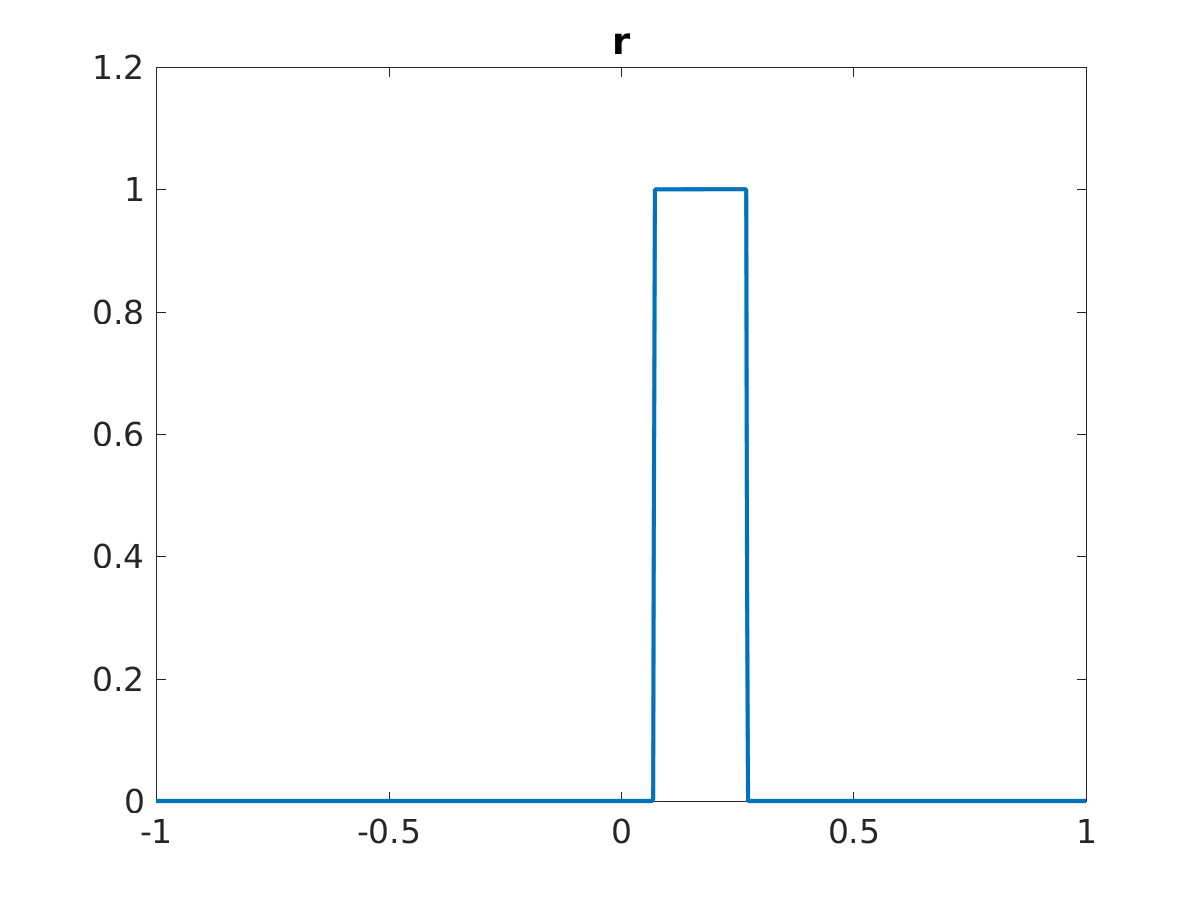}\\
      \includegraphics[width=\textwidth]{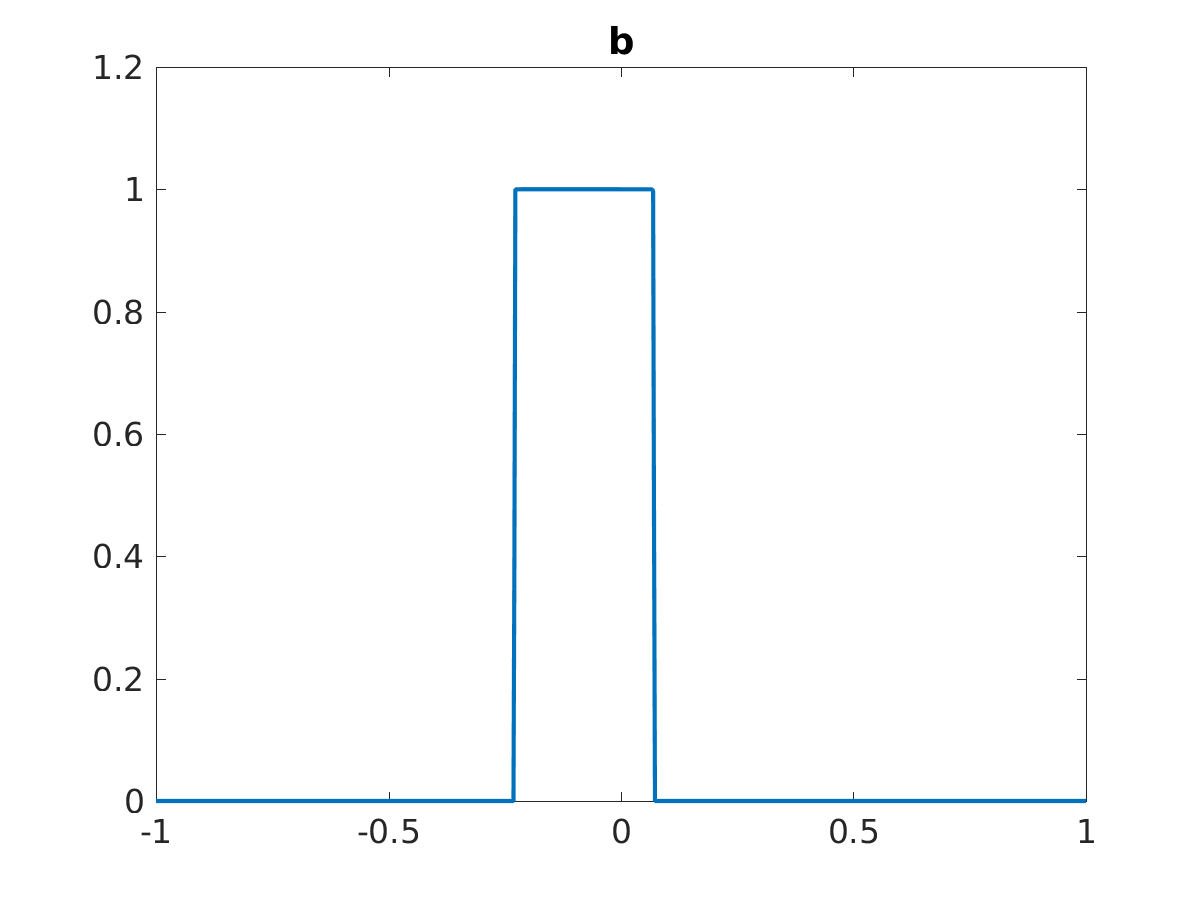}
      \caption{$(c_{11},c_{22}) = (-1.5,-2)$, \\ $m_r=0.2$, $m_b=0.3$}
  \end{subfigure}%
 \end{center}
 \caption{Densities $r$ (top row) and $b$ (bottom row) for the parameters as indicated. In the first two plots, both species have mass $m_r=m_b=1/3$ while on the right, we took $m_r=0.2$ and $m_b=0.3$.}
 \label{fig:epszero1d}
\end{figure} 
While all these examples are symmetric with respect to the origin, we can also obtain the asymmetric cases which appear for either $c_{11}=-1$ and $c_{22} < -1$ and $c_{11} < -1$ or $c_{22} = -1$, see Figure \ref{fig:admmasym}. To break the symmetry, we multiplied the random initial data of $r$ by the positive function $0.3x+1$. For the case $(-1,-2)$, left picture of Figure \ref{fig:admmasym}, we obtain one density surrounded by the other and taking the same initial conditions but values $(-2,1)$, we see that the support of the minimizers separate. This, for $N=2$, corresponds to the situation where the smaller ball touches the boundary of the larger one, see again \cite[Figure 2]{Cicalese2015}. Finally, again for $(-2,1)$ but the random initial data for $b$ multiplied by $x+1$, we obtain a picture anti-symmetric to the first one, shown on the right. Since all initial guesses have the same mass, this also illustrates the non-uniqueness of minimizers.
\begin{figure}
\begin{center}
\begin{subfigure}[t]{0.33\textwidth}
      \centering
      \includegraphics[width=\textwidth]{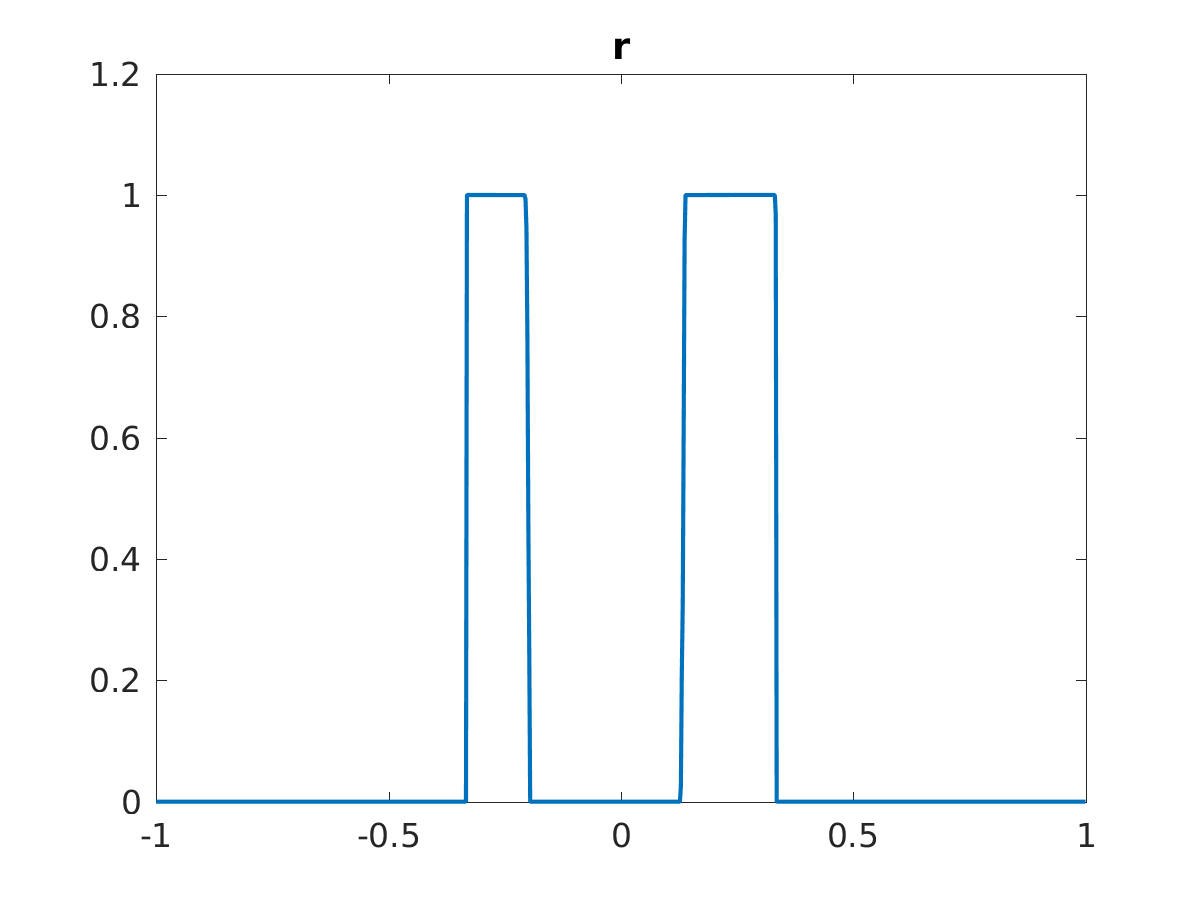}\\
      \includegraphics[width=\textwidth]{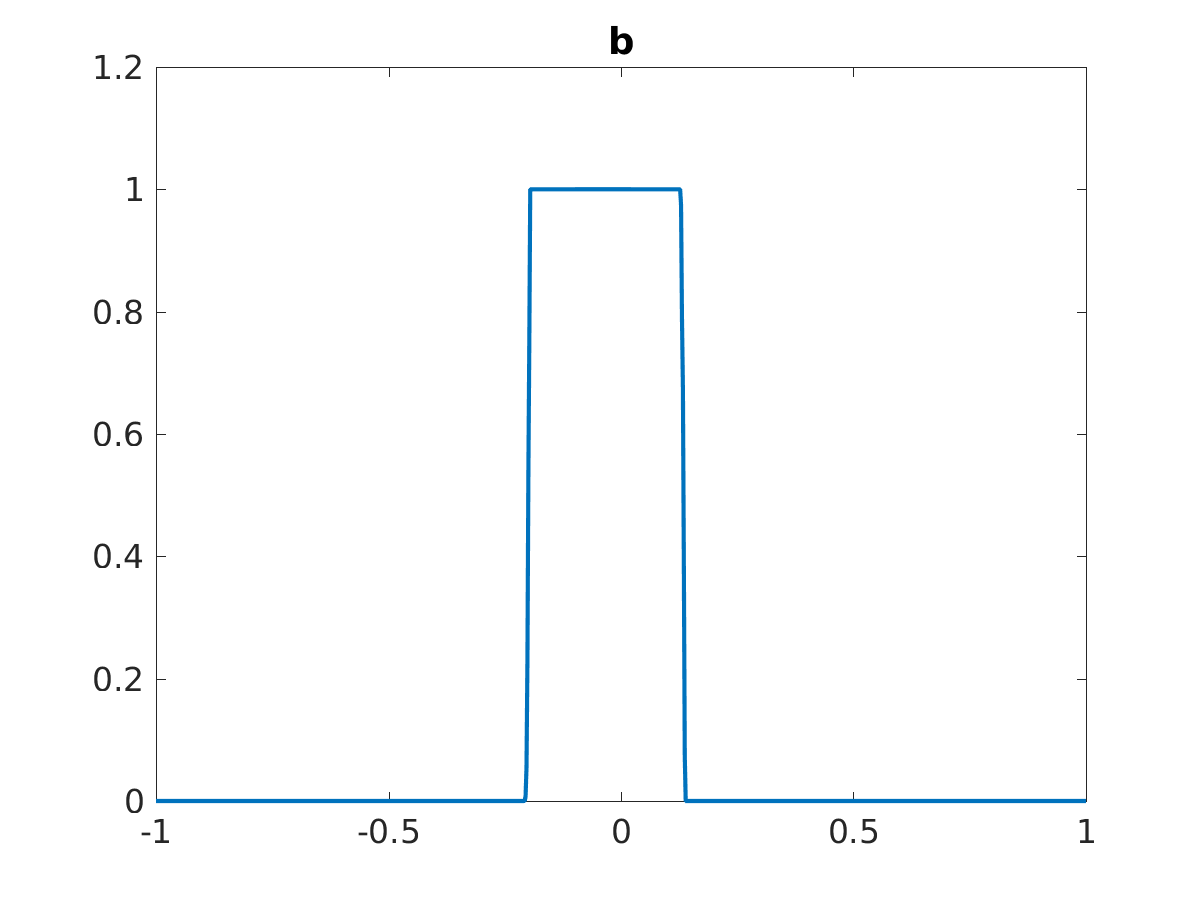}
      \caption{$(c_{11},c_{22}) = (-1,-2)$}
  \end{subfigure}%
  \begin{subfigure}[t]{0.33\textwidth}
      \centering
      \includegraphics[width=\textwidth]{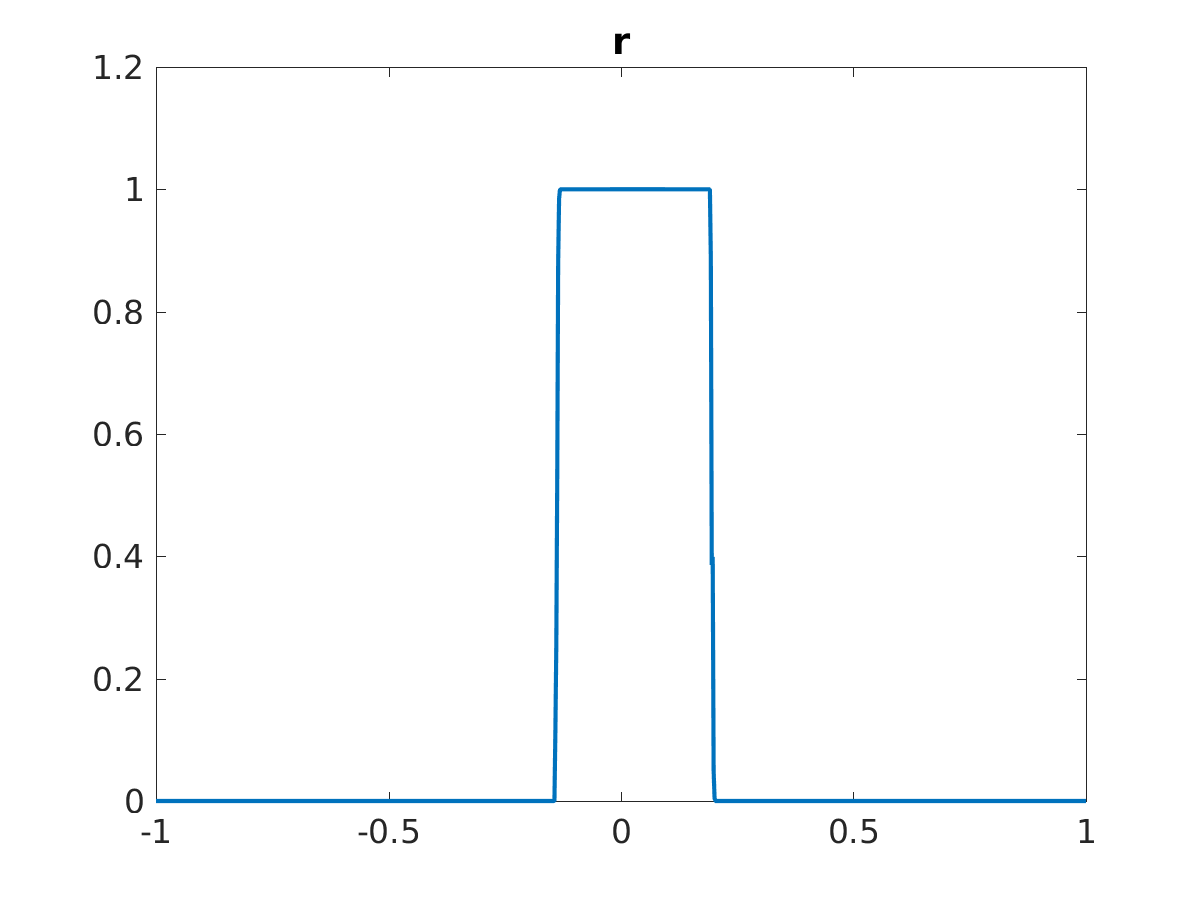}\\
      \includegraphics[width=\textwidth]{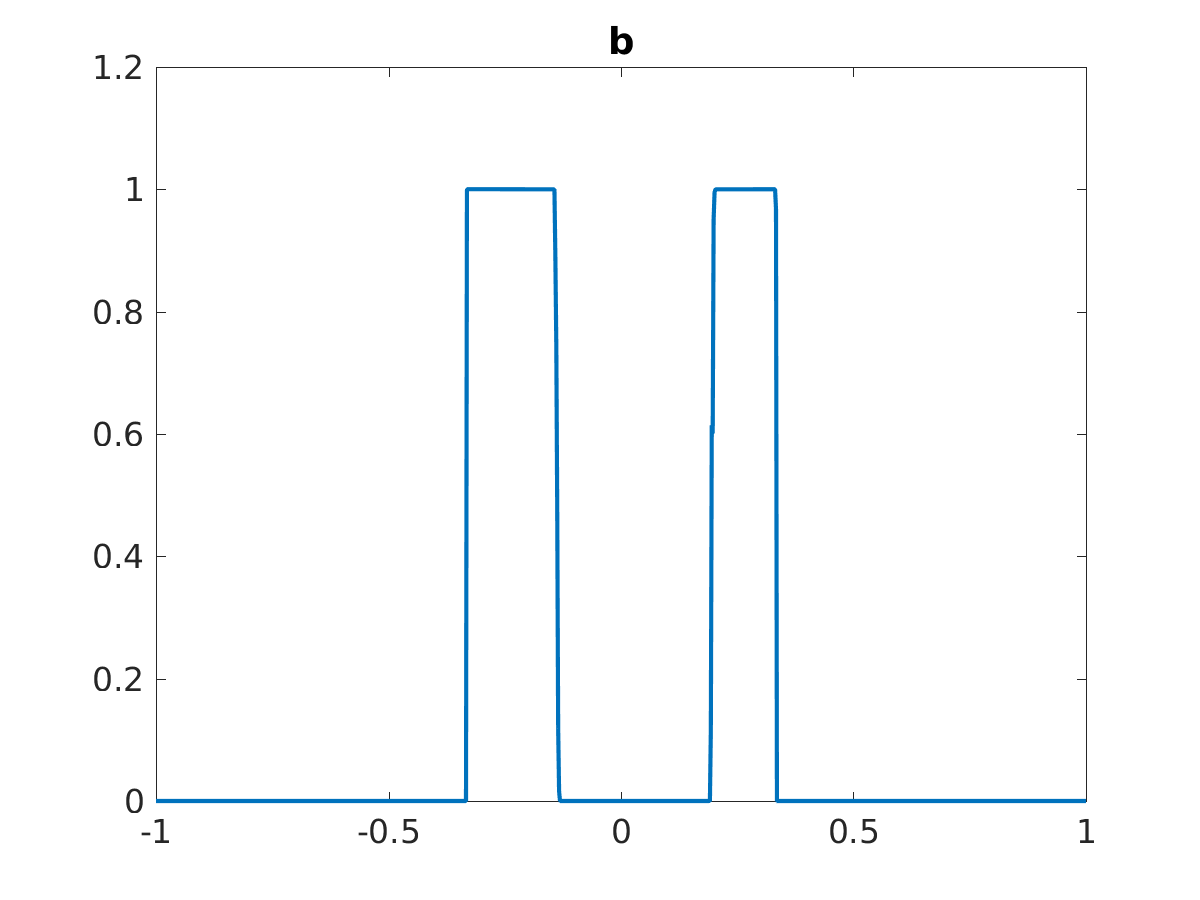}
      \caption{$(c_{11},c_{22}) = (-2,-1)$}
  \end{subfigure}%
  \begin{subfigure}[t]{0.32\textwidth}
      \centering
      \includegraphics[width=\textwidth]{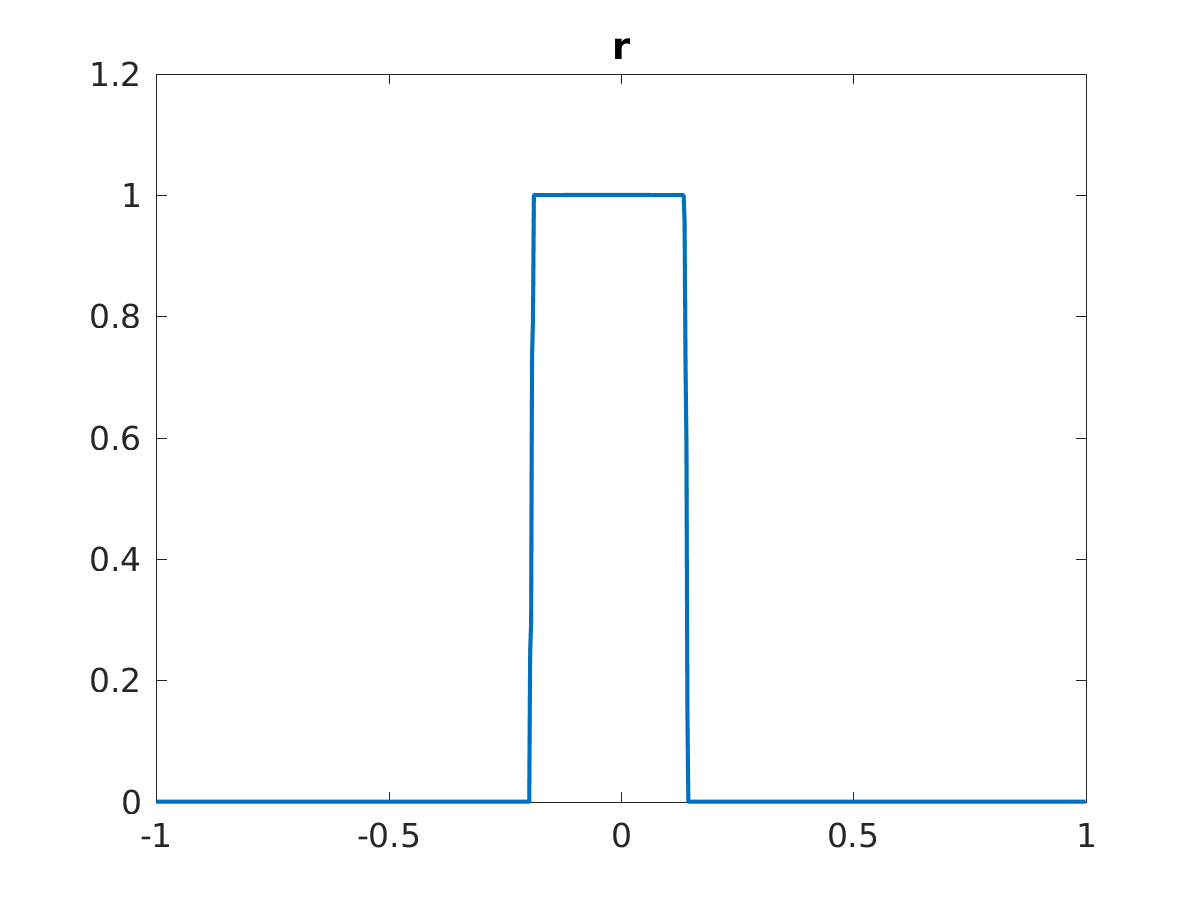}\\
      \includegraphics[width=\textwidth]{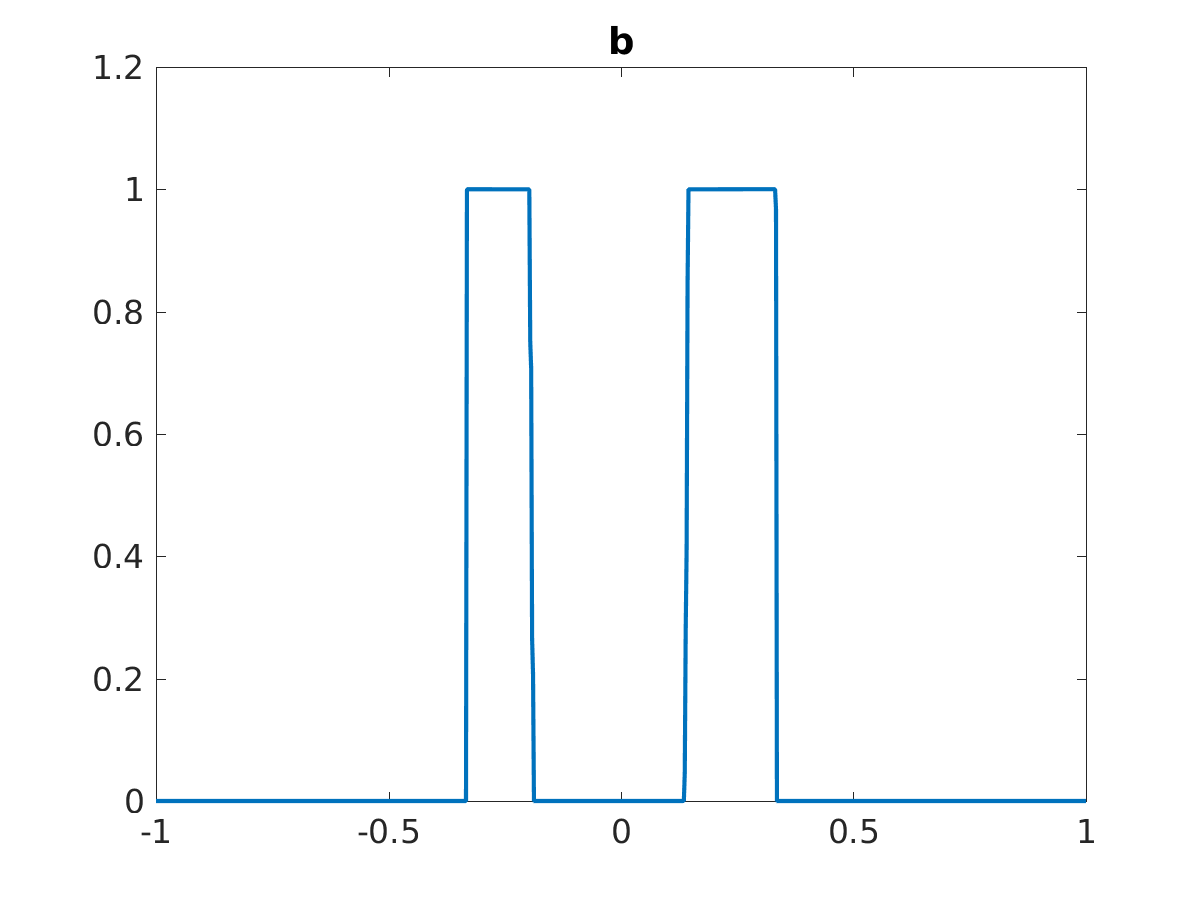}
      \caption{$(c_{11},c_{22}) = (-2,-1)$}
  \end{subfigure}%
  \end{center}
\caption{(a) and (b): Minimizers for random initial data, yet multiplied with $0.3x+1$ for species $r$ and values $(-1,-2)$ and $(-2,-1)$ for $c_{11}$ and $c_{22}$. (c): Again $c_{11}=-2$ and $c_{22}=-1$, but this time the initial datum of $b$ multiplied by $x+1$.}
\label{fig:admmasym}
\end{figure} 
Next, we solve the minimization problem \eqref{eq:minprob} for values $\eps = 0.01,\,0.05,\,0.1$. In Figure \ref{fig:epspos1d}, we compare the corresponding minimizers to the case $\eps = 0$ with the choice $c_{11} = -1$ and $c_{22} = -1.5$. As expected, for small $\eps$, the entropic part of the functional causes a smoothing of the minimizers. For larger $\eps$, the energy becomes convex and we observe mixing without phase seperation as discussed in the introduction. This is also illustrated in Figure \ref{fig:epspos1d}, where the integral of $rb$ over $\Omega$ is evaluated for different values of $\eps$. For small values, a linear increase is observed while for large values, the overlap becomes very large and the integral cannot serve as an indicator for the overlap anymore.
\begin{figure}
\begin{center}
\begin{subfigure}[t]{0.33\textwidth}
      \centering
       \includegraphics[width=\textwidth]{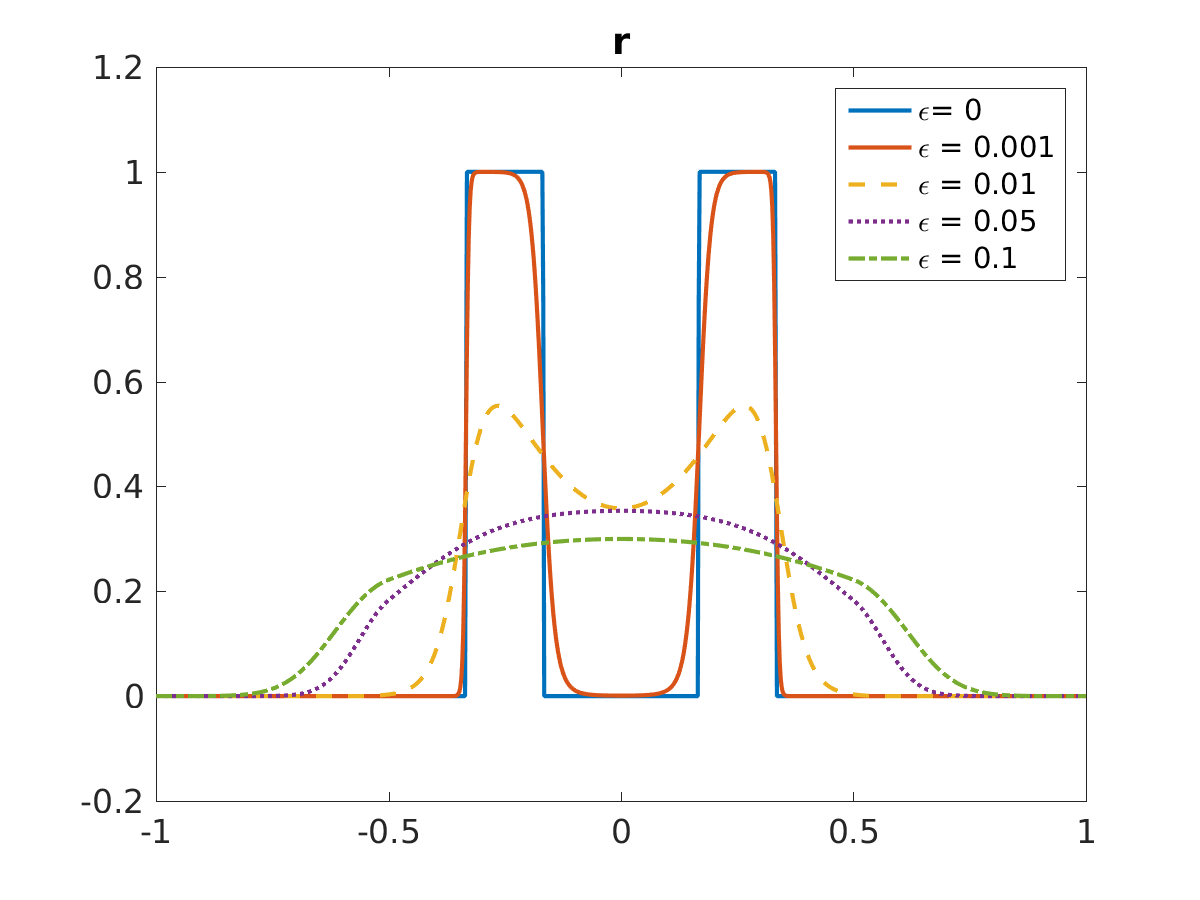}
      \caption{Density $r$}
  \end{subfigure}%
  \begin{subfigure}[t]{0.33\textwidth}
      \centering
      \includegraphics[width=\textwidth]{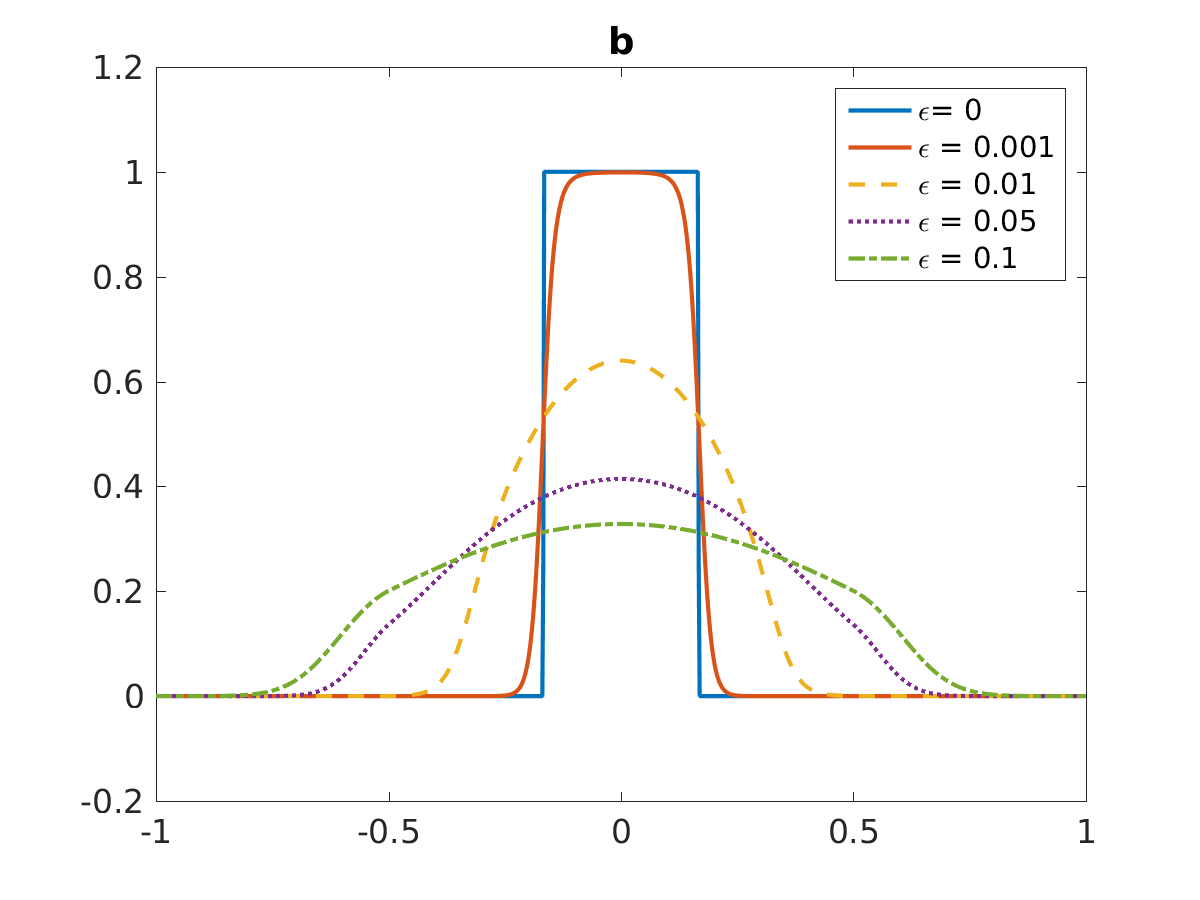}
      \caption{Density $b$}
  \end{subfigure}%
  \begin{subfigure}[t]{0.32\textwidth}
      \centering
      \includegraphics[width=\textwidth]{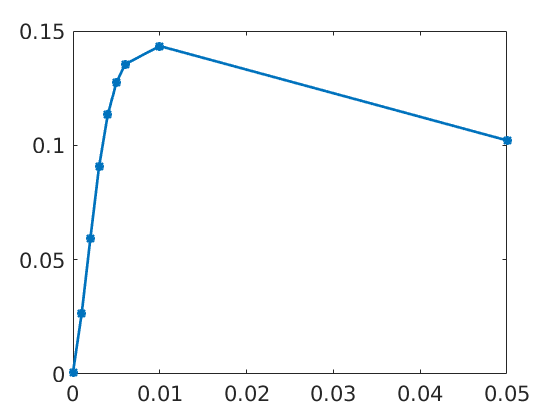}
      \caption{$\int rb\;dx$}
  \end{subfigure}%
 \end{center}
 \caption{Densities $r$ (a) and $b$ (b) computed for values of $\eps = 0,\,0.001,\,0.01,\,0.05,\,0.1$ and with $(c_{11}=-1,c_{22}=-1.5)$. Again, the initial values were taken random. (c): The integral $\int rb\;dx$ with $(c_{11}=-1,c_{22}=-1.5)$ and for $\eps = 0,\, 0.001,\, 0.002,, 0.003,\, 0.004,\, 0.005,\, 0.006,\, 0.01,\, 0.05,\, 0.1$.}
\label{fig:epspos1d}
\end{figure}

\subsection{Examples in Spatial Dimension Two}
Next, we use our algorithm to solve the minimization problem in two spatial dimensions. Again, we first verify that we can reproduce the results of \cite{Cicalese2015}. This is shown in Figure \ref{fig:epszero2d}. We do not show the simulations for different $\eps$ since the results are very similar to the ones in one spatial dimension. The only difference is that we are not able to produce an asymmetric situation. For example, in Figure \ref{fig:epszero2d}, right picture, we have the case  $c_{11}=-1,\,c_{22}=-3$ with $m_r=0.2$, $m_b=0.4$ and asymmetric initial data, yet we obtain a completely symmetric result. The same holds true for simulations with both $c_{11}$ and $c_{22}$ smaller than one.
\begin{figure}
\begin{center}
\begin{subfigure}[t]{0.33\textwidth}
      \centering
      \includegraphics[width=\textwidth]{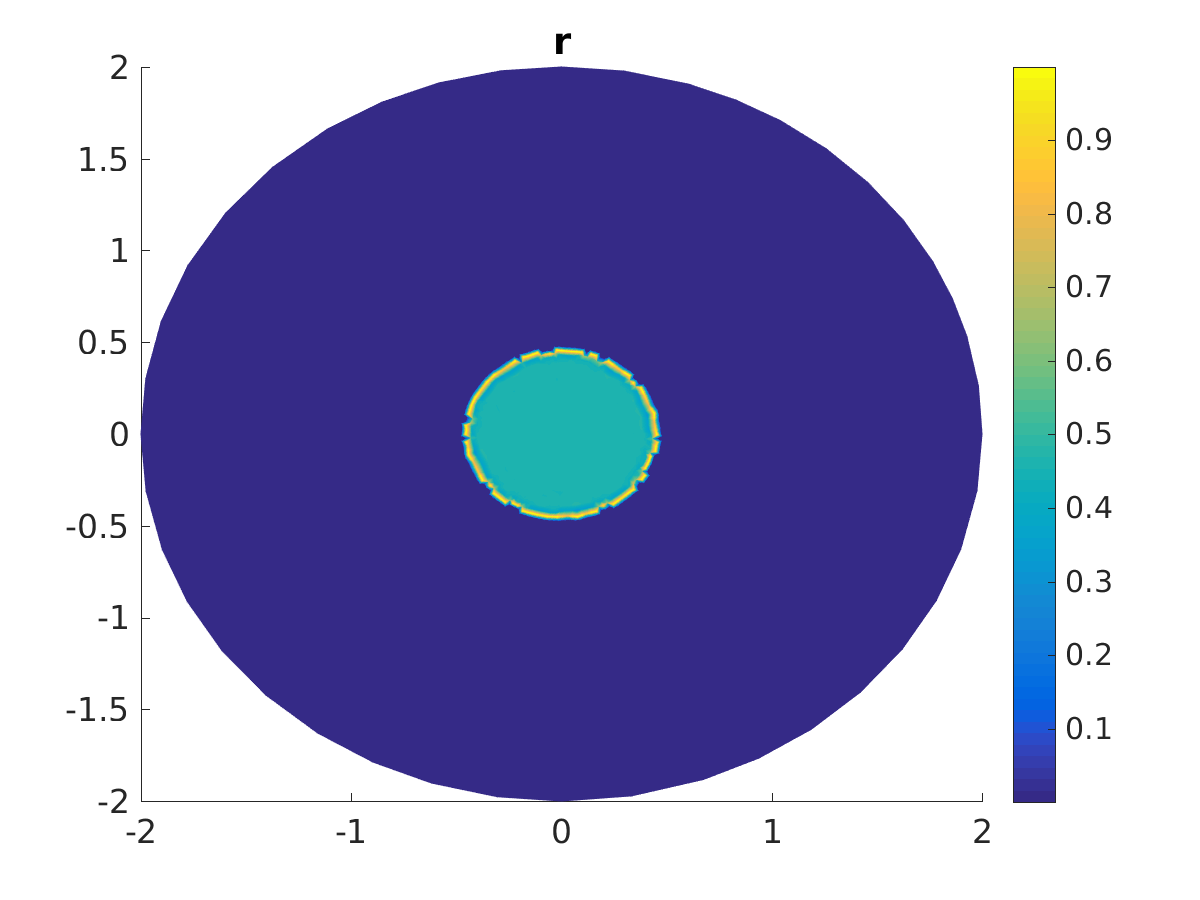}\\
      \includegraphics[width=\textwidth]{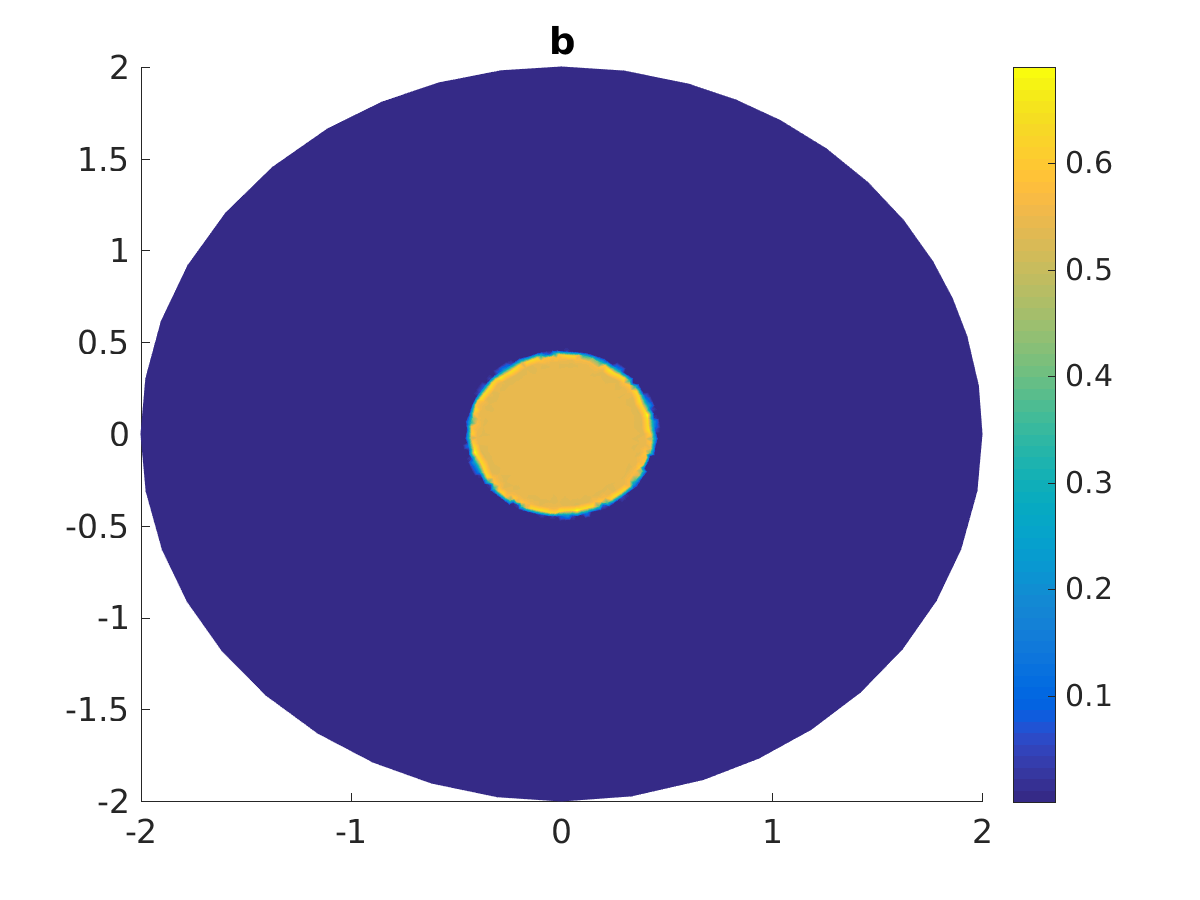}
      \caption{$(c_{11},c_{22}) = (-0.4,-0.5)$}
  \end{subfigure}%
  \begin{subfigure}[t]{0.33\textwidth}
      \centering
      \includegraphics[width=\textwidth]{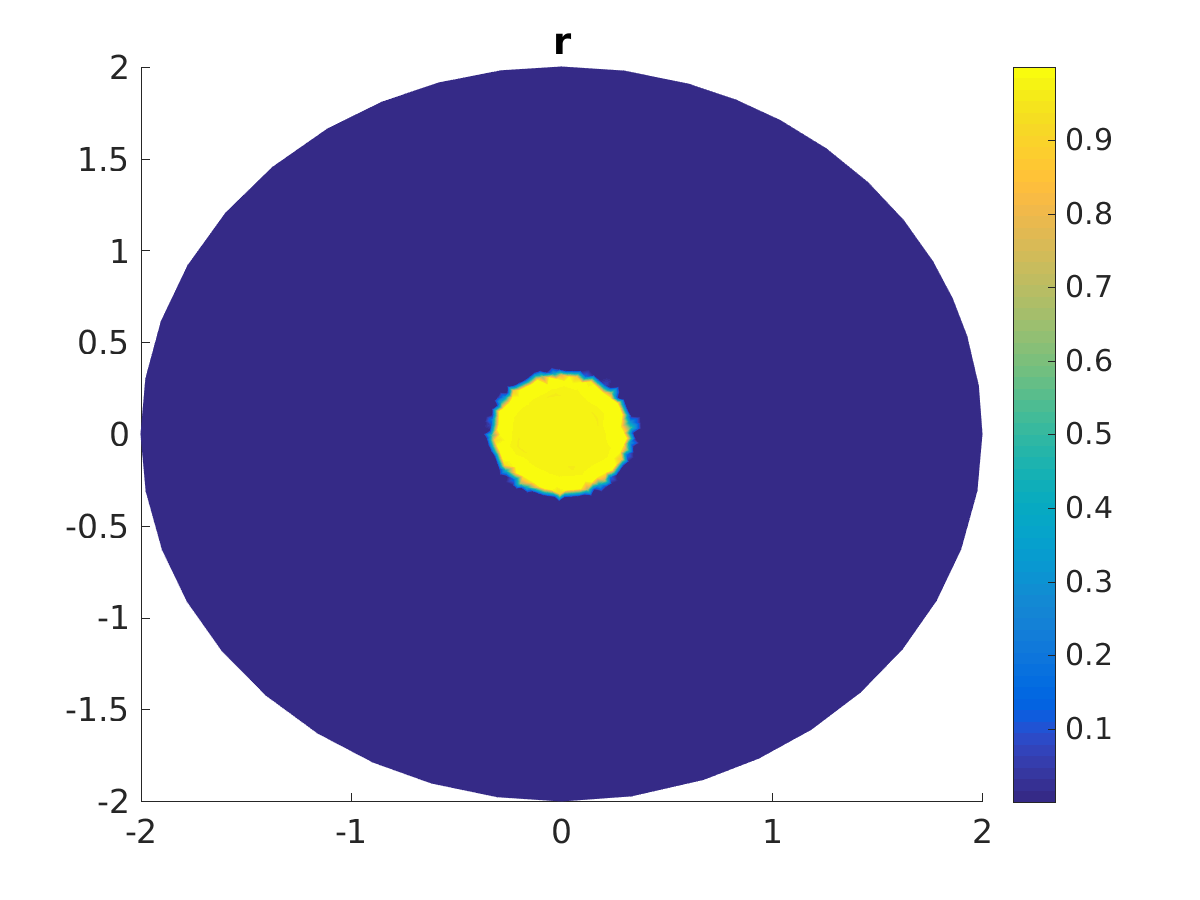}\\
      \includegraphics[width=\textwidth]{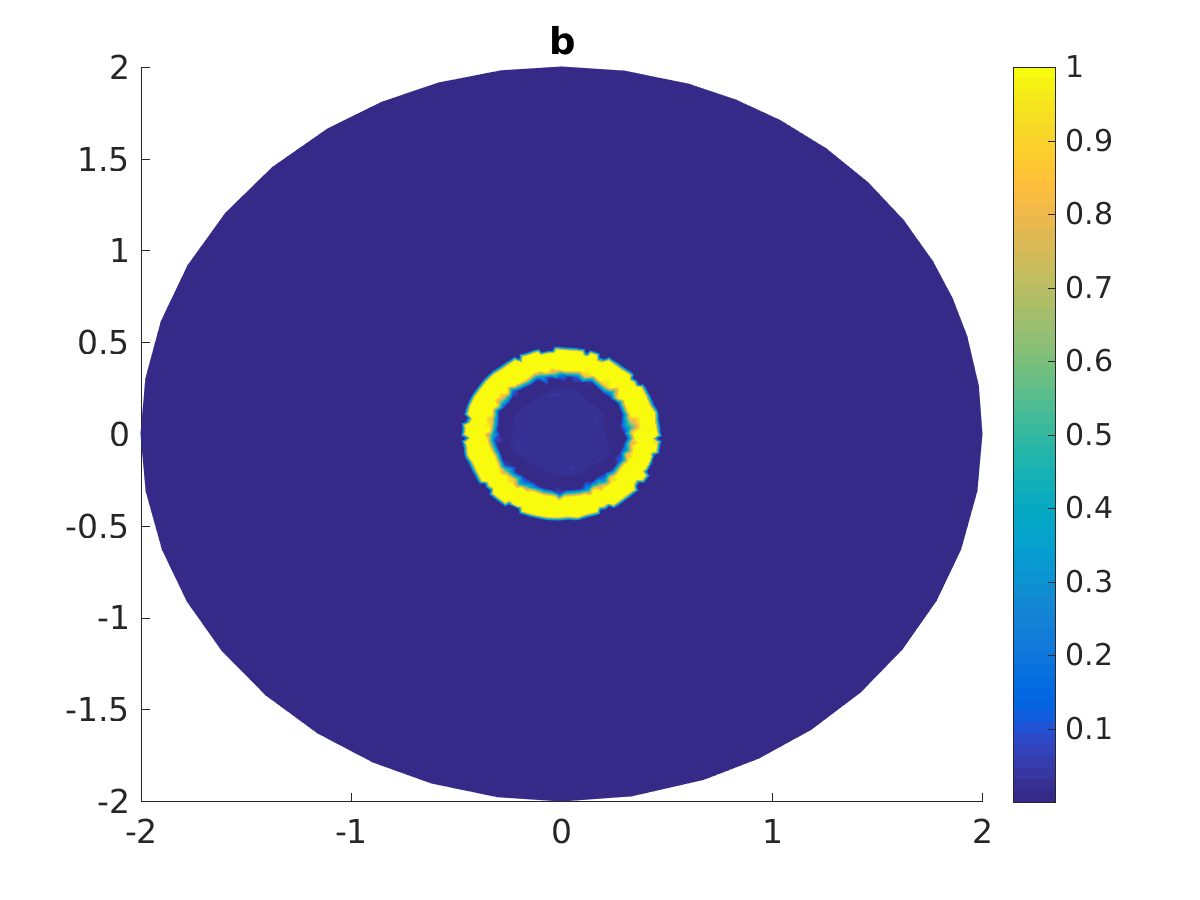}
      \caption{$(c_{11},c_{22}) = (-1,-0.5)$}
  \end{subfigure}%
  \begin{subfigure}[t]{0.32\textwidth}
      \centering
      \includegraphics[width=\textwidth]{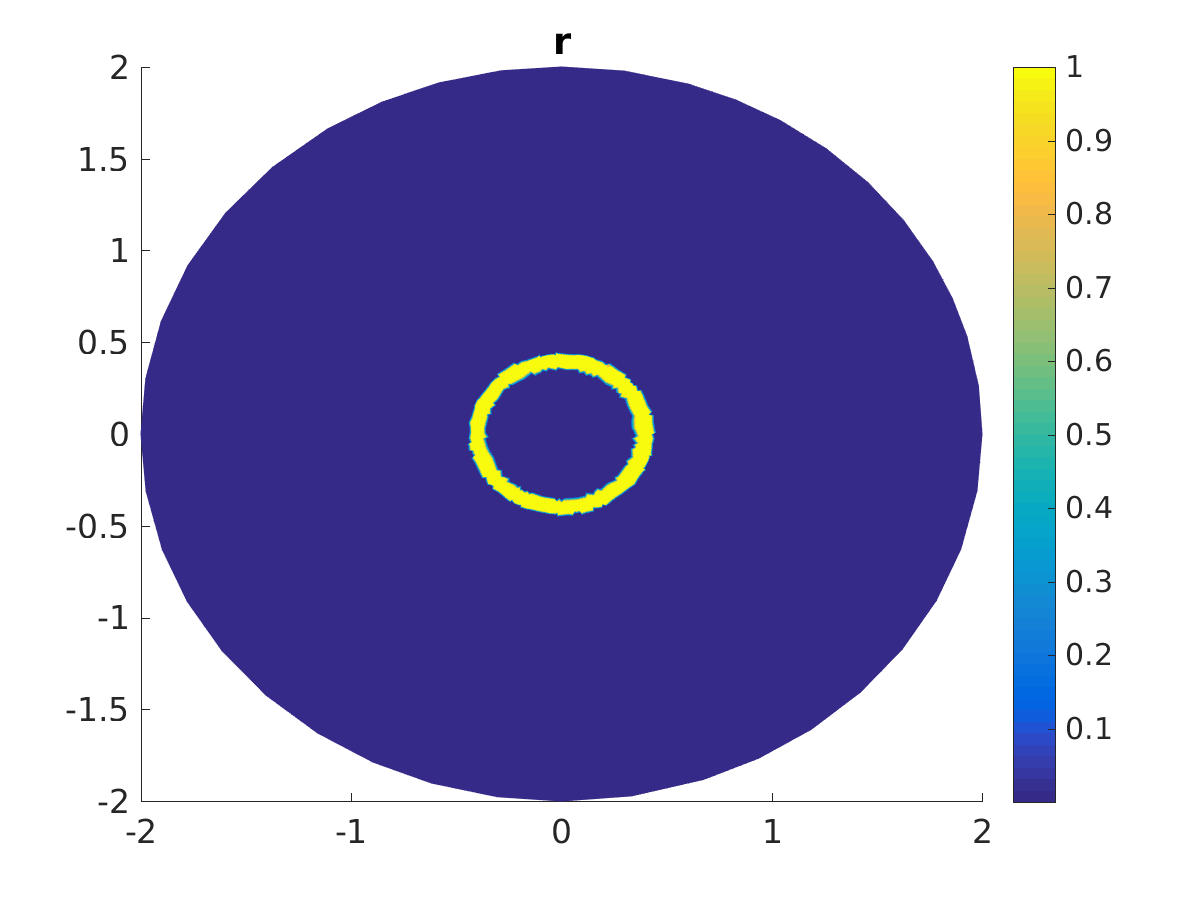}\\
      \includegraphics[width=\textwidth]{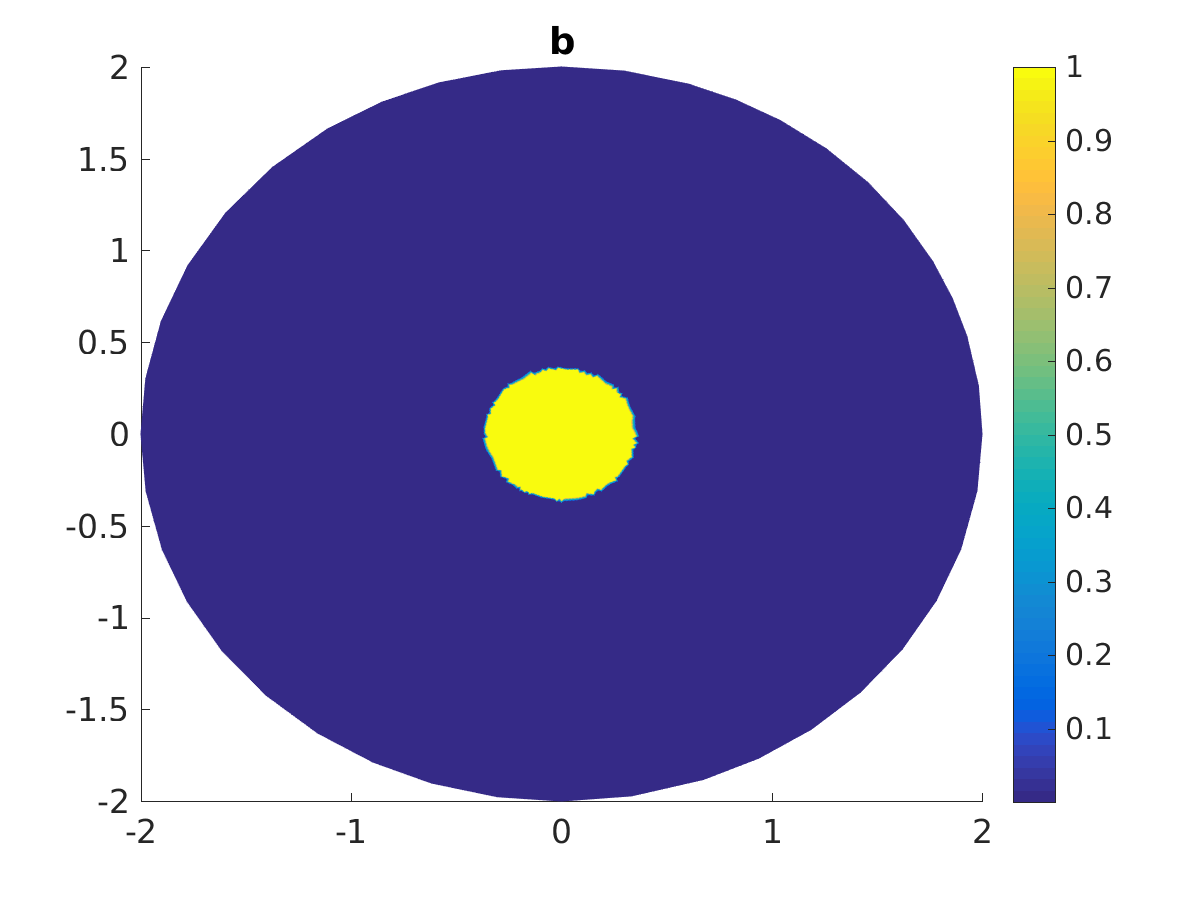}
      \caption{$(c_{11},c_{22}) = (-1,-3)$\\$m_r=0.2$, $m_b=0.4$}
  \end{subfigure}%
 \end{center}
  \caption{(a): $c_{11}=-0.4,\,c_{22}=-0.5$ and (b): $c_{11}=-1,\,c_{22}=-0.5$, again with $m_r=m_b=1/3$. (c): $c_{11}=-1,\,c_{22}=-3$ with $m_r=0.2$, $m_b=0.4$ and asymmetric initidal data.}
  \label{fig:epszero2d}
\end{figure}
More interesting is again the case
\begin{align*}
 c_{11} < -1,\; c_{22} < -1.
\end{align*}
In \cite{Cicalese2015}, an explicit characterization of minimizers in this case is only possible in one spatial dimension. It is, however, known that for $N>1$, minimizers are characteristic functions of sets. By performing simulations, we confirm this results and are also able to characterize minimizers more precisely. In fact, for $c_{11}=c_{22}=-2$ and $m_r=m_b=1/3$, we see that the minimizers are two half balls, Figure \ref{fig:admm2d} (left). Changing the mass of one species, we see that the interface between then changes from a straight line as in the previous case to a curved structure, Figure \ref{fig:admm2d} (second from left). Varying also the constants $c_{11}$ and $c_{22}$, as well as the masses, leads to similar results, see Figure \ref{fig:admm2d} (third from left and right picture).
\begin{figure}
\begin{center}
\begin{subfigure}[t]{0.26\textwidth}
      \centering
      \includegraphics[width=\textwidth]{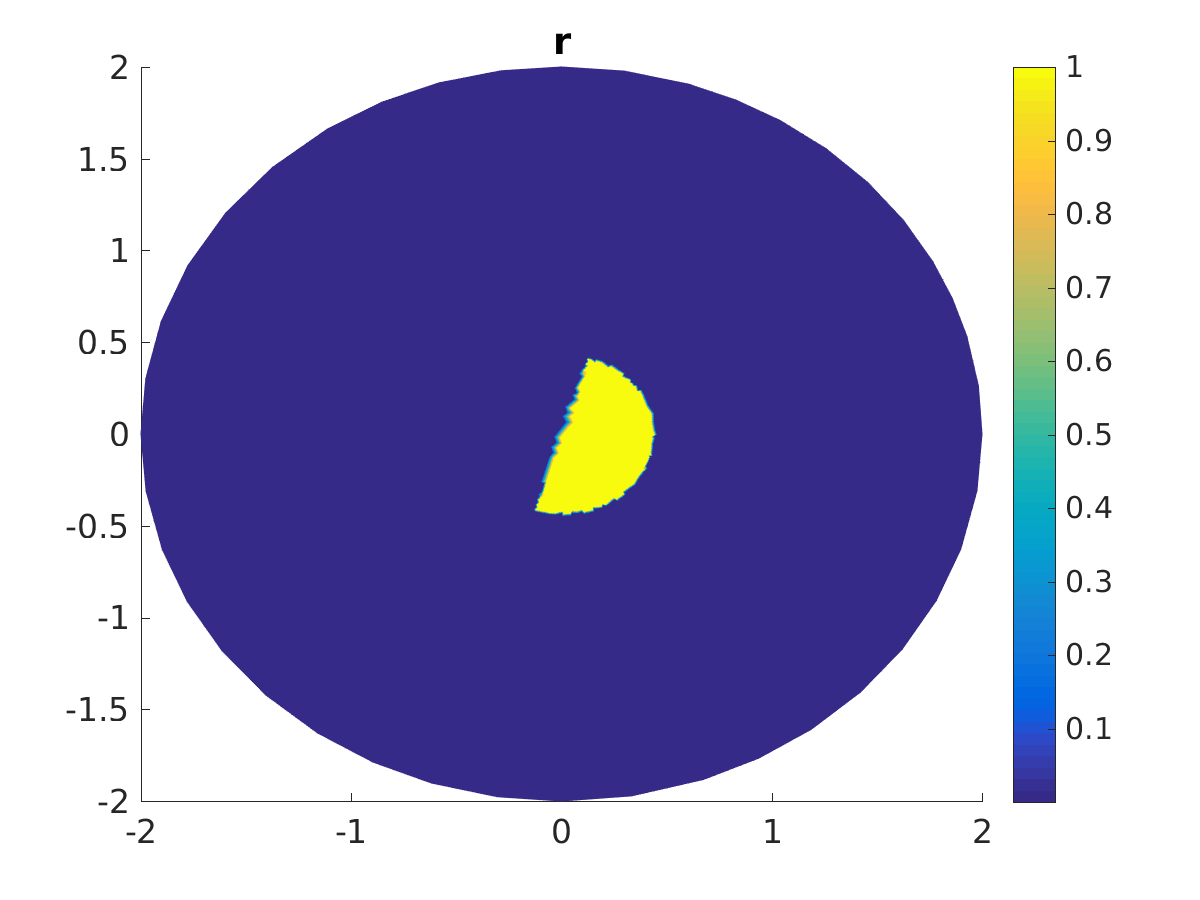}\\
      \includegraphics[width=\textwidth]{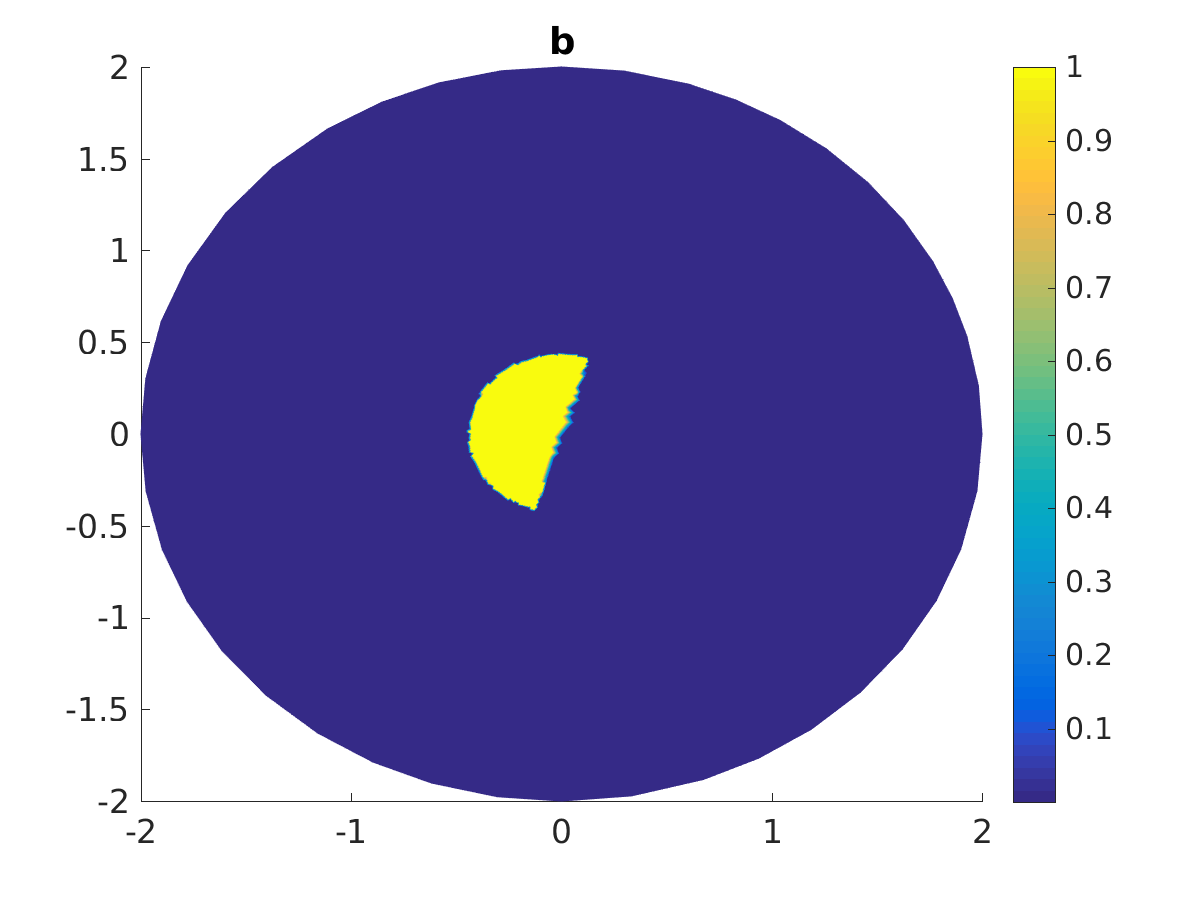}
      \caption{$(c_{11},c_{22}) = (-2,-2)$\\$m_r=0.3$, $m_b=0.3$}
  \end{subfigure}%
  \begin{subfigure}[t]{0.26\textwidth}
      \centering
      \includegraphics[width=\textwidth]{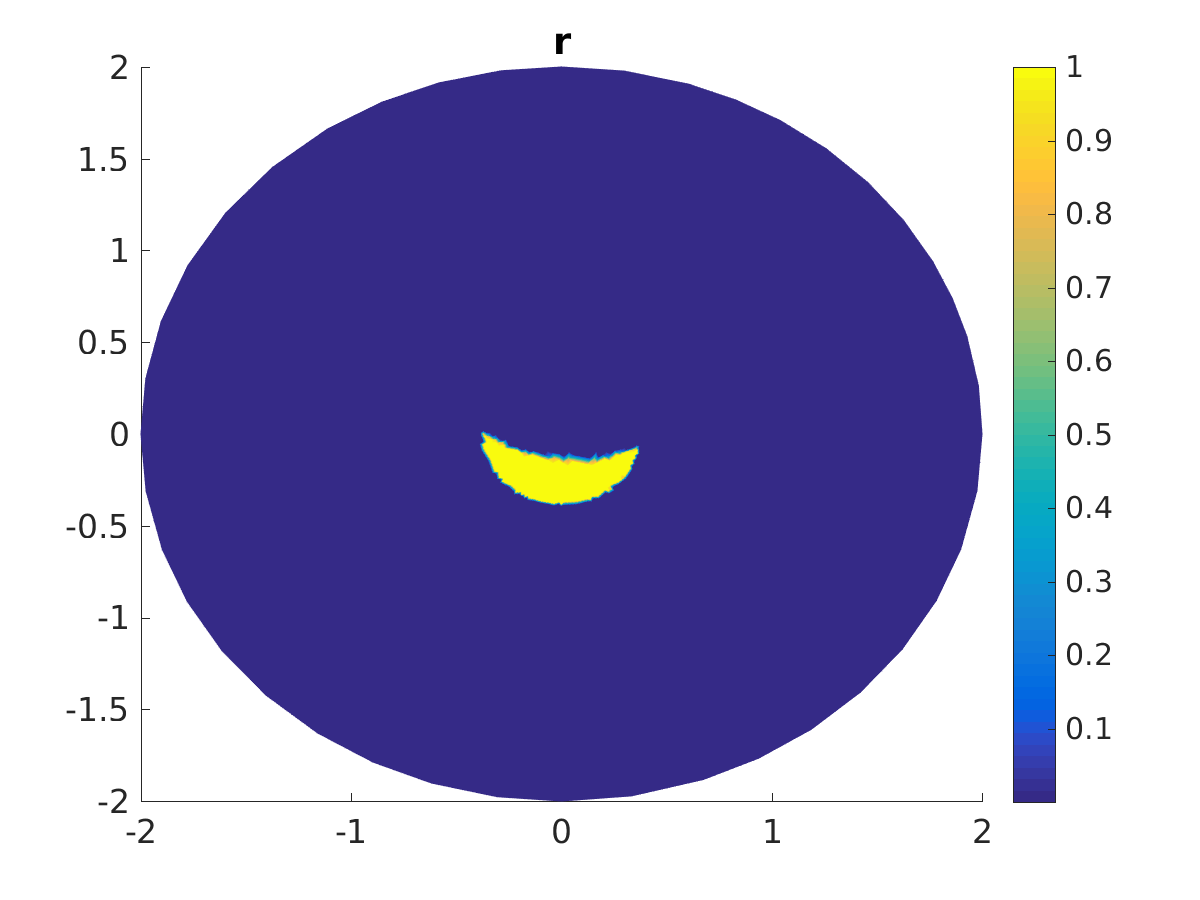}\\
      \includegraphics[width=\textwidth]{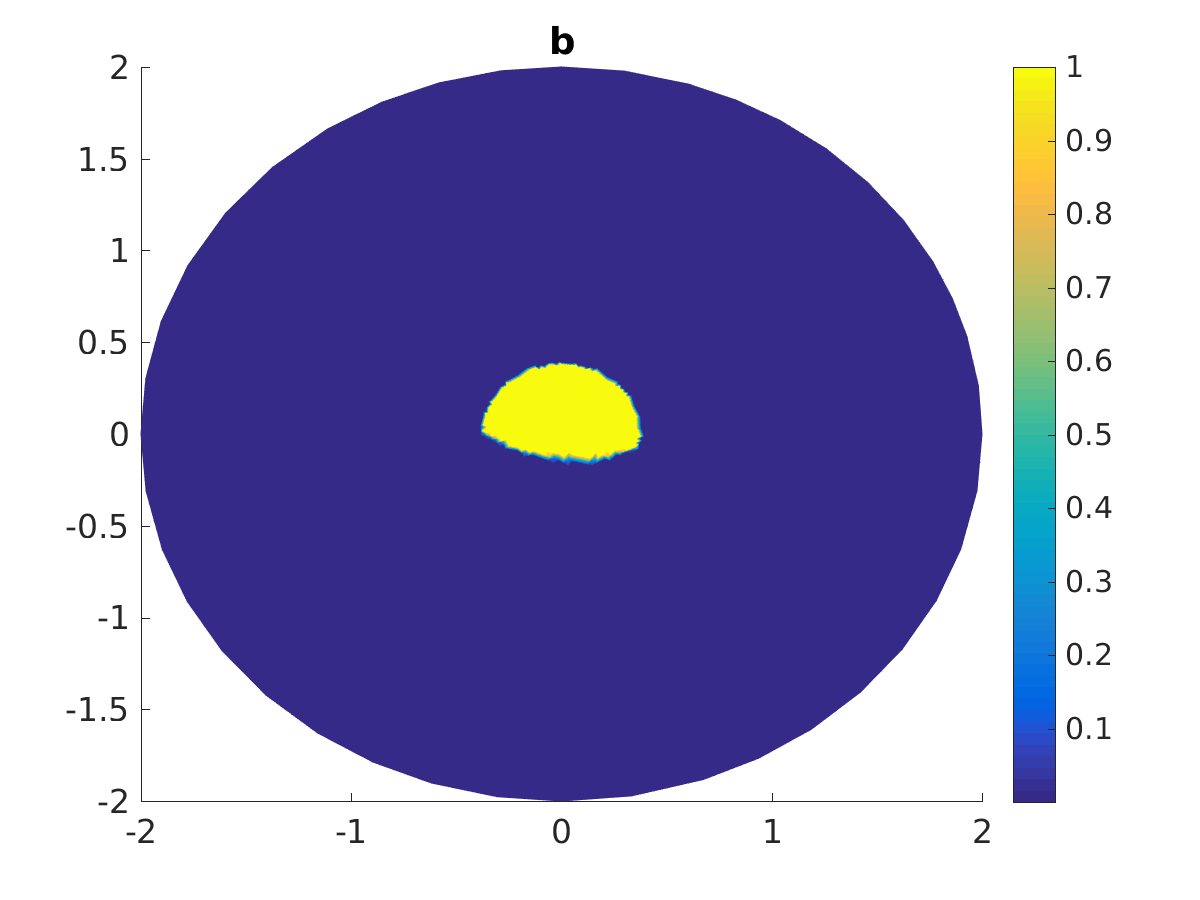}
      \caption{$(c_{11},c_{22}) = (-2,-2)$\\$m_r=0.15$, $m_b=0.3$}
  \end{subfigure}%
  \begin{subfigure}[t]{0.26\textwidth}
      \centering
      \includegraphics[width=\textwidth]{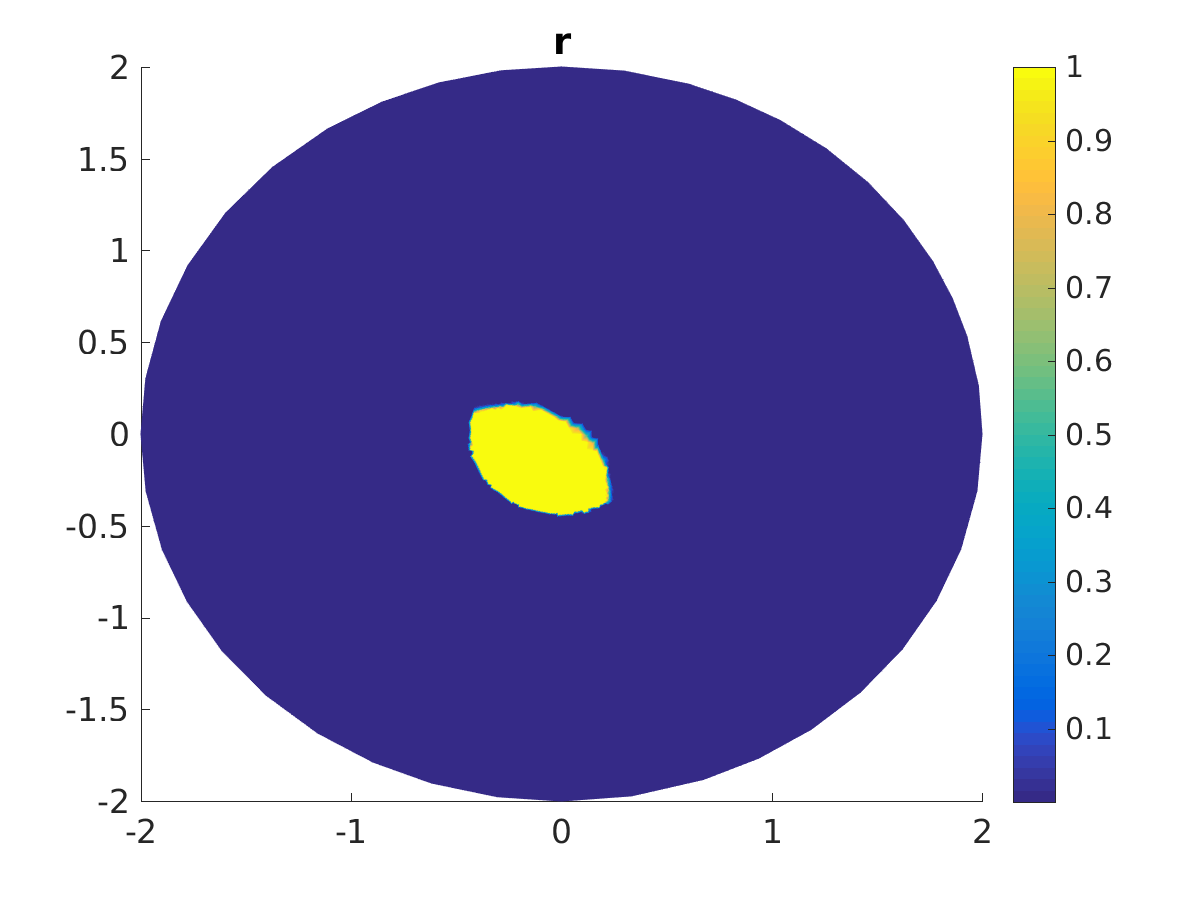}\\
      \includegraphics[width=\textwidth]{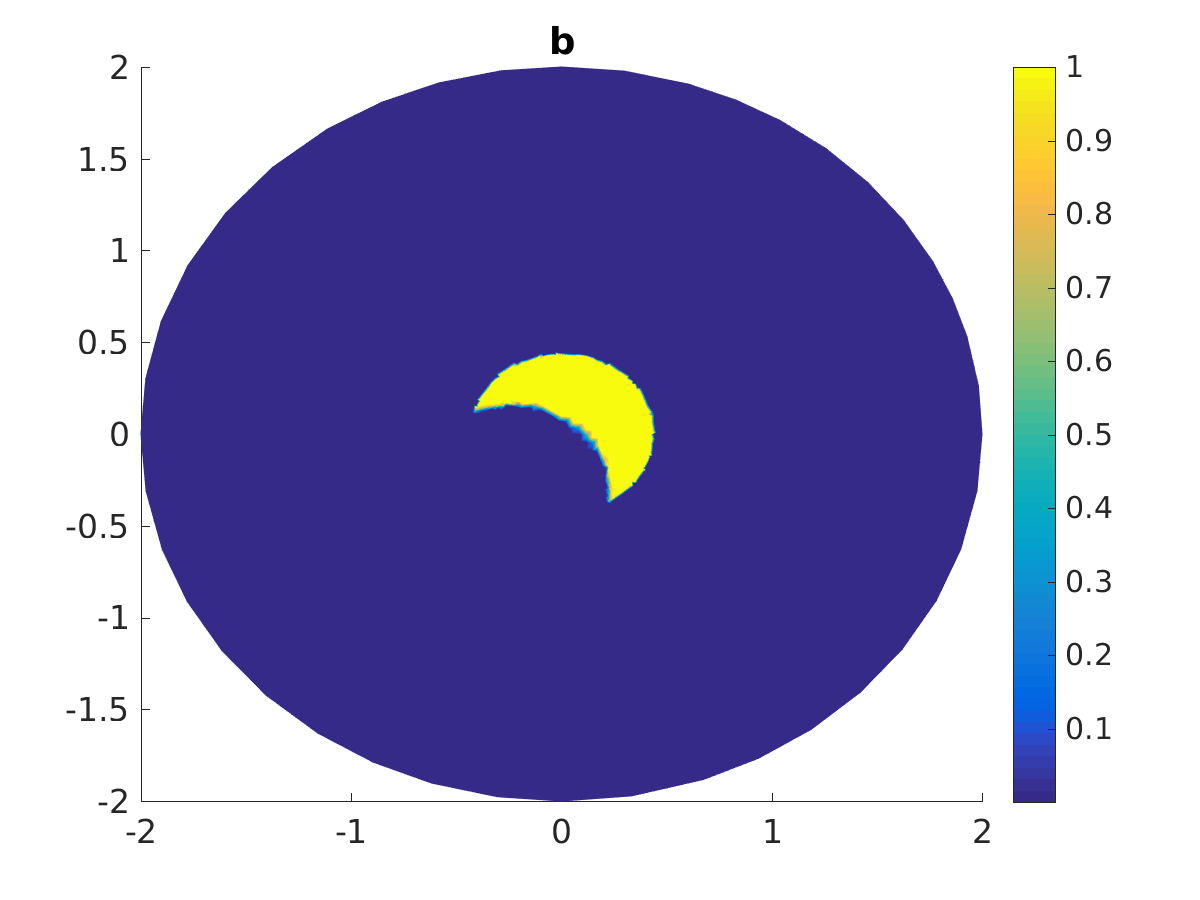}
      \caption{$(c_{11},c_{22}) = (-2,-1.5)$\\$m_r=0.3$, $m_b=0.3$}
  \end{subfigure}%
  \begin{subfigure}[t]{0.26\textwidth}
      \centering
      \includegraphics[width=\textwidth]{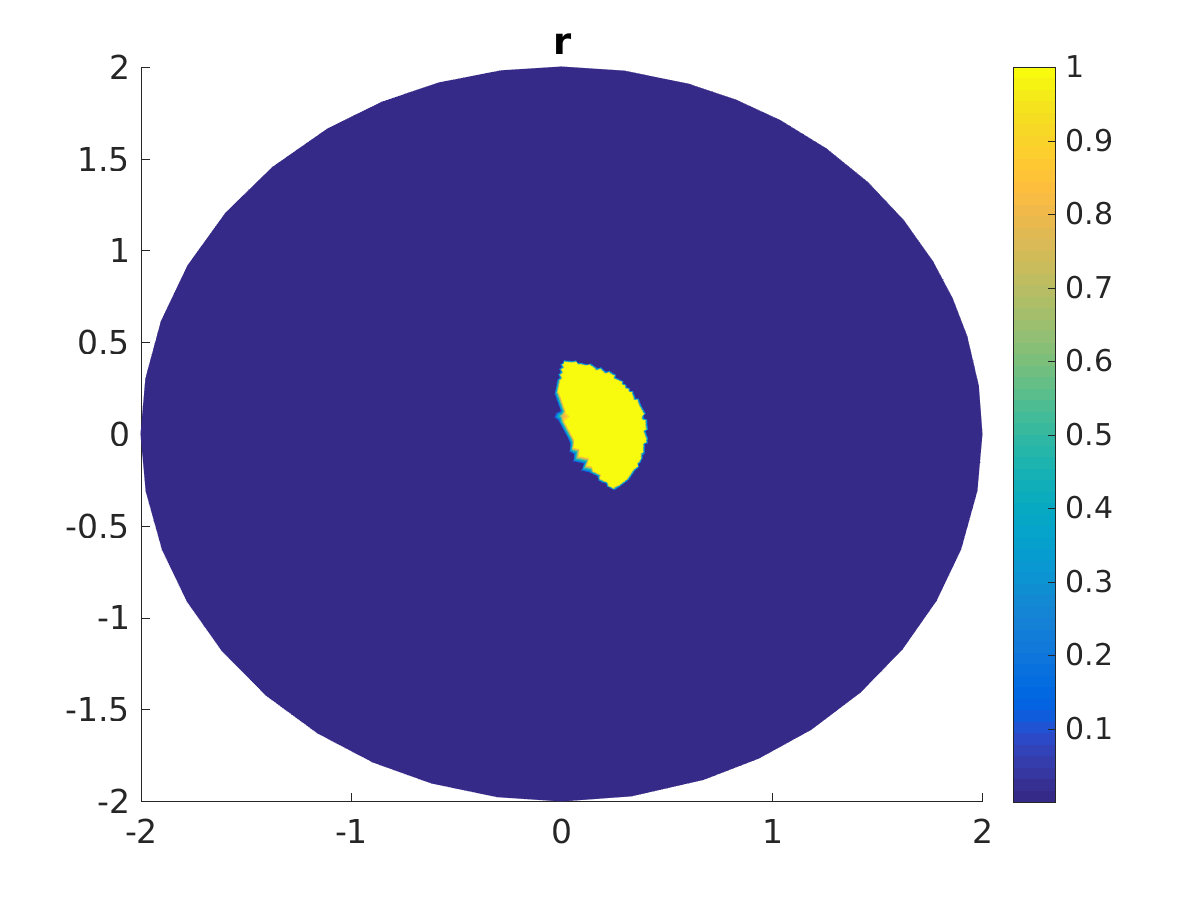}\\
      \includegraphics[width=\textwidth]{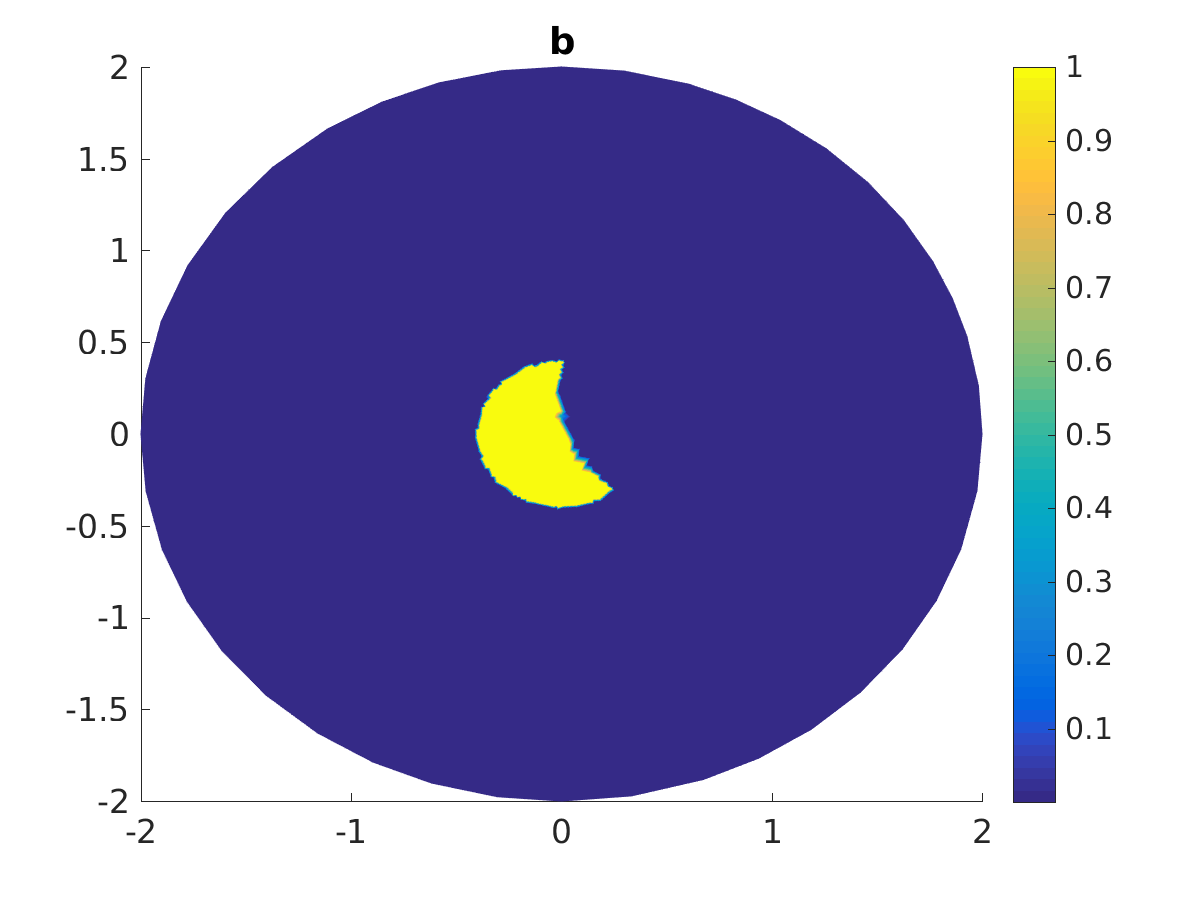}
      \caption{$(c_{11},c_{22}) = (-2,-1.5)$\\$m_r=0.2$, $m_b=0.3$}
  \end{subfigure}%
 \end{center}
 \caption{The cases $c_{11}=-2,\,c_{22}=-2$ with $m_r=m_b=0.3$ (a) and $m_r=0.15$ and $m_b=0.3$ (b) and the cases $c_{11}=-2,\,c_{22}=-1.5$ with $m_r=m_b=0.3$ (c) and $m_r=0.2$ and $m_b=0.3$ (d).}
 \label{fig:admm2d}
\end{figure}

It is instructive to reinterpret the last results in view of \eqref{eq:partitioning}. In the second column of Figure \ref{fig:epszero2d} we find the situation $d_{RV} > d_{RB} > d_{BV}$. Hence it is energetically favourable to avoid the interface between red and void, and put the largest interface between blue and void, which leads to an isoperimetric problem for the overall density. The partioning inside is simply obtained by a local isoperimetric problem for the red phase. The third column can be equally explained with the setting $d_{BV} > d_{RB} > d_{RV}$. The situation in Figure \ref{fig:admm2d} is more subtle. Here $d_{RV}$ and $d_{BV}$ are expected to be lower than $d_{BV}$ but the difference seems not large enough that a separation into two disjoint structures is favourable. On the other hand for $d_{RB}$ large it is not favourable to create a ring as in Figure \ref{fig:epszero2d}, but there is a local partitioning inside the ball leading to a smaller perimeter of the interface. We mention that numerical simulations with very large absolute values of the $c_{ii}$ yield similar results, hence it is not clear whether the case of $d_{RB}$ so large that it is favourable to create two separate structures for red and blue can be obtained as an asymptotic from our model. 

\section{Transient Model}\label{sec:PDE}
As described in the introduction, minimizers of \eqref{eq:Feps} are closely related to solutions of systems of nonlinear cross-diffusion partial differential equations \cite{Burger2010}. In fact, introducing the nonlinear diagonal mobility tensor ${\bf M}(r,b)$ given by
\begin{equation*}
{\bf M}(r,b) = \left(\begin{matrix}
r(1-\rho) & 0 \\ 0 & Db (1-\rho)
\end{matrix}\right)
\end{equation*}
we can introduce the following formal gradient flow with respect to $E^\eps$
\begin{equation}
	\partial_t \left( \begin{array}{l} r \\ b \end{array} \right) = \nabla \cdot
\left({\bf M}(r,b) {\bf :} \nabla \left( \begin{array}{l} \partial_r E^\eps(r,b) \\ \partial_b E^\eps(r,b) \end{array} \right)\right). \label{eq:gradflow0}
\end{equation}
Inserting the definition \eqref{eq:FC} of $E^\eps$, we obtain \eqref{eq:crossdiffusion1}-\eqref{eq:crossdiffusion2}.
This system is set on a domain $\Omega\subset\RR^n$ and for $t\in [0,T]$. If $\Omega$ is bounded, we impose the following no-flux boundary conditions
\begin{align}\label{eq:noflux}
	n \cdot( (1-\rho) \nabla r + r \nabla \rho + r(1-\rho)\left[ \nabla (c_{11}K\ast r - K\ast b) + \nabla V\right]) &= 0,\\\nonumber
	n \cdot(  (1-\rho) \nabla b + b \nabla \rho + b(1-\rho)\left[\nabla (c_{22}K\ast b - K\ast r) + \nabla V\right]) &= 0,
\end{align}
on $\partial \Omega \times (0,T)$. This system is further supplemented by initial values
\begin{equation}
	r(\cdot,0) = r_0, \quad b(\cdot,0) = b_0 \text{ in } \Omega.
\end{equation}
Using the functional $\eps F^E+F^C$, we introduce \emph{entropy variables}, given as the first derivative with respect to $r$ and $b$
\begin{align}
	u := \partial_r (\eps F^E+F^C) &=\eps (\log r - \log(1-\rho))  + V, \label{eq:etar} \\
	v := \partial_b (\eps F^E+F^C) &=\eps( \log b - \log(1-\rho))  + V. \label{eq:etab}
\end{align}
Inverting these relations yields
\begin{align}\label{eq:etainv}
r= \frac{e^{\frac{u-V}{\eps}}}{1+e^{\frac{u-V}{\eps}}+e^{\frac{v-V}{\eps}}},\quad b= \frac{e^{\frac{v-V}{\eps}}}{1+e^{\frac{u-V}{\eps}}+e^{\frac{v-V}{\eps}}}. 
\end{align}
\subsection{Existence of weak solutions}
For further use we define the set $\mathcal{M}$
\begin{align*}
\mathcal{M}=\{(r,\, b) \in [L^2(\Omega), \RR^+]^2 : 0 < r,\, b; \, r+b=\rho < 1\text{ a.e.} \}.
\end{align*}
Global existence is guaranteed by the following theorem
\begin{thm}\label{thm:pdebnd}
Let $T>0$, $\eps>0$ and $D>0$. Consider the following partial differential equations on an open and bounded Lipschitz domain $\Omega \subset \mathbb{R} ^N$ 
\begin{align}\label{eq:pdebnd1}
 \partial_t r&= \nabla \cdot (\eps [(1-\rho)\nabla r +r \nabla \rho ]+r(1-\rho)\nabla V+ r(1-\rho)\nabla(c_{11}K\ast r - K\ast b)) \\\label{eq:pdebnd2}
\partial_t b&=  \nabla \cdot (\eps D [(1-\rho)\nabla b +b \nabla \rho ]+b(1-\rho)\nabla V+ Db(1-\rho)\nabla(c_{22}K\ast b - K\ast r))
\end{align}
for $V$ satisfying \textbf{(V1)}, $c_{ij}\leq 0$ for $i,\, j \in \{1,2\}$ and $K$ as in definition \eqref{def:admissible}. The system is equipped with initial conditions 
\begin{align*}
r(\cdot,\, 0)=r_0 \quad \text{and} \quad b(\cdot,\, 0)=b_0 
\end{align*}
with $(r_0,b_0)\in\mathcal{M}$ and no flux boundary conditions \eqref{eq:noflux}.
Then there exists a weak solution $(r,\, b)$ in 
$$W=(L^2((0,T),\, L^2(\Omega))\cap H^1((0,\, T),\, H^{-1}(\Omega)))^2$$ such that 
$$\rho,\, \sqrt{1-\rho}r,\, \sqrt{1-\rho}b \in L^2((0,\,T),\, H^1(\Omega))$$
and furthermore $0 \leq r,\, b$ and $ \rho \leq 1$ almost everywhere.
\end{thm}
We remark that the proof without non-local interactions is given in \cite[Theorem 4.1]{Burger2010}, yet using a different approximation technique. Here, we use a semi-discretization in time that has also been used in \cite{Burger2016}, but applied to a different, anisotropic system.
%
\begin{proof}
Without loss of generality we choose $D=1$ in the proof and furthermore denote by $C$ a generic constant whose value can change from line to line. For different $D>0$, some constants change yet the arguments remain the same. We prove the theorem in three steps: First we obtain an a priori bound based on the entropy dissipation, second we apply a time discretization and use a fixed point argument to ensure existence of weak solutions for every time step. Finally, we finish the proof passing to continuous time.\\
\textbf{First step: Entropy inequality} We calculate the dissipation of the diffusive part of the energy functional
\begin{align*}
&\frac{d}{dt} (\eps F^E+F^C)(r,\, b) = \int_{\Omega} \partial_t r u+ \partial_t b v \,dx \\
&\quad=  \int_{\Omega} \nabla \cdot (\eps [(1-\rho)\nabla r +r \nabla \rho ]+r(1-\rho)\nabla V + r(1-\rho)\nabla(c_{11}K\ast r - K\ast b))u \\
&\quad \quad+ \nabla \cdot (\eps[(1-\rho)\nabla b +b \nabla \rho ]+b(1-\rho)\nabla V+ b(1-\rho)\nabla(c_{22}K\ast b - K\ast r ))v \,dx.
\end{align*}
Using the gradient of the entropy variables \eqref{eq:etar}--\eqref{eq:etab}
\begin{align*}
\nabla u&=\frac{\eps}{r} \nabla r +\frac{\eps}{1-\rho} \nabla \rho+ \nabla V,\\
\nabla v&=\frac{\eps}{b} \nabla b +\frac{\eps}{1-\rho} \nabla \rho+ \nabla V,
\end{align*}
we obtain
\begin{align} \label{Youngapp}
 \frac{d}{dt}& (\eps F^E+F^C) + \int_{\Omega} r(1-\rho)(\nabla u)^2+ b(1-\rho)(\nabla v)^2 \,dx  \\ \nonumber
 &\le \int_{\Omega} -r(1-\rho) (c_{11}\nabla K *r - \nabla K *b)\nabla u - b(1-\rho) (c_{22}\nabla K\ast b - \nabla K\ast r)\nabla v \, dx,
\end{align}
By the definition of the entropy variables \eqref{eq:etar}--\eqref{eq:etab} we see that $(u,\,v)\in [L^1(\Omega)]^2$. Therefore $(u,\, v)$ take finite values almost everywhere and we conclude that $(r,\,b) \in \mathcal{M}$, see also \cite{Juengel2014}.
Since $K \in W^{1,1}_{loc}(\Omega)$ as defined in \ref{def:admissible} and $\Omega$ is bounded, we can apply Young's inequality for convolutions   
\begin{align*}
\|\nabla K\ast r\|_{L^2(\Omega)} \le \|\nabla K\|_{L^1(\Omega)} \|r\|_{L^2(\Omega)},
\end{align*}
and the analogue for $b.$
Now, rewriting $r(1-\rho)=\sqrt{r(1-\rho)}\sqrt{r(1-\rho)}$ (analogously for $b(1-\rho)$) and using Young's inequality allows us to absorb terms  $r(1-\rho)(\nabla u)^2,\,b(1-\rho)(\nabla v)^2$ and we obtain
\begin{align} \label{entropyinequality}
&\frac{d}{dt} (\eps F^E+F^C) +\frac{1}{2} \int_{\Omega} r(1-\rho)(\nabla u)^2+ b(1-\rho)(\nabla v)^2 \,dx  \nonumber \\
&\quad \leq (-c_{11}\|r\|_{L^2(\Omega)}+\|b\|_{L^2(\Omega)})^2 \| r(1-\rho) \|_{L^{\infty}(\Omega)} 
+ (-c_{22}\|b\|_{L^2(\Omega)}+\|r\|_{L^2(\Omega)})^2\|b(1-\rho)\|_{L^{\infty}(\Omega)}  \nonumber \\
&\quad \leq \frac{1}{4}(-c_{11}\|r\|_{L^2(\Omega)}+\|b\|_{L^2(\Omega)})^2+ \frac{1}{4}(-c_{22}\|b\|_{L^2(\Omega)}+\|r\|_{L^2(\Omega)})^2 =: C
\end{align} 
where in the last step we used $r(1-\rho), b(1-\rho)\le \frac{1}{4}$. The positive constant $C$ depends on the interaction kernel $K$ and on the constants $c_{11},\, c_{12}$ only. In particular, it only depends on the size of the domain $\Omega$ via the $L^2$-Norms of $r$ and $b$. Expanding the terms containing $\nabla u$ and $\nabla v$ and integrating in time we finally obtain
\begin{align*} 
& (\eps F^E+F^C)(r,\, b) + \frac{\eps^2}{4} \int_0^T\int_{\Omega} (1-\rho)|\nabla\sqrt{r}|^2 + (1-\rho)|\nabla\sqrt{b}|^2 + |\nabla\sqrt{1-\rho}|^2 + 4|\nabla \rho|^2 \,dx dt  \nonumber \\
&\quad + 2\eps \int_0^T\int_{\Omega} \nabla \rho \nabla V + \rho (1-\rho) (\nabla V)^2 \,dxdt \nonumber \\
&\quad \le (\eps F^E+F^C)(r_0,\, b_0) + CT.
\end{align*}
Now we can use the Young's inequality weighted with $\frac{1}{4}$ together with the fact that $\nabla V$ is bounded in $L^2(\Omega)$
\begin{align*} 
& (\eps F^E+F^C)(r,\, b) + \frac{\eps^2}{4} \int_0^T\int_{\Omega} (1-\rho)|\nabla\sqrt{r}|^2 + (1-\rho)|\nabla\sqrt{b}|^2 + |\nabla\sqrt{1-\rho}|^2 + 4|\nabla \rho|^2 \,dx dt  \nonumber \\
&\quad \leq (\eps F^E+F^C)(r_0,\, b_0) + CT - 2\eps \int_0^T\int_{\Omega} \nabla \rho \nabla V - \rho (1-\rho) (\nabla V)^2 \,dxdt \nonumber \\ .
&\quad \leq (\eps F^E+F^C)(r_0,\, b_0) + CT + 2\eps \int_0^T\int_{\Omega} \left| \nabla \rho \nabla V\right| + \rho (1-\rho) (\nabla V)^2 \,dxdt \nonumber \\ 
&\quad \leq (\eps F^E+F^C)(r_0,\, b_0) + CT + \int_0^T\int_{\Omega} \frac{\eps^2}{2} (\nabla \rho)^2 + (\nabla V)^2 \,dxdt+  2\eps\int_0^T\int_{\Omega} \rho (1-\rho) (\nabla V)^2 \,dxdt \nonumber \\ 
&\quad \leq (\eps F^E+F^C)(r_0,\, b_0) + CT + T\|\nabla V\|_{L^2(\Omega)}^2(1+\frac{\eps}{2})+ \int_0^T\int_{\Omega} \frac{\eps^2}{2}(\nabla \rho)^2\,dxdt
\end{align*}
Summarizing we find 
\begin{align} \label{entropyinequality2}
& (\eps F^E+F^C)(r,\, b) + \frac{\eps^2}{4} \int_0^T\int_{\Omega} (1-\rho)|\nabla\sqrt{r}|^2 + (1-\rho)|\nabla\sqrt{b}|^2 + |\nabla\sqrt{1-\rho}|^2 + 2|\nabla \rho|^2 \,dx dt  \nonumber \\
&\quad \leq (\eps F^E+F^C)(r_0,\, b_0) + CT.
\end{align}
\textbf{Second step: Time discretization} We continue by rewriting the system using again the definition of the entropy variables \eqref{eq:etar}--\eqref{eq:etab}. We obtain
\begin{align*}
 \left(\begin{array}{c}
      \partial_t r \\
      \partial_t b
    \end{array}\right) 
= \left( \begin{array}{c}
\nabla \cdot ( r(1-\rho)\nabla u + r(1-\rho)[c_{11}\nabla K\ast r-\nabla K\ast b + \nabla V] ) \\
\nabla \cdot ( b(1-\rho)\nabla v + b(1-\rho)[c_{22}\nabla K\ast b-\nabla K\ast r + \nabla V] )
\end{array} \right).
\end{align*}
Next we discretize the system in time using the Euler implicit scheme. We split the time interval as 
\begin{align*}
 (0,\, T]=\sum_{k=1}^n ((k-1)\tau,\, k\tau],
\end{align*}
with $\tau=\frac{T}{n}.$ Furthermore, we add an additional regularization term. The system then becomes 
\begin{align} \label{problem}
\frac{1}{\tau}\left(
\begin{array}{c}
r_{k-1}-r_k\\
 b_{k-1}-b_k\\
\end{array}
\right)&=\left( \begin{array}{c}
\nabla \cdot ( r_k(1-\rho_k)\nabla u_k + r_k(1-\rho_k)[c_{11}\nabla K\ast r_k-\nabla K\ast b_k + \nabla V] ) \\
\nabla \cdot ( b_k(1-\rho_k)\nabla v_k + b_k(1-\rho_k)[c_{22}\nabla K\ast b_k-\nabla K\ast r_k + \nabla V] )
\end{array} \right) \nonumber \\
&+\tau \left(
\begin{array}{c}
\Delta u_{k}-u_k\\
 \Delta v_k-v_k\\
\end{array}
\right).
\end{align}
We consider the weak formulation of the problem \eqref{problem} for all $(\Phi_1,\Phi_2) \in H^1(\Omega)^2$ 
\begin{align} \label{weakformulation}
&\frac{1}{\tau} \int_{\Omega}\left(
\begin{array}{c}
r_k-r_{k-1}\\
 b_k-b_{k-1}\\
\end{array}
\right) \cdot \left(
\begin{array}{c}
\Phi_1\\
\Phi_2\\
\end{array}
\right)\,dx+ \int_{\Omega}\left(
\begin{array}{c}
\nabla \Phi_1\\
\nabla \Phi_2\\
\end{array}
\right)^T G(r_k,\, b_k) \left(
\begin{array}{c}
\nabla u_k \\
\nabla v_k\\
\end{array}
\right)\,dx \nonumber \\
&+ \int_{\Omega} H(r_k,\, b_k)\left(
\begin{array}{c}
\nabla \Phi_1\\
\nabla \Phi_2\\
\end{array}
\right)\,dx+ \tau R\left(\left(
\begin{array}{c}
 \Phi_1\\
 \Phi_2\\
\end{array}
\right), \left(
\begin{array}{c}
 u_{k}\\
 v_{k}\\
\end{array}
\right) \right)=0,
\end{align}
where we introduced
\begin{equation} \label{G}
G(r,\, b)=  \left( \begin{array}{cc}
r(1-\rho) & 0 \\
0 & b(1-\rho) \end{array} \right),
\end{equation}
and
\begin{equation} \label{defH}
H(r,\, b)=  \left( \begin{array}{c}
r(1-\rho)[c_{11}\nabla K\ast r-\nabla K\ast b] \\
b(1-\rho)[c_{22}\nabla K\ast b-\nabla K\ast r] \end{array} \right).
\end{equation}
Note that the matrix $G(r,\, b)$ is positive semi-definite, because $r,\, b \geq 0$ and $\rho \leq 1.$ We furthermore abbreviated the regularization term as
\begin{equation} \label{R}
R\left(\left(
\begin{array}{c}
 \Phi_1\\
 \Phi_2\\
\end{array}
\right), \left(
\begin{array}{c}
 u_{k}\\
 v_{k}\\
\end{array}
\right) \right) =\int_{\Omega} \Phi_1 u_k + \Phi_2 v_k+ \nabla \Phi_1 \nabla u_k + \nabla \Phi_2 \nabla v_k \, dx dy.
\end{equation}
To solve the nonlinear equation \eqref{weakformulation} for $(u_k, \, v_k)$, we first linearize it by replacing the terms $G(r_k,\,b_k)$ and $H(r_k,b_k)$ by $G(\tilde r,\tilde b)$ and $H(\tilde r, \tilde b)$, respectively for given $(\tilde r, \tilde b) \in [L^2(\Omega)]^2$ and $(\tilde r,\tilde b) \in \mathcal{M}$.
To treat this linear problem, we introduce the bilinear form
\begin{align*}
a&: [H^1(\Omega)]^2 \times [H^1(\Omega)]^2 \to \mathbb{R}
\end{align*}
given by
\begin{align}
a((u,\, v),(\Phi_1,\, \Phi_2))&=\int_{\Omega} \left(
\begin{array}{c}
\nabla \Phi_1\\
\nabla \Phi_2\\
\end{array}
\right)^T G(\tilde{r},\, \tilde{b}) \left(
\begin{array}{c}
\nabla u \\
\nabla v\\
\end{array}
\right)\,dx
+ \tau R\left(\left(
\begin{array}{c}
 \Phi_1\\
 \Phi_2\\
\end{array}
\right), \left(
\begin{array}{c}
 u\\
 v\\
\end{array}
\right) \right),
\end{align}
and the linear form $B$ 
$$B(\Phi_1,\, \Phi_2)=-\frac{1}{\tau} \int_{\Omega}\left(
\begin{array}{c}
\tilde{r}-r_{k-1}\\
 \tilde{b}-b_{k-1}\\
\end{array}
\right) \cdot \left(
\begin{array}{c}
\Phi_1\\
\Phi_2\\
\end{array}
\right)\,dx+ \int_{\Omega} H(\tilde{r},\tilde{b}) \left(
\begin{array}{c}
\nabla \Phi_1\\
\nabla \Phi_2\\
\end{array}
\right)\,dx.$$ 
Thus, \eqref{weakformulation} is equivalent to 
\begin{equation} \label{defa}
a((u,v),(\Phi_1,\, \Phi_2))=B(\Phi_1,\Phi_2)\quad\forall (\Phi_1,\Phi_2) \in [H^1(\Omega)]^2.
\end{equation}
Writing out the definitions of the matrices, it is straight forward to see that $a$ is bounded
$$ a\left( \left(
\begin{array}{c}
u\\
v\\
\end{array}
\right), \left(
\begin{array}{c}
\Phi_1\\
\Phi_2\\
\end{array}
\right)\right) \leq (1+\tau) \left\Vert \left(
\begin{array}{c}
u\\
v\\
\end{array}
\right)\right\Vert_{H^1(\Omega)}  \left\Vert\left(
\begin{array}{c}
\Phi_1\\
\Phi_2\\
\end{array}
\right)\right\Vert_{H^1(\Omega)}.$$
The regularization term ensures the coerciveness of $a$
$$a ((u,\, v), (u,\, v))= \int_{\Omega} \left(
\begin{array}{c}
\nabla u\\
\nabla v\\
\end{array}
\right)^T G(\tilde{r},\tilde{b}) \left(
\begin{array}{c}
\nabla u \\
\nabla v\\
\end{array}
\right)\,dx
+ \tau R\left(\left(
\begin{array}{c}
 u\\
 v\\
\end{array}
\right), \left(
\begin{array}{c}
 u\\
 v\\
\end{array}
\right) \right) $$ 
$$\geq \tau R\left(\left(
\begin{array}{c}
 u\\
 v\\
\end{array}
\right), \left(
\begin{array}{c}
 u\\
 v\\
\end{array}
\right) \right)= \tau( \|u\|^2_{H^1(\Omega)}+\|v\|^2_{H^1(\Omega)}).$$ 
The boundedness of $B$ as a linear functional $H^1(\Omega)^2 \to \mathbb{R}$ is a consequence of Young's inequality for convolutions and the regularity of $K$. 
We conclude existence of unique solutions $(u,v)$ since all assumptions of the Lax-Milgram lemma, \cite[Sec. 6.21., Thm 1]{Evans1998}, are satisfied. This allows us to define a fixed point operator
$$\mathcal{F}: \mathcal{M} \to \mathcal{M}, (\tilde{r},\, \tilde{b}) \to (r,\,b)=DE^{-1}(u,\, v)$$
with $(u,\,v) \in [H^1(\Omega)]^2$ being the unique solution to \eqref{defa}. We aim to apply Leray-Schauder's fixed point theorem as stated in \cite[Thm 11.3]{Gilbarg1983}. We know that $\mathcal{M}$ is bounded in $L^2(\Omega)^2,$ because $\Omega$ is assumed to be bounded 
$$  \|r\|_{L^2(\Omega)} \leq \|r\|_{L^{\infty}(\Omega)}|\Omega| = |\Omega| < \infty$$
and the analogue for the second component $b.$
It remains to establish compactness of $\mathcal{F}$. Consider a sequence $(\tilde{r}_k, \, \tilde{b}_k) \in\mathcal{M}$ that converges to $(\tilde{r},\, \tilde{b})$ strongly in $[L^2(\Omega)]^2$. Since $\tau > 0$, the regularization term ensures that $(u_k,\,v_k)$ are uniformly bounded in $[H^1(\Omega)]^2$ and thus there exists subsequences (which we don't relabel) such that their gradients converge weakly in $[L^2(\Omega)]^N$. Since the entries of $G(r_k,b_k)$ are uniformly bounded in $L^\infty(\Omega)$, we also have strong convergence of $G(r_k,b_k)$ to $G(r,b)$ in $L^2(\Omega)$. This, together with a standard approximation of the test functions in $W^{1,\infty}(\Omega)$ allows us to pass to the limit and conclude the continuity of $a$. For the linear form $F$, continuity, now with respect to $\tilde r_n,\tilde b_n$ is once more a consequence of Young's inequality for convolutions. Finally, the fact that the mapping \eqref{eq:etainv} is differentiable and thus Lipschitz continuous, we conclude the continuity of the whole operator $\mathcal{F}$. Together with the compactness of the embedding $H^1(\Omega) \hookrightarrow L^2(\Omega)$, Schauder's fixed point theorem can be applied.
That means we have established existence of weak solutions for every iteration step. 
\\ \textbf{Third step: Limit $\tau \to 0$} Now we extend the result to continuous time, i.e. we consider the limit $\tau \to 0$. We want to use the convexity of the entropy $(\eps F^E+F^C),$ which together with its differentiability implies that
\begin{align*}
 \frac{d}{dt} (\eps F^E+F^C)(\phi_1)(\phi_1-\phi_2) \geq (\eps F^E+F^C)(\phi_1)-(\eps F^E+F^C)(\phi_2). 
\end{align*}
Now we discretize time and choose $\phi_1=(r_k,\, b_k)$ and $\phi_2=(r_{k-1},\, b_{k-1}).$ Using the definition of the entropy variables given in \eqref{eq:etar}-\eqref{eq:etab} yields 
\begin{align*}
\frac{1}{\tau}\int_{\Omega} \left(
\begin{array}{c}
r_k-r_{k-1}\\
b_k-b_{k-1}\\
\end{array}
\right) \cdot \left(
\begin{array}{c}
u_k\\
v_k\\
\end{array}
\right)\, dx \geq  ((\eps F^E+F^C)(r_k,\, b_k)-(\eps F^E+F^C)(r_{k-1},\, b_{k-1}).
\end{align*} 
We insert this inequality into the weak formulation \eqref{weakformulation} of our problem for $(\Phi_1,\Phi_2)=(u_k,\, v_k)$, we find
\begin{align*}
&(\eps F^E+F^C)(r_k,\, b_k) + 
\int_{\Omega} \left(
\begin{array}{c}
\nabla u_k\\
\nabla v_k\\
\end{array}
\right)^T G(r_k,\, b_k) \left(
\begin{array}{c}
\nabla u_k \\
\nabla v_k\\
\end{array}
\right)\,dx\\
&+ \int_{\Omega} H(r_k,\, b_k)\left(
\begin{array}{c}
\nabla u_k\\
\nabla v_k\\
\end{array}
\right)\, dx+ \tau R\left(\left(
\begin{array}{c}
 u_k\\
 v_k\\
\end{array}
\right), \left(
\begin{array}{c}
 u_k\\
 v_k\\
\end{array}
\right) \right) \\
&\leq (\eps F^E+F^C)(r_{k-1},\, b_{k-1}). 
\end{align*} 
Solving this recursion w.r.t. $k$ gives 
\begin{align*}
& (\eps F^E+F^C)(r_k,\, b_k) + \sum_{j=1}^k
\int_{\Omega} \left(
\begin{array}{c}
\nabla u_j\\
\nabla v_j\\
\end{array}
\right)^T G(r_j,\, b_j) \left(
\begin{array}{c}
\nabla u_j \\
\nabla v_j\\
\end{array}
\right)\,dx \\
&\quad+ \sum_{j=1}^k \int_{\Omega} H(r_j,\, b_j)\left(
\begin{array}{c}
\nabla u_j\\
\nabla v_j\\
\end{array}
\right)\, dx + \tau \sum_{j=1}^k R\left(\left(
\begin{array}{c}
 u_j\\
 v_j\\
\end{array}
\right), \left(
\begin{array}{c}
 u_j\\
 v_j\\
\end{array}
\right) \right)
\leq (\eps F^E+F^C)(r_0,\, b_0)
\end{align*} 
As a next step we use a piecewise constant interpolation, i.e. 
for $x \in \Omega$ and $t \in ((k-1)\tau, k\tau]$ we define
$$r_{\tau}(x,\, t)=r_k(x) \quad \text{and} \quad b_{\tau}(x,\, t)=b_k(x).$$
Furthermore let $\sigma_{\tau}$ be the shift operator such that for $\tau \leq t \leq T$
$$(\sigma_{\tau}r_{\tau})(x,\, t)=r_{\tau}(x,\, t-\tau) \quad \text{and} \quad 
(\sigma_{\tau}b_{\tau})(x,\, t)=b_{\tau}(x,\, t-\tau).$$
Then $(r_{\tau},\, b_{\tau})$ solves the following equation
\begin{align} \label{eq:previous}
&\frac{1}{\tau}\int_0^T \int_{\Omega} \left(\begin{array}{c}
  r_{\tau}-\sigma_{\tau}r_{\tau}\\
  b_{\tau}-\sigma_{\tau}b_{\tau}\\
  \end{array}
  \right) \cdot \left(
\begin{array}{c}
u_{\tau}\\
v_{\tau}\\
\end{array}
\right)\, dx dt + \int_0^T\int_{\Omega} \left(
\begin{array}{c}
(1-\rho_{\tau})\nabla r_{\tau}+ r_{\tau}\nabla \rho_{\tau}\\
(1-\rho_{\tau})\nabla b_{\tau}+ b_{\tau}\nabla \rho_{\tau}\\
\end{array}
\right) \cdot \left(
\begin{array}{c}
\nabla u_{\tau}\\
\nabla v_{\tau}\\
\end{array}
\right) \, dx dt \nonumber \\
&\quad + \int_0^T\int_{\Omega} H(r_\tau,\, b_\tau)\left(
\begin{array}{c}
\nabla u_{\tau}\\
\nabla v_{\tau}\\
\end{array}
\right)\, dxdt
+ \tau R\left(\left(
\begin{array}{c}
u_{\tau}\\
v_{\tau}\\
\end{array}
\right), \left(
\begin{array}{c}
 u_{\tau}\\
 v_{\tau}\\
\end{array}
\right) \right)=0.
\end{align}
Now, using again Young's inequality and reiterating the same steps as performed when deriving \eqref{entropyinequality2} from \eqref{Youngapp}, the above inequality becomes 
\begin{align}\label{eq:entropydiscrete}
&(\eps F^E+F^C)(r_{\tau}(T),\, b_{\tau}(T)) \nonumber \\
&+\frac{1}{2} \int_0^T\int_{\Omega} (1-\rho_\tau)|\nabla\sqrt{r_\tau}|^2 + (1-\rho_\tau)|\nabla\sqrt{b_\tau}|^2 + |\nabla\sqrt{1-\rho_\tau}|^2 + |\nabla \rho_\tau|^2 \,dxdt \nonumber \\
&\quad +\tau \int_0^T R\left(\left(
\begin{array}{c}
 u_{\tau}\\
 v_{\tau}\\
\end{array}
\right), \left(
\begin{array}{c}
 u_{\tau}\\
 v_{\tau}\\
\end{array}
\right) \right) \, dxdt 
\leq (\eps F^E+F^C)(r_0,\, b_0)+CT.
\end{align}
In order to pass to the limit $\tau \to 0$ and to show the convergence $r_{\tau} \to r$ and $b_{\tau} \to b$ we need the following two lemmas.
\begin{lemma} \label{estimates1}
The following estimates hold for a constant $\tilde{C} \in \mathbb{R}_+$ and $r_{\tau}, \, b_{\tau} \in L^{\infty}(\Omega)$
\begin{align}\label{eq:ineq1}
\| \sqrt{1-\rho_{\tau}}\nabla \sqrt{r_{\tau}} \|_{L^2((0,T),\, L^2(\Omega))} + \| \sqrt{1-\rho_{\tau}}\nabla \sqrt{b_{\tau}} \|_{L^2((0,T),\, L^2(\Omega))} &\leq \tilde{C},\\ \label{ineq2}
\| \sqrt{1-\rho_{\tau}}\|_{L^2((0,T),\, H^1(\Omega))}+\| \rho_{\tau}\|_{L^2((0,T), \,H^1(\Omega))} &\leq \tilde{C}\\\label{eq:estreguv}
 \sqrt{\tau}(\|u_\tau\|_{L^2((0,T);H^1(\Omega))} + \|v_\tau\|_{L^2((0,T);H^1(\Omega))}) &\le \tilde{C}
\end{align}
\end{lemma}
\begin{lemma} \label{estimates2}
 The discrete time derivatives of $r_{\tau}$ and $b_{\tau}$ are uniformly bounded
 \begin{align}\label{eq:lemmarho}
 \frac{1}{\tau} \|r_{\tau}-\sigma_{\tau}r_{\tau}\|_{L^2((0,T),\, H^{-1}(\Omega))}+\frac{1}{\tau} \|b_{\tau}-\sigma_{\tau}b_{\tau}\|_{L^2((0,T),\, H^{-1}(\Omega))} \leq C.
 \end{align}
 \end{lemma}
The first lemma is a direct consequence of the discrete in time entropy inequality \eqref{eq:entropydiscrete}, see \cite{Schlake2011,Burger2016} for details. The second results from using the estimates of the first one in the weak formulation \eqref{weakformulation}. Now we apply a special version of the Aubins-Lions lemma specifically designed for piecewise constant interpolations, cf. \cite[Theorem 1]{Dreher2012}. Using the Gelfand triple $H^1(\Omega)\hookrightarrow L^2(\Omega) \hookrightarrow H^{-1}(\Omega),$ this theorem implies that \eqref{eq:lemmarho} together with the boundedness of $\rho_\tau$ in $L^2((0,T);H^1(\Omega))$ is sufficient to obtain a strongly converging subsequence such that (without relabeling)
\begin{align*}
\rho_\tau \to \rho\text{ in }L^2((0,T);L^2(\Omega))\text{ as }\tau \to 0.
\end{align*}
This implies
\begin{align*}
\sqrt{1-\rho_{\tau}} \to \sqrt{1-\rho} \quad \text{in} \quad L^4((0,T),\, L^4(\Omega)).
\end{align*}
But since the space $L^4((0,T),\, L^4(\Omega))$ embeds continuously in $L^2((0,T),\, L^2(\Omega)),$ the function $\sqrt{1-\rho_{\tau}}$ must also converge strongly to $\sqrt{1-\rho}$ in $L^2((0,T),\,L^2(\Omega))$. Note that \eqref{eq:estreguv} immediately implies
\begin{align*}
 \tau u_\tau,\,\tau v_\tau \to 0 \text{ as }\tau\to 0\text{ strongly in }L^2((0,T);H^1(\Omega)),\\
 \tau \nabla u_\tau,\,\tau \nabla v_\tau \to 0 \text{ as }\tau\to 0\text{ strongly in }L^2((0,T);[L^2(\Omega)]^N).
\end{align*}
Finally, since $r_\tau,\,b_\tau$ are uniformly bounded in $L^\infty(\Omega)$, we know that there exist subsequences such that
\begin{align}\label{eq:rbweakstar}
 r_\tau \rightharpoonup^{*} r,\; b_\tau \rightharpoonup^{*} b\quad\text{ in } L^\infty((0,T);L^\infty(\Omega)).
\end{align}
In order to be able to pass to the limit in \eqref{weakformulation}, we rewrite
\begin{align*}
 (1-\rho_\tau)\nabla r_\tau + r_\tau\nabla \rho_\tau = \sqrt{1-\rho_\tau}\nabla (\sqrt{1-\rho_\tau}r_\tau) - 3 \sqrt{1-\rho_\tau}r_\tau \nabla\sqrt{1-\rho_\tau}
\end{align*}
The first term on the right hand side can be dealt with by further expanding the term $\nabla (\sqrt{1-\rho_\tau}r_\tau)$, the fact that $\sqrt{1-\rho_\tau}$ converges strongly, using the bounds \eqref{eq:ineq1} and the fact \eqref{eq:rbweakstar}. For the second term, one has to prove strong convergence of $\sqrt{1-\rho_\tau}r_\tau,$ which can be done either by using the Kolmogorov-Riesz theorem, see \cite{Burger2010} or again a generalized Aubins-Lions lemma, see \cite{Zamponi2015}. Note that this strong convergence also allows us to identify the corresponding, weakly converging gradient.
\end{proof}
As a next step, we extend the previous result to the case $\Omega = \RR^N$. The crucial ingredient will be the entropy inequality, which also holds on the whole space. We have
\begin{thm}
Let $T>0$, $\eps>0$, $\Omega = \RR^N$ and the assumptions on $V$, $K$ and the constants $c_{11}$ and $c_{22}$ as in Theorem \ref{thm:existence}. Then, the system \eqref{eq:pdebnd1}--\eqref{eq:pdebnd2}, 
with initial conditions
\begin{align*}
r(\cdot,\, 0)=r_0 \quad \text{and} \quad b(\cdot,\, 0)=b_0 
\end{align*}
with $(r_0,b_0)\in\mathcal{M}$ and the additional requirement that the second moment of both $r_0$ and $b_0$ are bounded. Then there exists a weak solution $(r,\, b)$ in 
\begin{align*}
W=(L^2((0,T),\, L^2(\RR^N))\cap H^1((0,\, T),\, H^{-1}(\RR^N)))^2 
\end{align*}
such that 
\begin{align*}
\rho,\, \sqrt{1-\rho}r,\, \sqrt{1-\rho}b \in L^2((0,\,T),\, H^1(\RR^N)), 
\end{align*}
and furthermore $0 \leq r,\, b$ and $\rho \leq 1$ almost everywhere.
\end{thm}
\begin{proof}
The proof is based on the same a priori estimates as in the bounded domain case and thus we only sketch the additional arguments. Indeed, the main ingredient is the entropy inequality, \eqref{entropyinequality2} which, due to the uniform $L^1 \cap L^\infty$-bounds on $r$ and $b$ also holds on $\RR^N$. This in particular implies the second moment bound on $r$ and $b$ which in turn yields the compactness of the embedding from $H^1(\RR^N)$ to $L^2(\RR^N)$. Indeed, taking a sequence $u_n \in H^1(\RR^N)$ we have, for some $u \in L^2$ and fixed $R>0$ that
\begin{align*}
 \int_{\RR^N} (u-u_n)^2\;dx = \int_{B_R(0)} (u-u_n)^2\;dx + \int_{\RR^N \setminus B_R(0)} (u-u_n)^2\;dx.
\end{align*}
Using the $L^\infty$-bound on $u_n$ and $u$, we can estimate the second term on the right hand side and obtain
\begin{align*}
 \int_{\RR^N} (u-u_n)^2\;dx \le \int_{B_R(0)} (u-u_n)^2\;dx + \frac{1}{R^2}\|u-u_n\|_{L^\infty(\RR^N)}\int_{\RR^N \setminus B_R(0)} |x|^2(u-u_n)\;dx.
\end{align*}
Taking the $\limsup_{n\to\infty}$, we see that the first term on the right hand side vanishes due to the compactness of $H^1(B_R(0)) \hookrightarrow L^2(B_R(0))$ for fixed $R$. Taking, in a second step the limit $R\to \infty$ yields the assertion.\\
Now we generate a sequence of approximate solutions $r_n$ and $b_n$ by solving the problem
\begin{align} \label{eq:weakbn}
&\int_0^T \int_{B_n(0)} \left(\begin{array}{c}
  \partial_t r_n\\
  \partial_t b_n
  \end{array}  \right) \cdot \left(
\begin{array}{c}
\Phi_1\\
\Phi_2\\
\end{array}\right)\, dx dt + \int_0^T\int_{B_n(0)} \left(
\begin{array}{c}
(1-\rho_n)\nabla r_n + r_n\nabla \rho_n\\
(1-\rho_n)\nabla b_n + b_n\nabla \rho_n
\end{array}
\right) \cdot \left(
\begin{array}{c}
\nabla \Phi_1\\
\nabla \Phi_2
\end{array}
\right) \, dx dt \nonumber \\
&\quad + \int_0^T\int_{B_n(0)} H(r_\tau,\, b_\tau)\left(
\begin{array}{c}
\nabla \Phi_1\\
\nabla \Phi_2
\end{array}
\right)\, dxdt = 0,
\end{align}
where we use restrictions of $r_0$ and $b_0$ to $B_n(0)$ as initital conditions. Due to the two ingredients, namely entropy inequality and compactness, we can pass to the limit exactly the same way as we did for time discretization in the proof of the previous theorem. The bounds of the right hand side of the weak formulation \eqref{eq:weakbn}, again together with the second moment bound, then also yield the $H^{-1}(\RR^N)$-bounds for the time derivatives.
\end{proof}
For $D=1$, we have the following regularity result for $\rho$ which may serve as a first step in proving uniqueness of solutions via a fixed point argument in that case.
\begin{prop}[Regularity of $\rho$ for $D=1$] Assume that the diffusion coefficient $D$ is one, and let $(r,b)$ be a weak solution to \eqref{eq:pdebnd1}--\eqref{eq:pdebnd2}. Then, on a bounded Lipschitz domain $\Omega$, there are positive constants $C_1, C_2$ such that the sum $r+b=\rho$ satisfies the following a priori estimate
\begin{align*}
 \|\rho\|_{L^\infty(0,T;W^{1,p}(\Omega))} \le C_1 + C_2 \|\rho_0\|_{W^{1,p}(\Omega)},
\end{align*}
for all $2 \le p < \infty.$
\end{prop}
\begin{proof} For $D=1$, we can add equations \eqref{eq:pdebnd1} and \eqref{eq:pdebnd2} to obtain an equation for $\rho$ given as
\begin{align}\label{eq:eqrho}
\partial_t\rho - \eps \Delta \rho = -\nabla\cdot h 
\end{align}
with 
\begin{align*}
h :=   r(1- \rho)(c_{11}\nabla K\ast  r - K\ast  b) +  b(1- \rho)(c_{22}\nabla K \ast  b - K\ast  r).
\end{align*}
and initial datum
\begin{align*}
\rho(\cdot,0) = \rho_0\text{ in }W^{1,p}(\Omega).
\end{align*}
Using the regularity of $r$ and $b$ as well as the properties of $K$, we can estimate $h$ as follows
\begin{align*}
\|h(\cdot,t)\|_{L^\infty(\Omega)} &\le \frac{1}{4}\|\nabla K\|_{L^1(\Omega)}\left [(c_{11}-1)\|r(\cdot,t)\|_{L^\infty(\Omega)} + (c_{22}-1)\|b(\cdot,t)\|_{L^\infty(\Omega)}\right] = C
\end{align*}
Since the constant on the right hand side does not depend on $t$ we can take the supremum over all $t\in[0,T]$ and conclude that $h\in L^\infty((0,T);L^\infty(\Omega))$. Thus applying the $L^p$ theory for linear parabolic equation, see e.g. \cite[Chapter VII]{Lieberman1996}, yields the assertion.
\end{proof}
\begin{rem}[Minimizers as stationary solutions] By remark \ref{rem:regmin} we know that minimizers of the energy $E^\eps$ are in fact in the space $[W^{1,\infty}(\Omega)]^2$. Therefore they can be inserted in the weak formulation of the stationary equation and are thus also stationary solutions of the PDE system \eqref{eq:pdebnd1}--\eqref{eq:pdebnd2} with no flux boundary conditions.
\end{rem}



\section{Numerical Simulations}\label{sec:numerics}
This Section deals with a numerical algorithm to solve the the PDE system \eqref{eq:pdebnd1}--\eqref{eq:pdebnd2} supplemented with no flux boundary conditions. We use the same P$1$ finite element method as described in \eqref{sec:nummin}. To discretize in time, we use the following implicit-explicit scheme: Given inital values $r^0,\,b^0 \in V_h$ and denoting the current iterates by $r^n,b^n$ (with $\rho^n=r^n+b^n$), we obtain new iterates $r^{n+1},\,b^{n+1}$ by solving the linear system
\begin{align*}
 \int_\Omega r^{n+1} \varphi\;dx &= \int_\Omega r^n \varphi\;dx + \tau \int_\Omega \eps\left[ (1-b^n)\nabla r^{n+1} + r^n\nabla b^{n+1}\right]\cdot \nabla \varphi\;dx \\
 &+ (1-\rho^n)\nabla (c_{11} K\ast r^n - K\ast b^n + V)\cdot \nabla \varphi\;dx\\
 \int_\Omega b^{n+1} \varphi\;dx &= \int_\Omega b^n \varphi\;dx + \tau \int_\Omega \eps\left[ (1-r^n)\nabla b^{n+1} + b^n\nabla r^{n+1}\right]\cdot \nabla \varphi\;dx \\
 &+ (1-\rho^n)\nabla (c_{22} K\ast b^n - K\ast r^n + V)\cdot \nabla \varphi\;dx.
\end{align*}
with time step size $\tau > 0$ given. Again, we use MATLAB for the implementation, both in one and two spatial dimensions.

\subsection{Examples in one spatial dimensions}
For all one-dimensional examples, we use the same potential $V$ as in Section \ref{sec:nummin}, defined in \eqref{eq:Vnum}
Furthermore, to decide when we reached a stationary state, we consider the relative $L^2$-error between two iterates, i.e
\begin{align*}
 \mathrm{err}_{L^2} := \|r^n - r^{n-1}\|_{L^2(V_h)} + \|b^n - b^{n-1}\|_{L^2(V_h)}
\end{align*}
and stop the iteration as soon as $\mathrm{err}_{L^2} < 1e-12$. We start by comparing the solutions for different values of $\eps$ to that of the ADMM scheme. As initial values we take 
\begin{align}
 r = b = \frac{1}{3}H_\gamma\left(\frac{1}{2}-|x|\right),\text{ with } \gamma = 0.001,
\end{align}
where $H^\gamma$ is the following approximation of the Heaviside function
\begin{align}
 H_\gamma(r) = \frac{1}{\pi}\left(\frac{\pi}{2}+\mathrm{atan}(\frac{r}{\gamma})\right),\quad \gamma > 0.
\end{align}
The results are visualized in Figure \ref{fig:diffeps1dpde} (left). In the top right picture the number of iterations until the error criterion is met are shown. On the bottom right, we plot, for a fixed time $t=19$ (and thus a fixed number of iterations), the relative error that is achieved for different values of $\eps$. As expected, both plots confirm that the dynamics become slower as $\eps$ becomes smaller - see the discussion about the relationship to systems of Cahn-Hilliards equations in the introduction.\\
 \begin{figure}
 \begin{center}
  \includegraphics[width=0.49\textwidth]{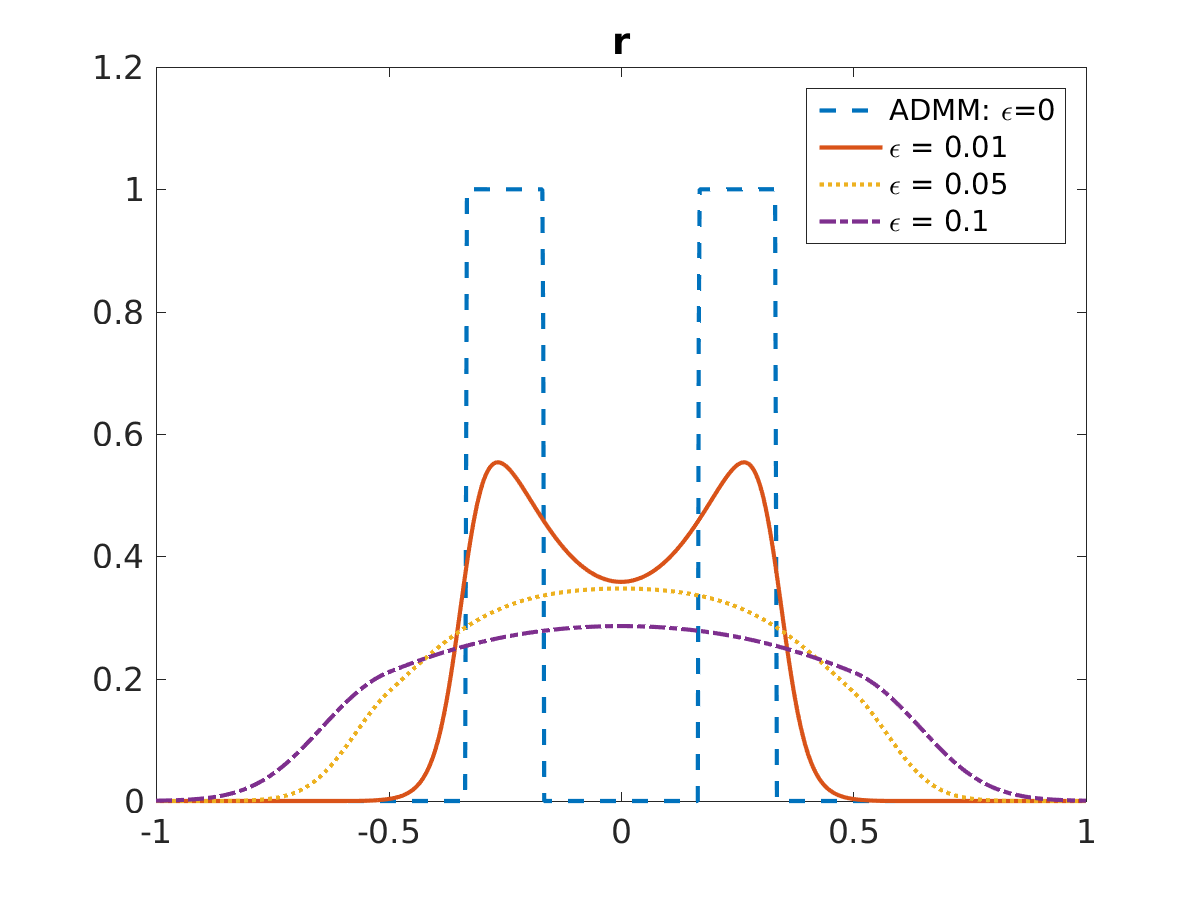}
  \includegraphics[width=0.49\textwidth]{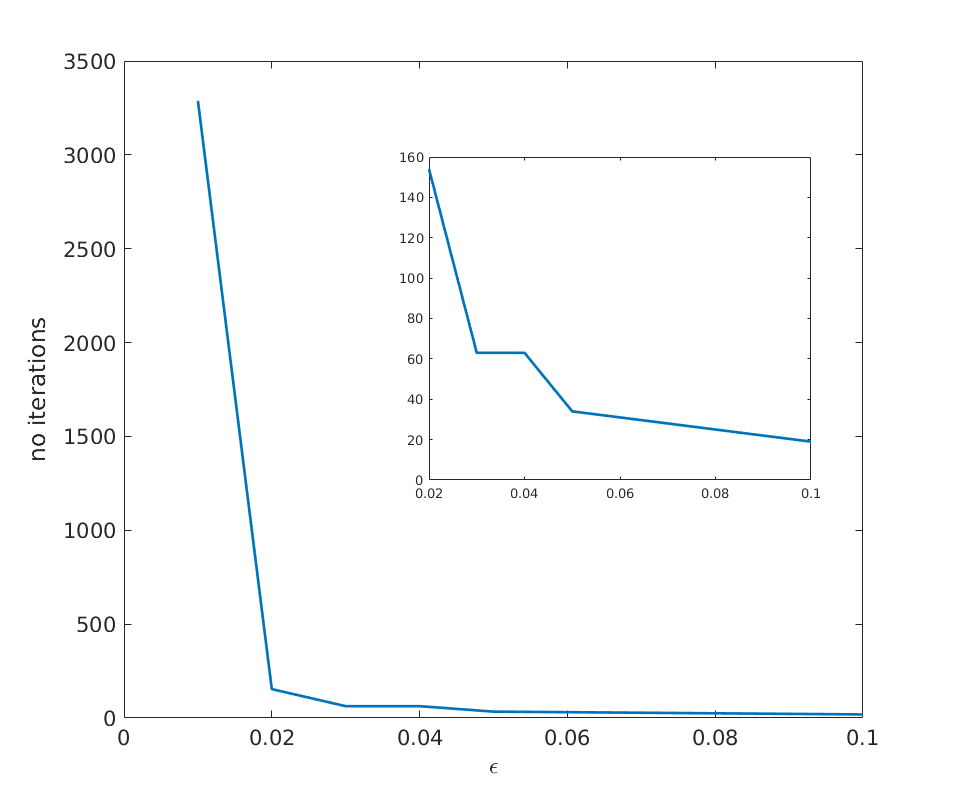}\\
  \includegraphics[width=0.49\textwidth]{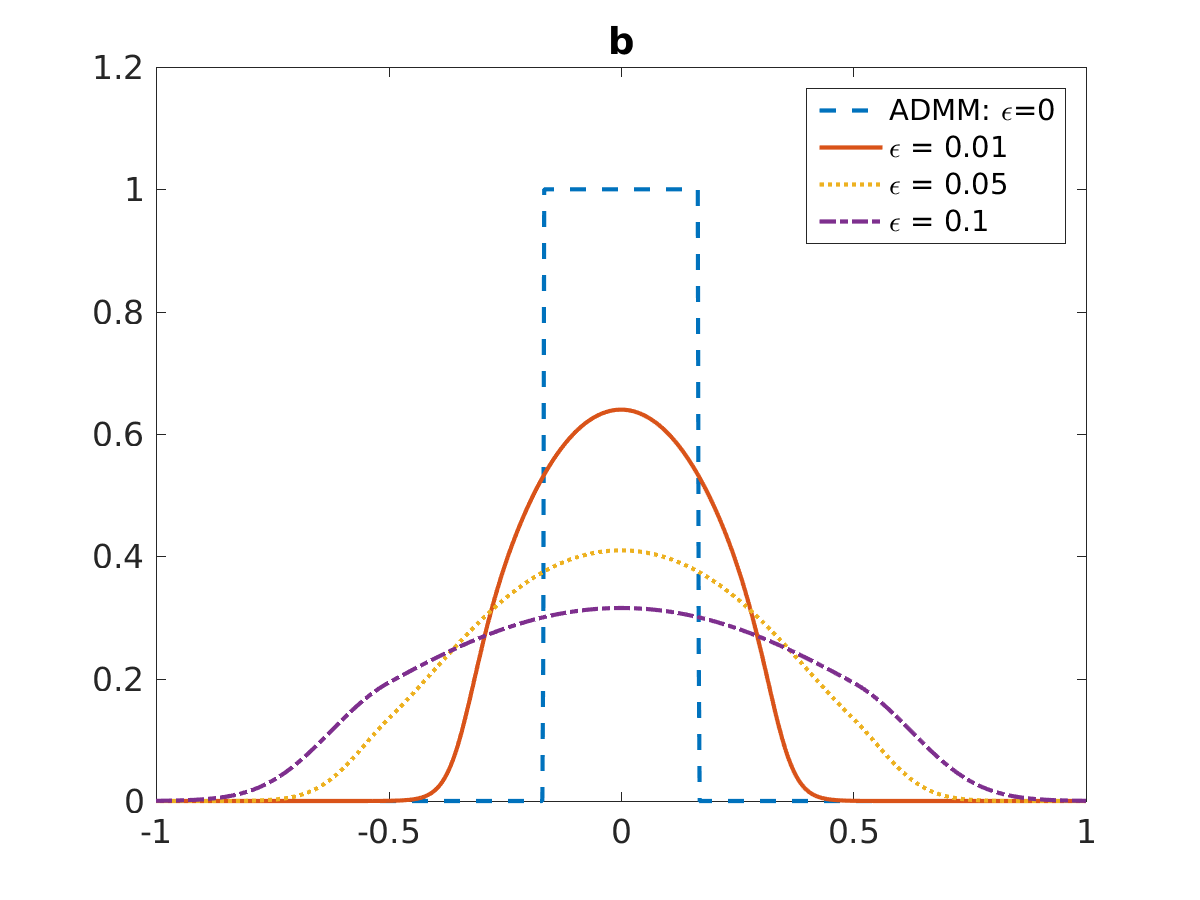}
  \includegraphics[width=0.49\textwidth]{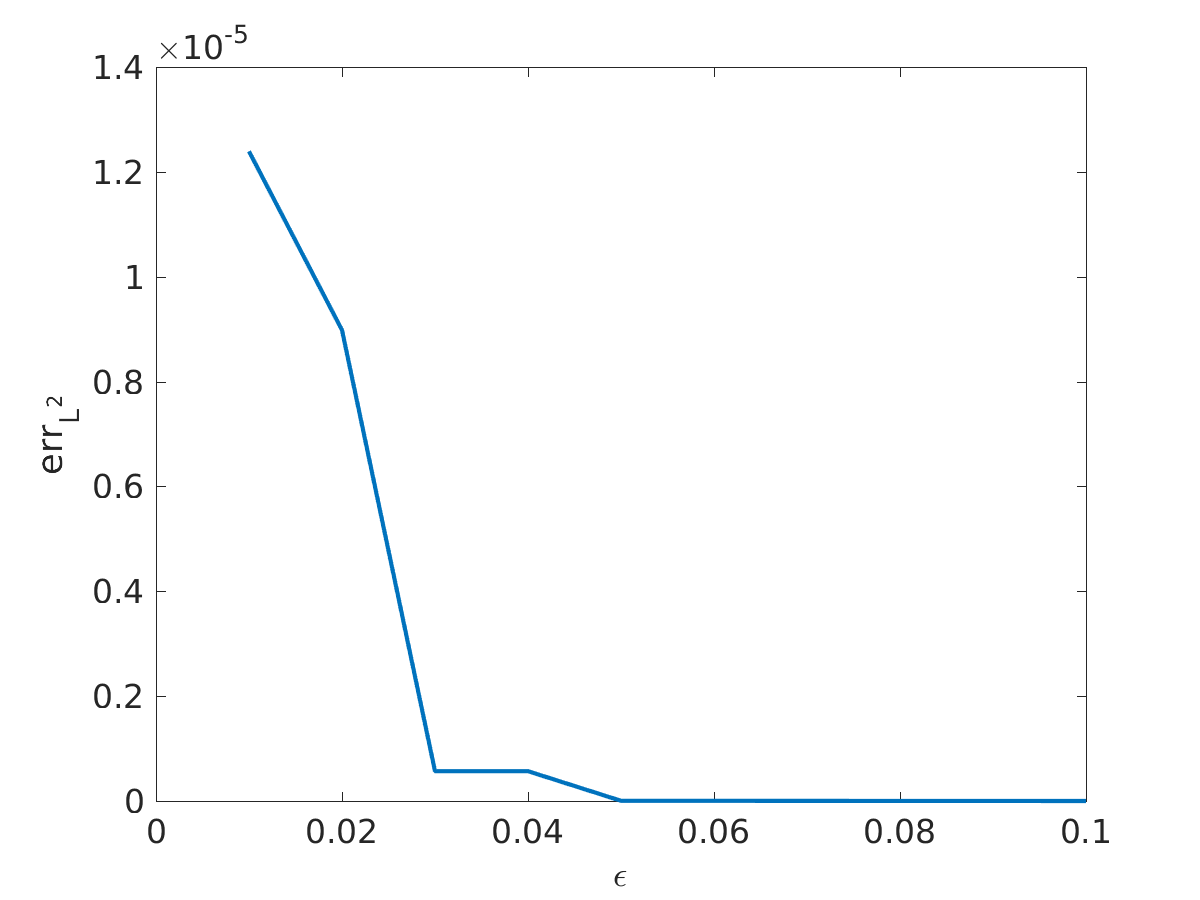}
  \end{center}
 \caption{(left) $r$ and $b$ solutions of the PDE with $\mathrm{err}_{L^2}<1e-12$. (Right) On the top: Value of $\eps$ versus number of iterations until error criterion is met, below: $\eps$ versus $\mathrm{err}_{L^2}$ at the fixed time $t=19$.}
 \label{fig:diffeps1dpde}
 \end{figure}
 As a next step, we examine the dynamical behaviour of solutions. To this end we consider a situation where initially, two compactly supported bumbs of both $r$ and $b$ are centered at $x=\pm 0.6$ in the domain, see Figure \ref{fig:dyninit}. All subsequent simulations are done for $\eps = 0.0002$ and $\tau = 0.0005$. First we examined the case $c_{11}=-2$ and $c_{22}=-0.5$ in which we observe unmixing of the two species first, before the two populations meet, see Figure \ref{fig:mixmeet}. In the case $c_{11}=-1$ and $c_{22}=-0.5$, we observe only a partial unmixing until the two densities meet, see Figure \ref{fig:mixmeet2}.
 \begin{figure}
 \begin{center}
  \includegraphics[width=0.4\textwidth]{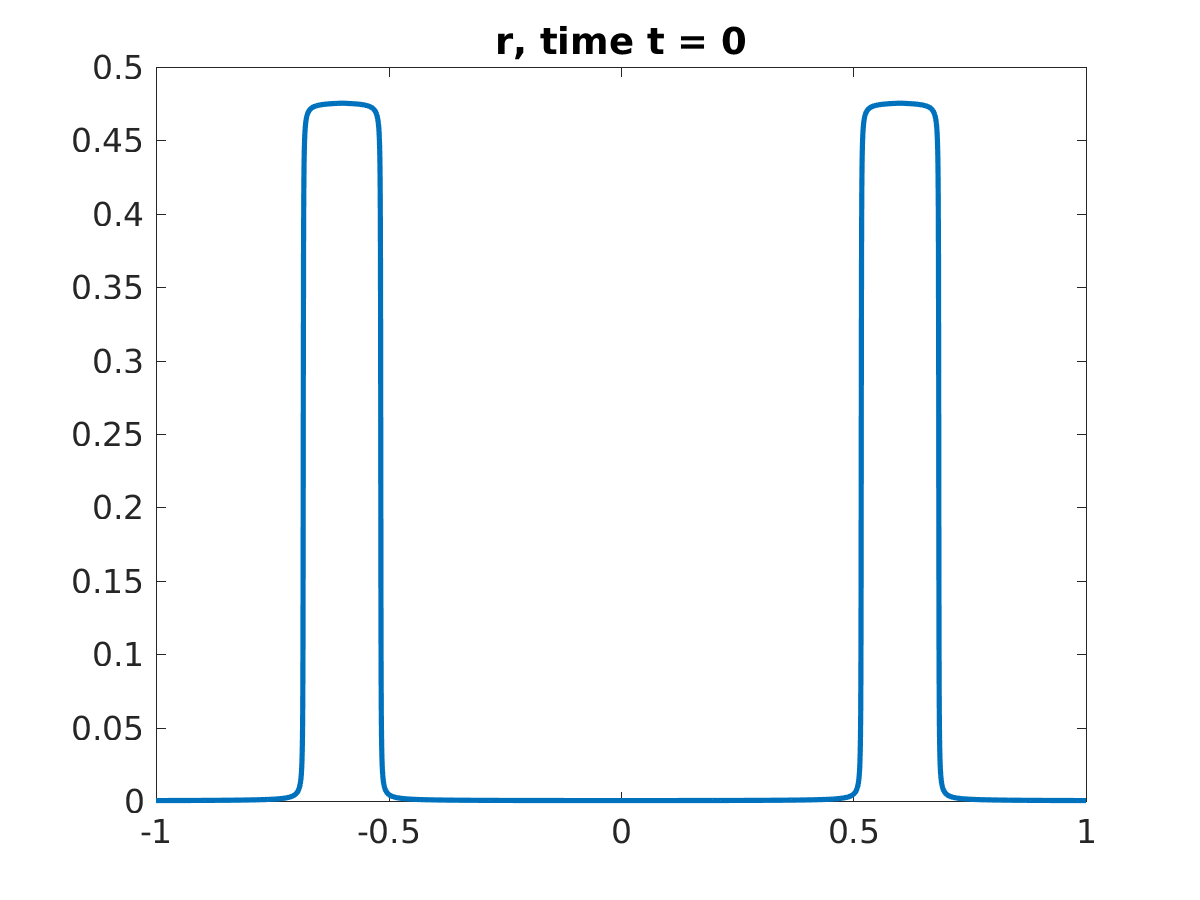}
  \includegraphics[width=0.4\textwidth]{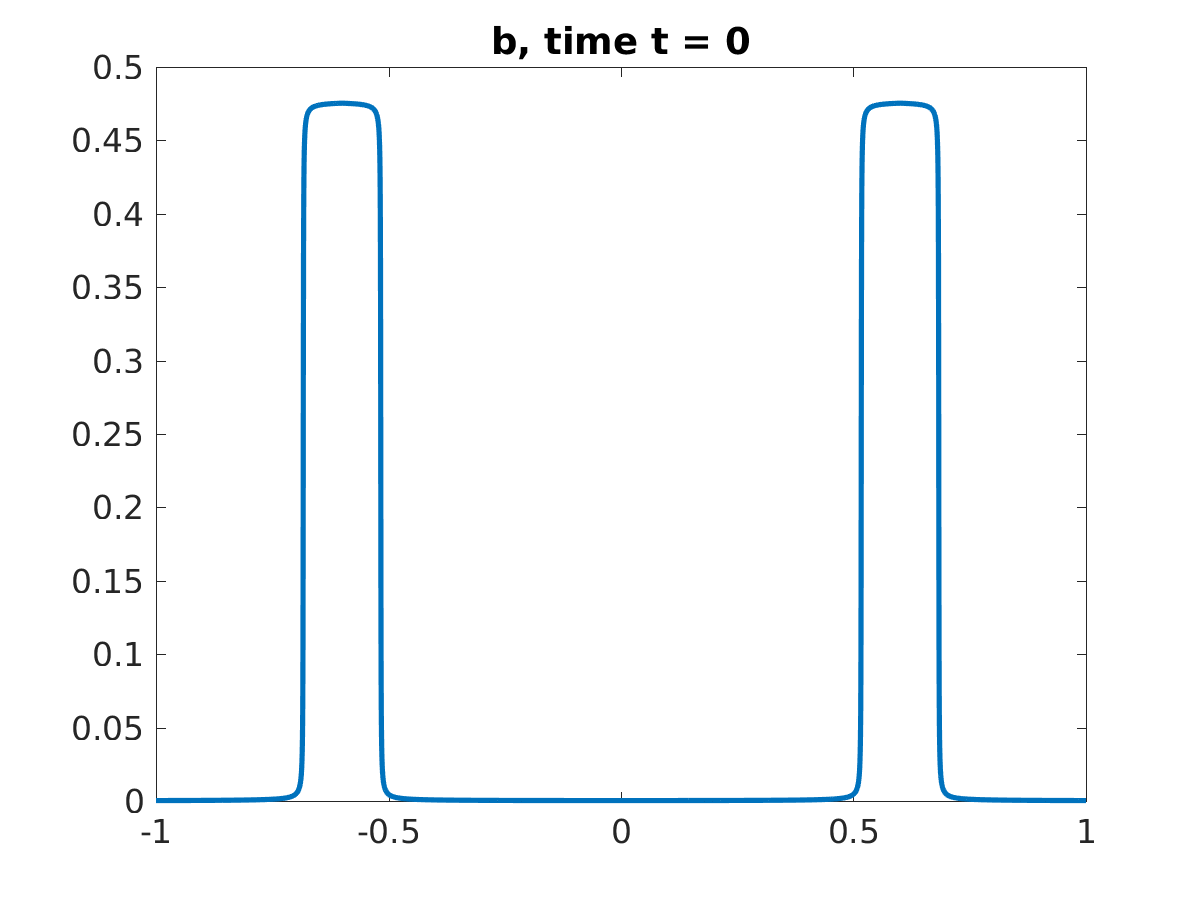}
  \end{center}
 \caption{Initial values for $r$ and $b$}
 \label{fig:dyninit}
 \end{figure} 
\begin{figure}
 \begin{center}
 \begin{subfigure}[t]{0.24\textwidth}
        \centering
        \includegraphics[width=\textwidth]{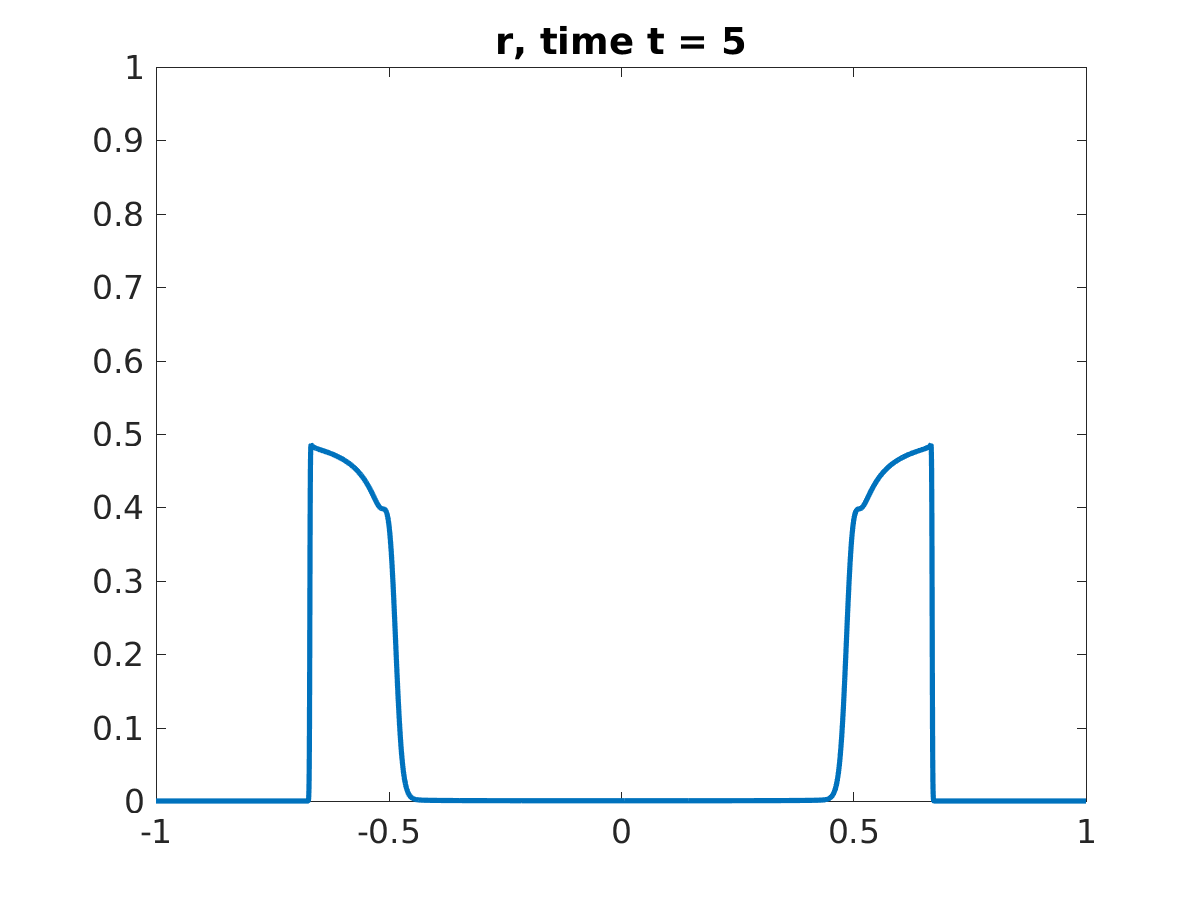}\\
        \includegraphics[width=\textwidth]{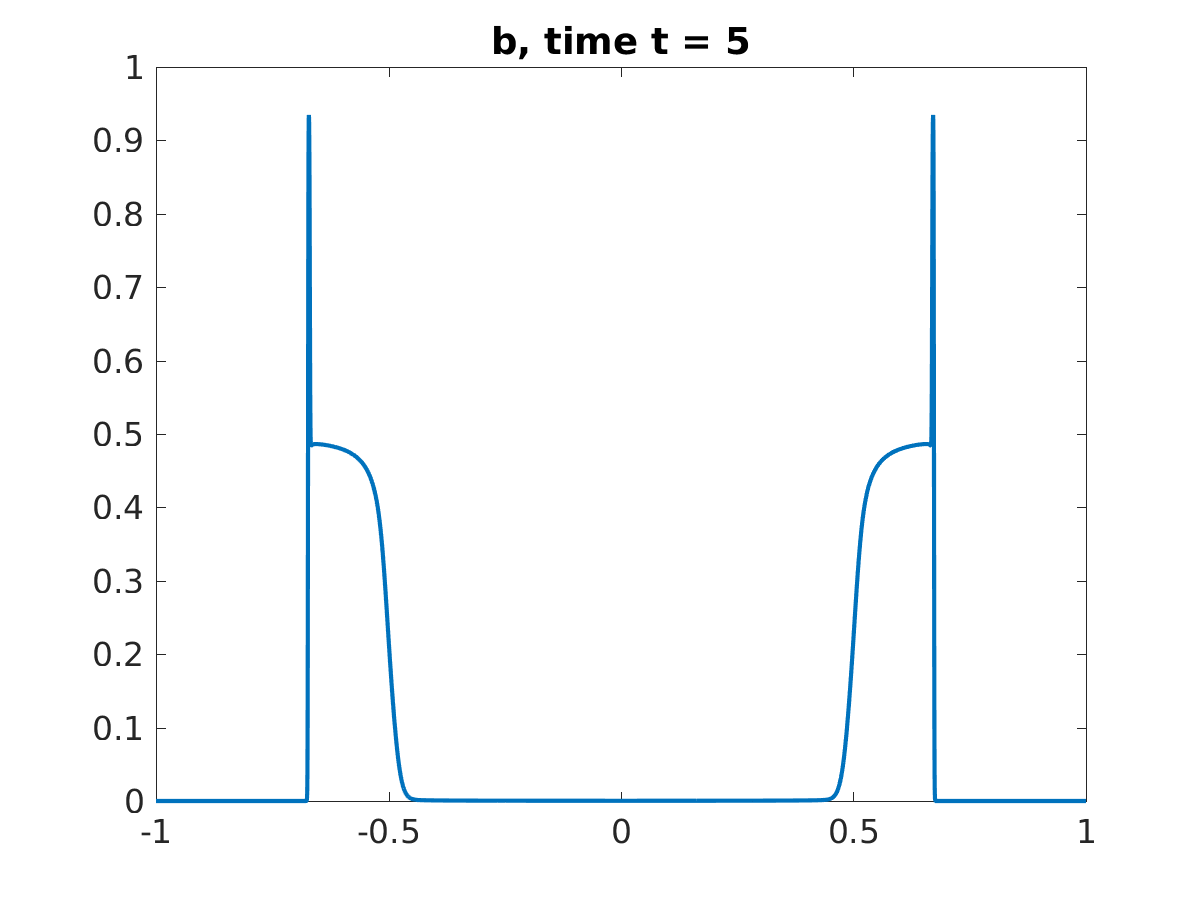}
        \caption{t = 5}
    \end{subfigure}%
    \begin{subfigure}[t]{0.24\textwidth}
        \centering
        \includegraphics[width=\textwidth]{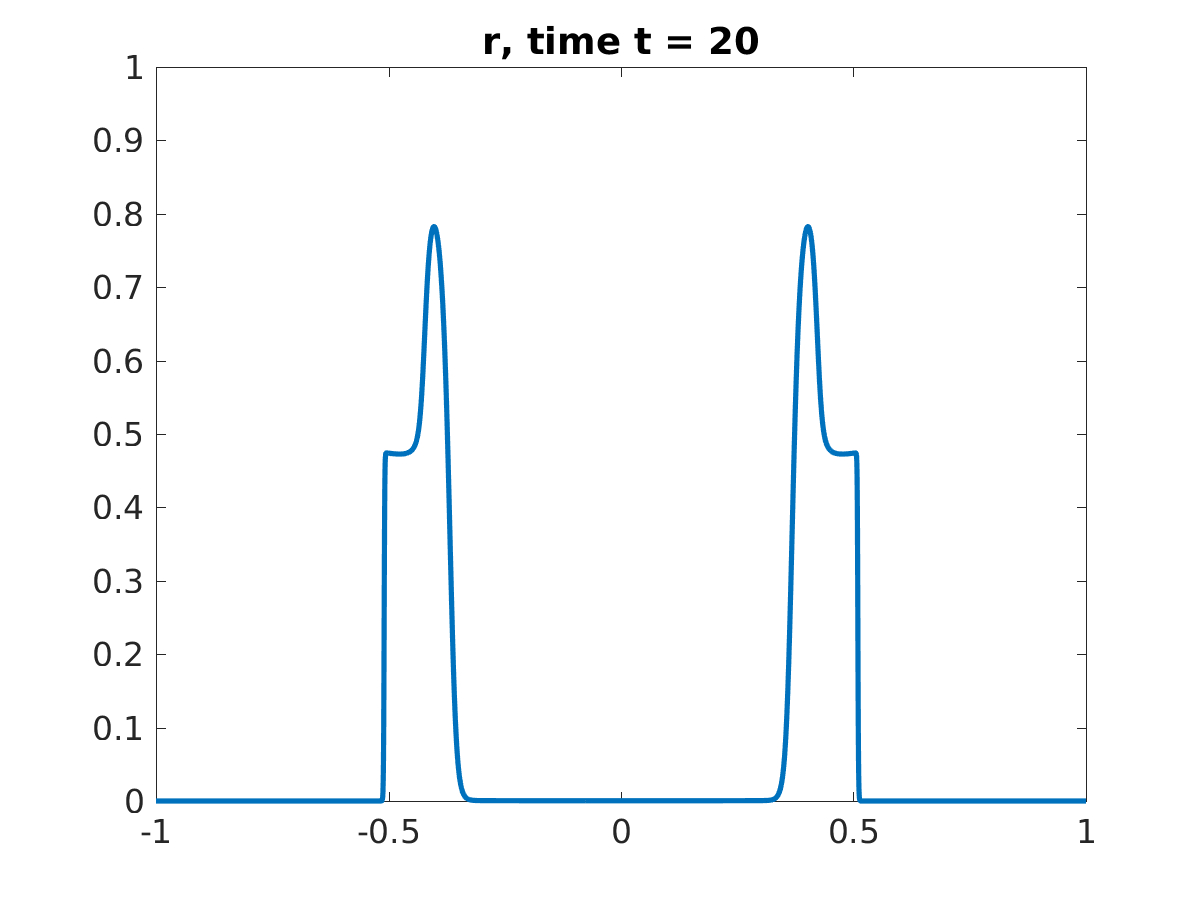}\\
        \includegraphics[width=\textwidth]{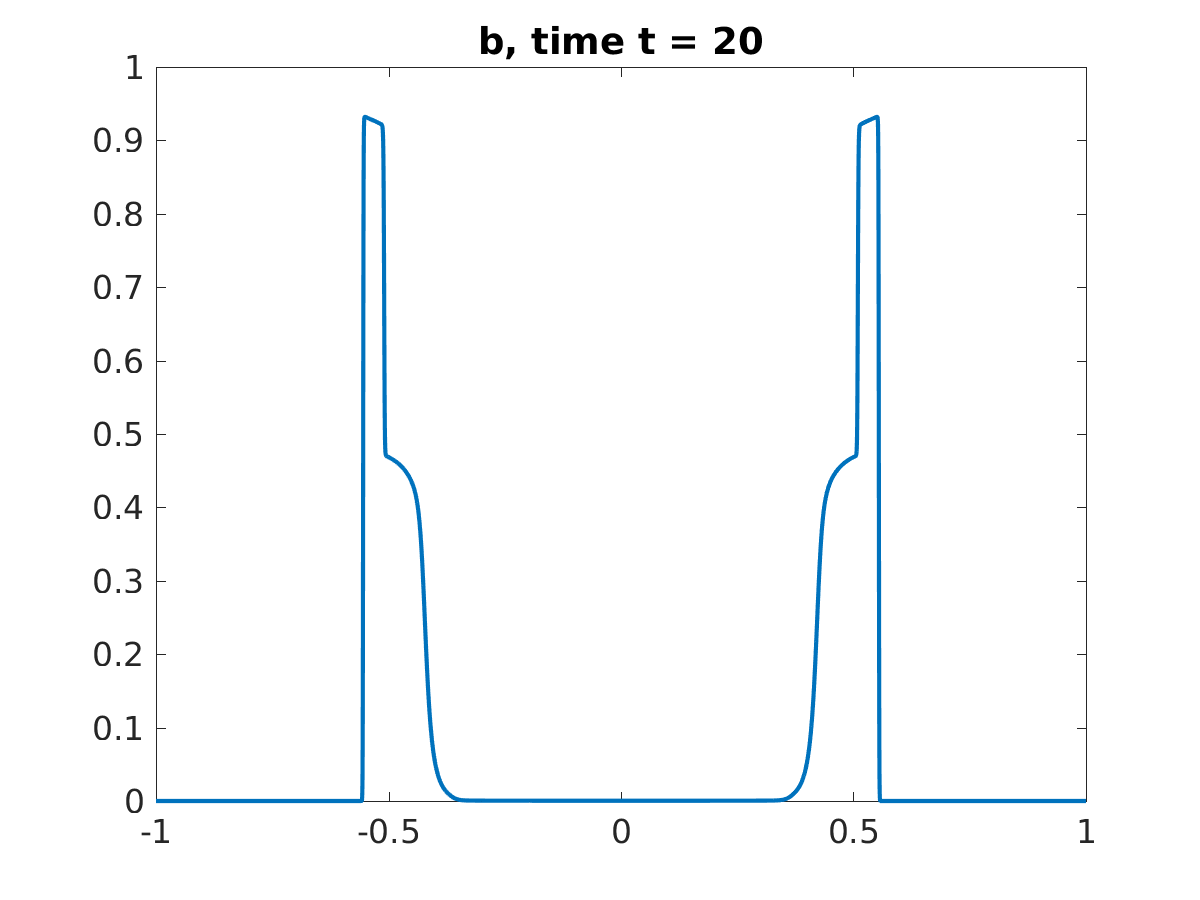}
        \caption{t = 20}
    \end{subfigure}%
    \begin{subfigure}[t]{0.24\textwidth}
        \centering
        \includegraphics[width=\textwidth]{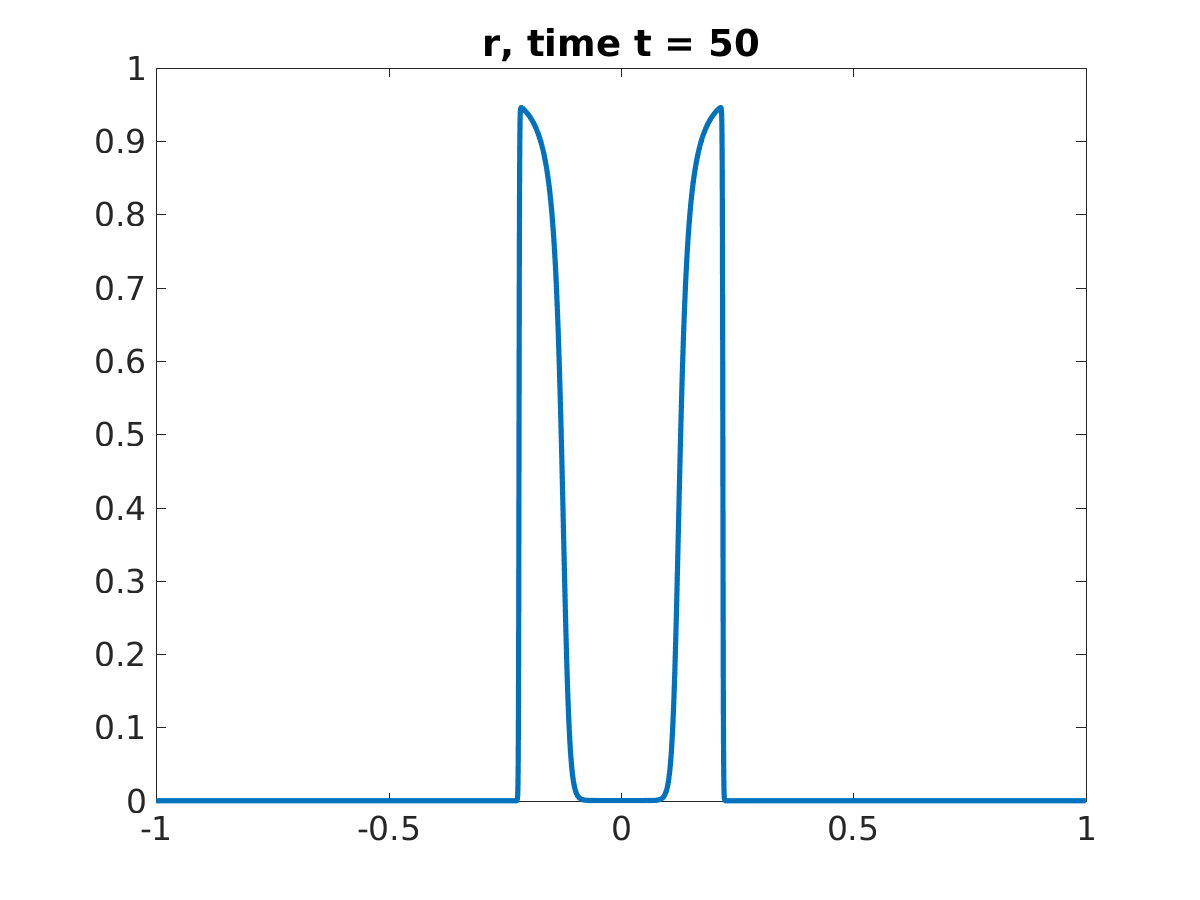}\\
        \includegraphics[width=\textwidth]{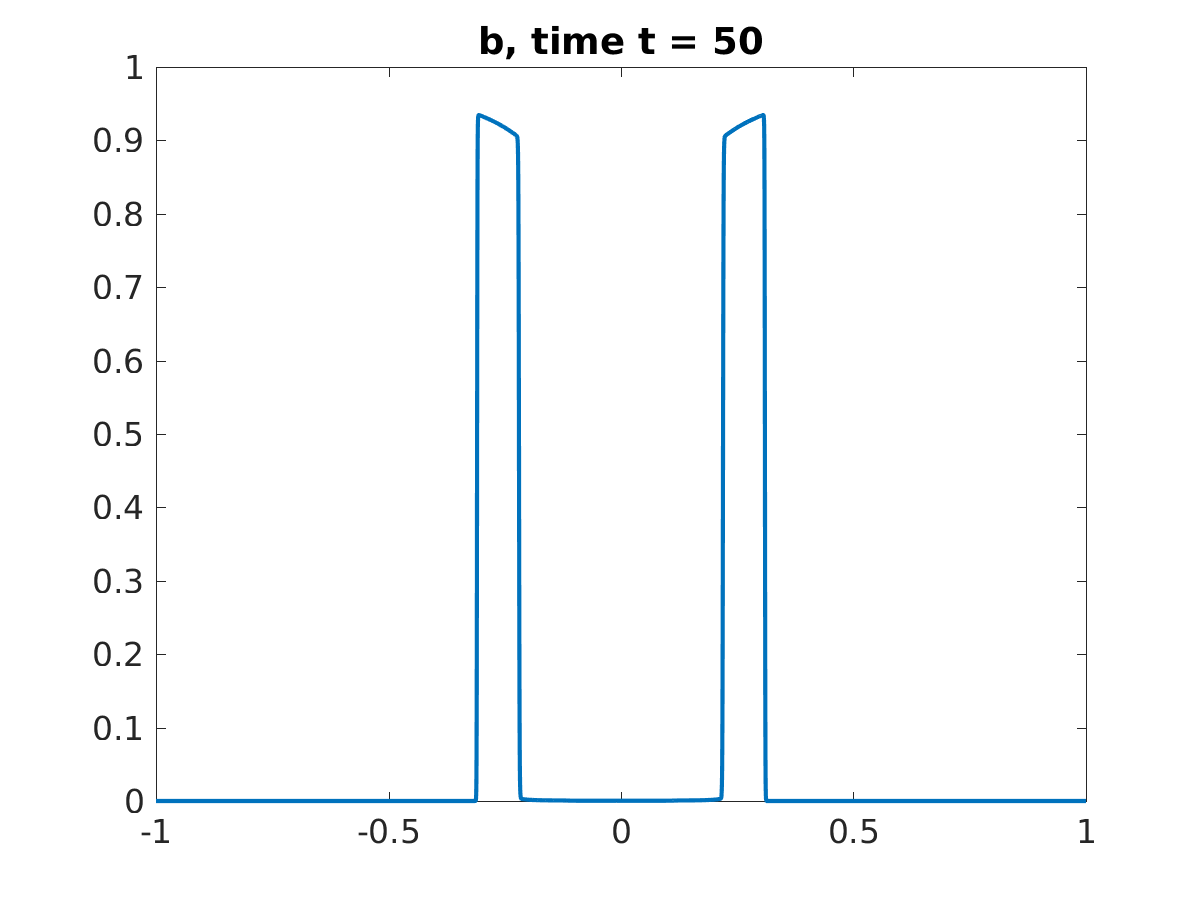}
        \caption{t = 50}
    \end{subfigure}%
    \begin{subfigure}[t]{0.24\textwidth}
        \centering
        \includegraphics[width=\textwidth]{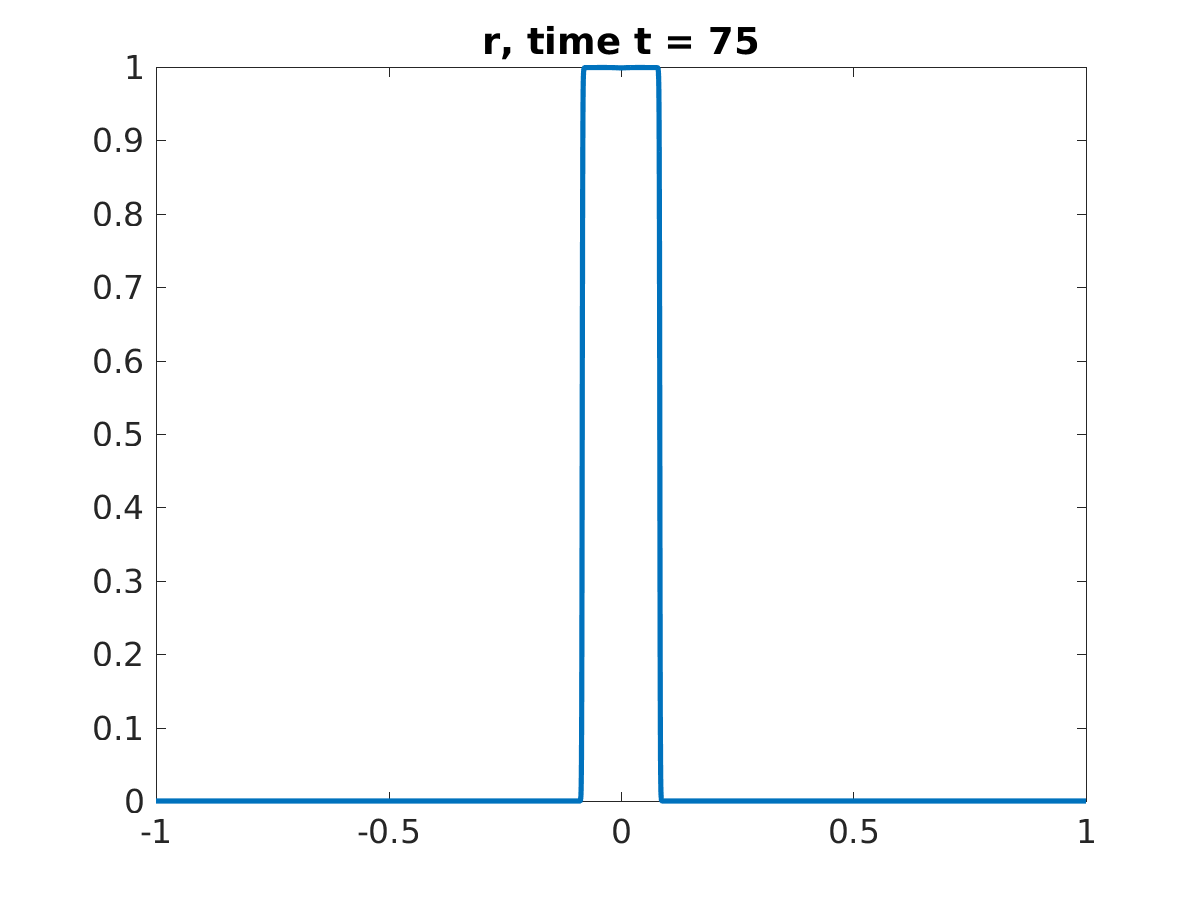}\\
        \includegraphics[width=\textwidth]{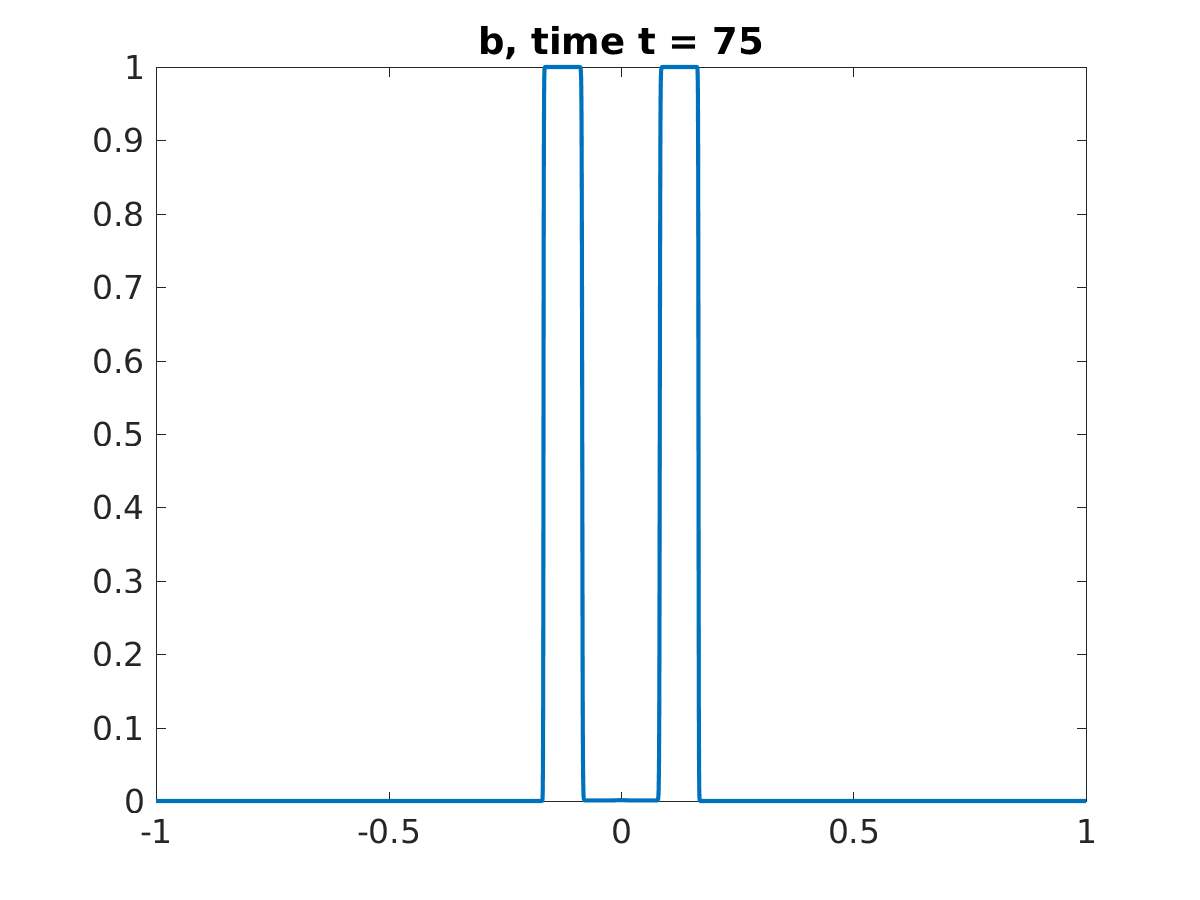}
        \caption{t = 75}
    \end{subfigure}%
 \end{center}
 \caption{The case $c_{11}=-2$ and $c_{22}=-0.5$. A video can be found at \cite{VidURL}}
 \label{fig:mixmeet}
 \end{figure} 
 \begin{figure}
  \begin{center}
 \begin{subfigure}[t]{0.24\textwidth}
        \centering
        \includegraphics[width=\textwidth]{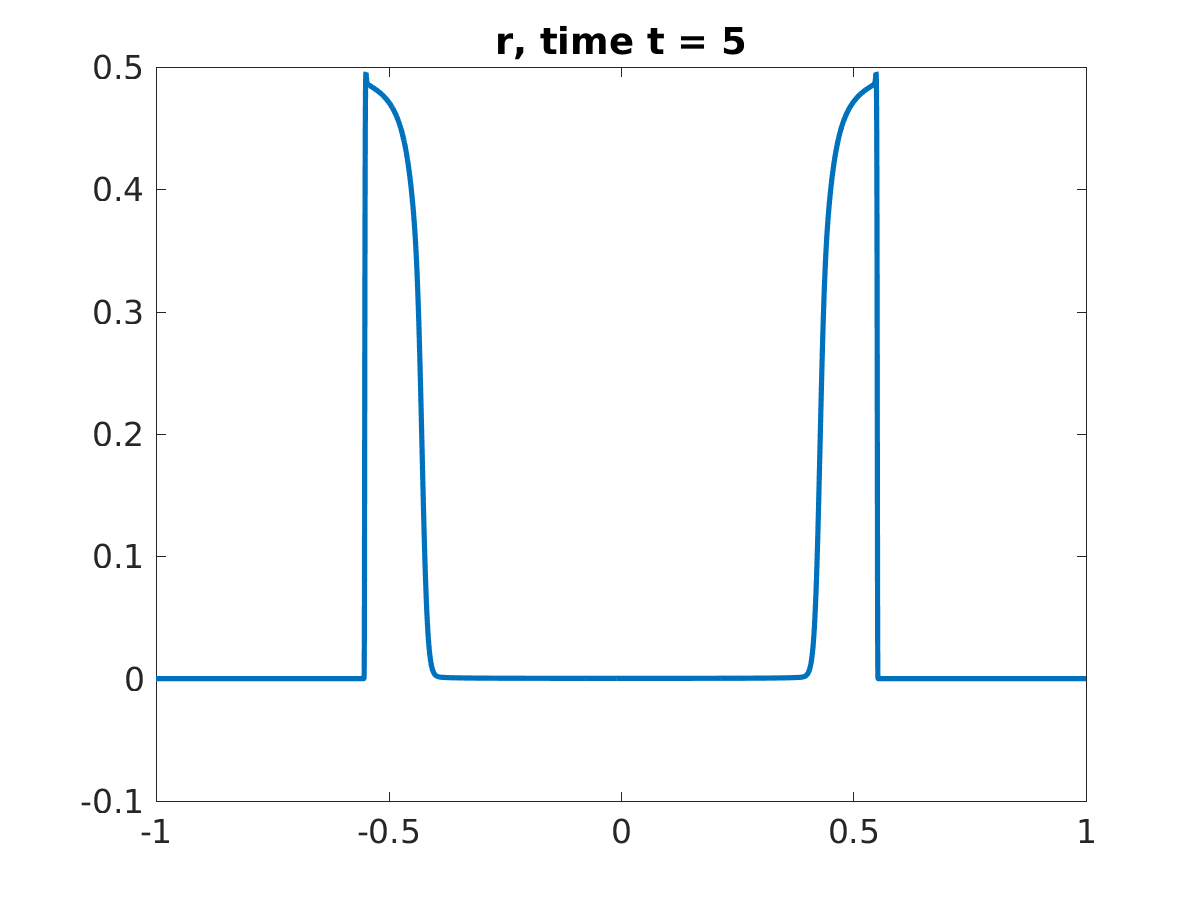}\\
        \includegraphics[width=\textwidth]{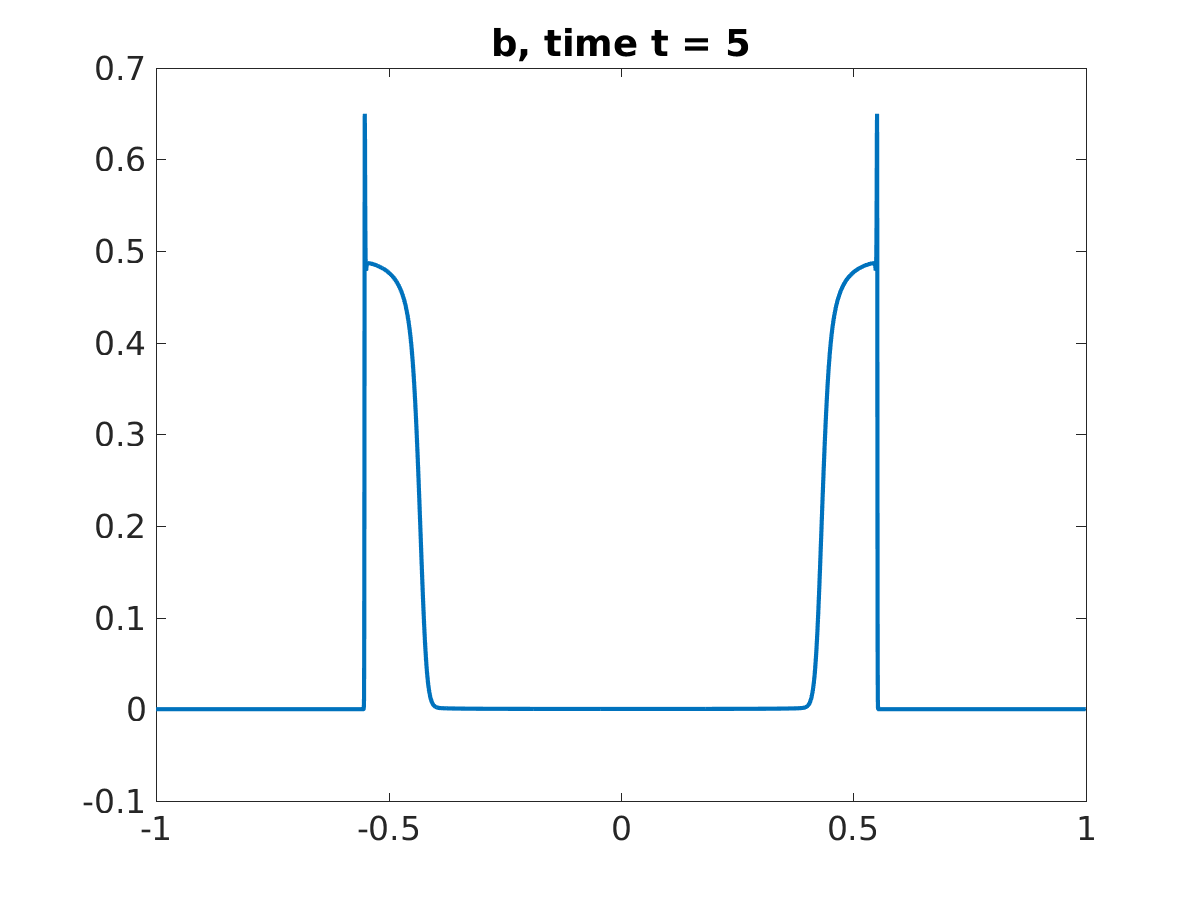}
        \caption{t = 5}
    \end{subfigure}%
    \begin{subfigure}[t]{0.24\textwidth}
        \centering
        \includegraphics[width=\textwidth]{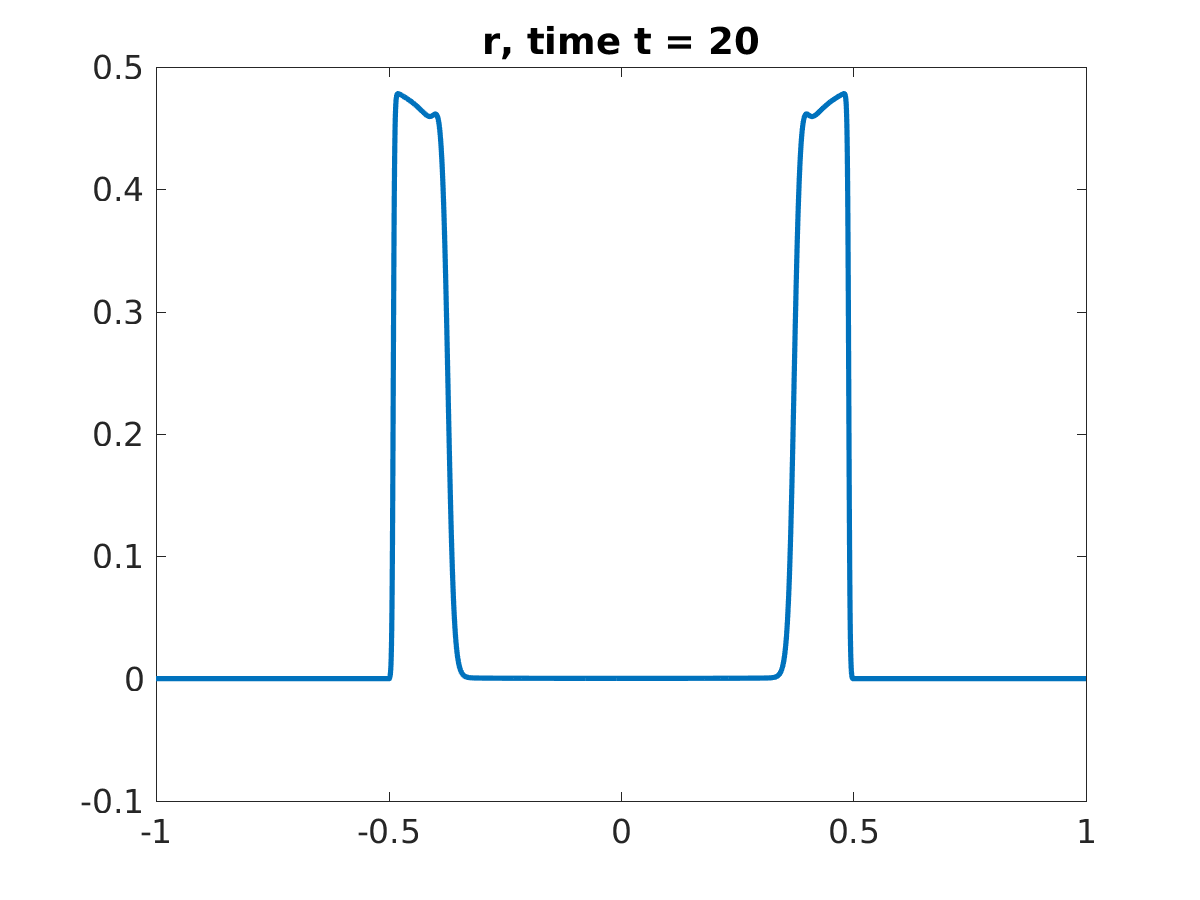}\\
        \includegraphics[width=\textwidth]{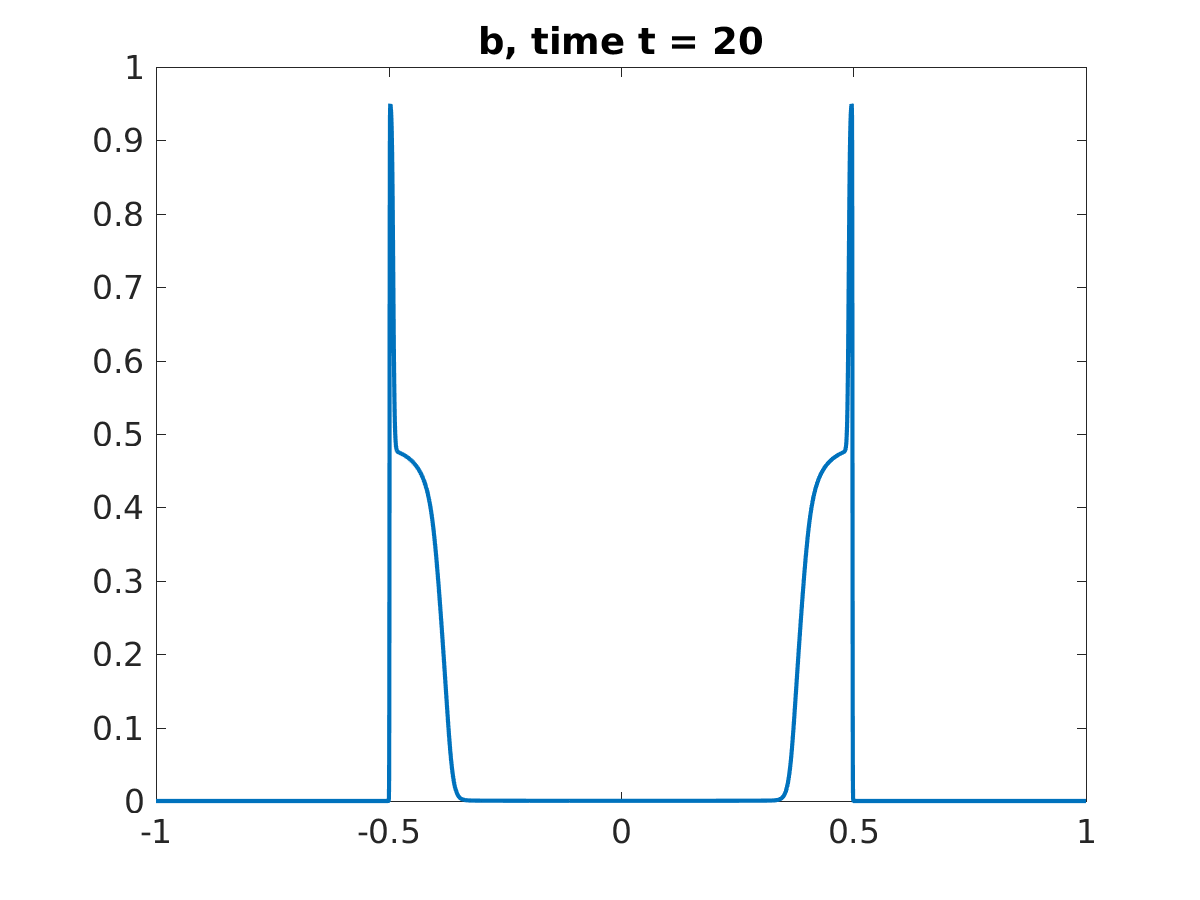}
        \caption{t = 20}
    \end{subfigure}%
    \begin{subfigure}[t]{0.24\textwidth}
        \centering
        \includegraphics[width=\textwidth]{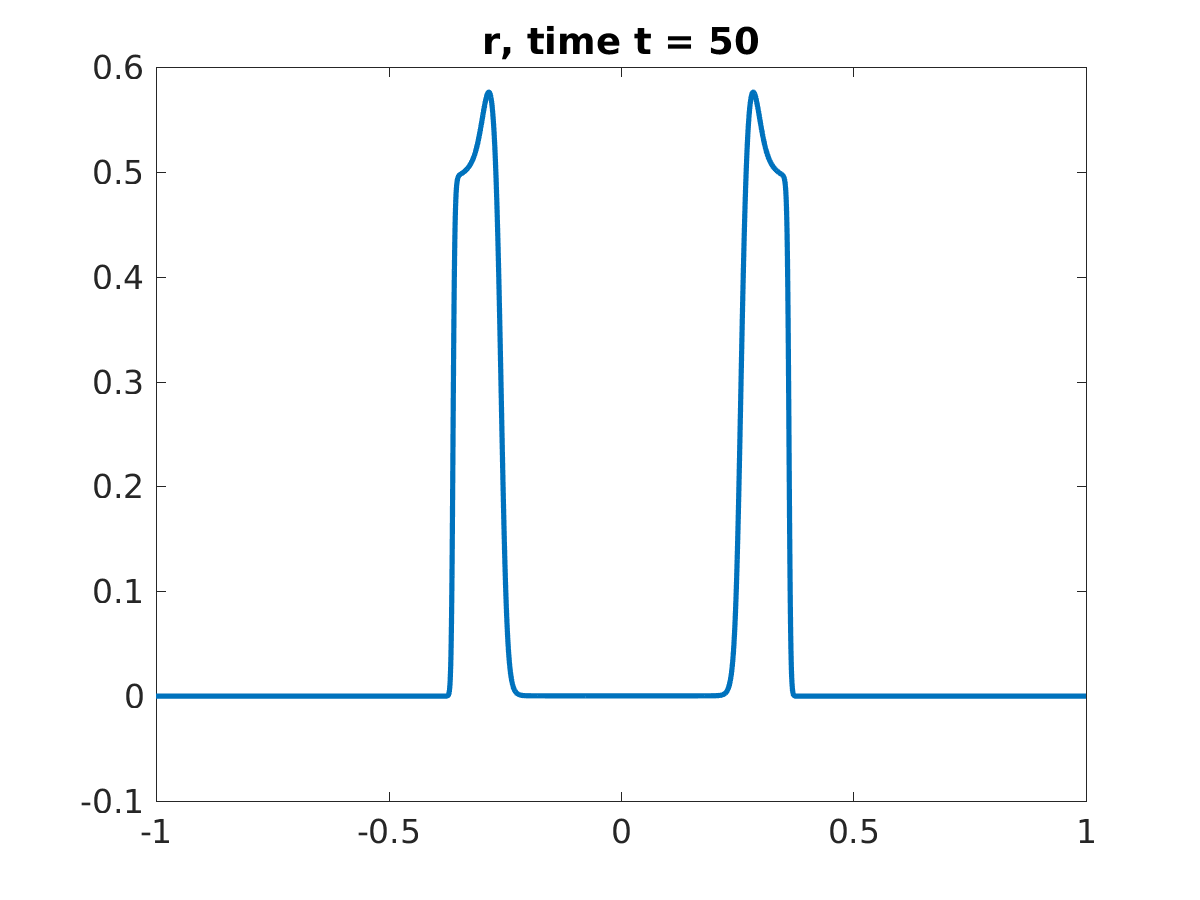}\\
        \includegraphics[width=\textwidth]{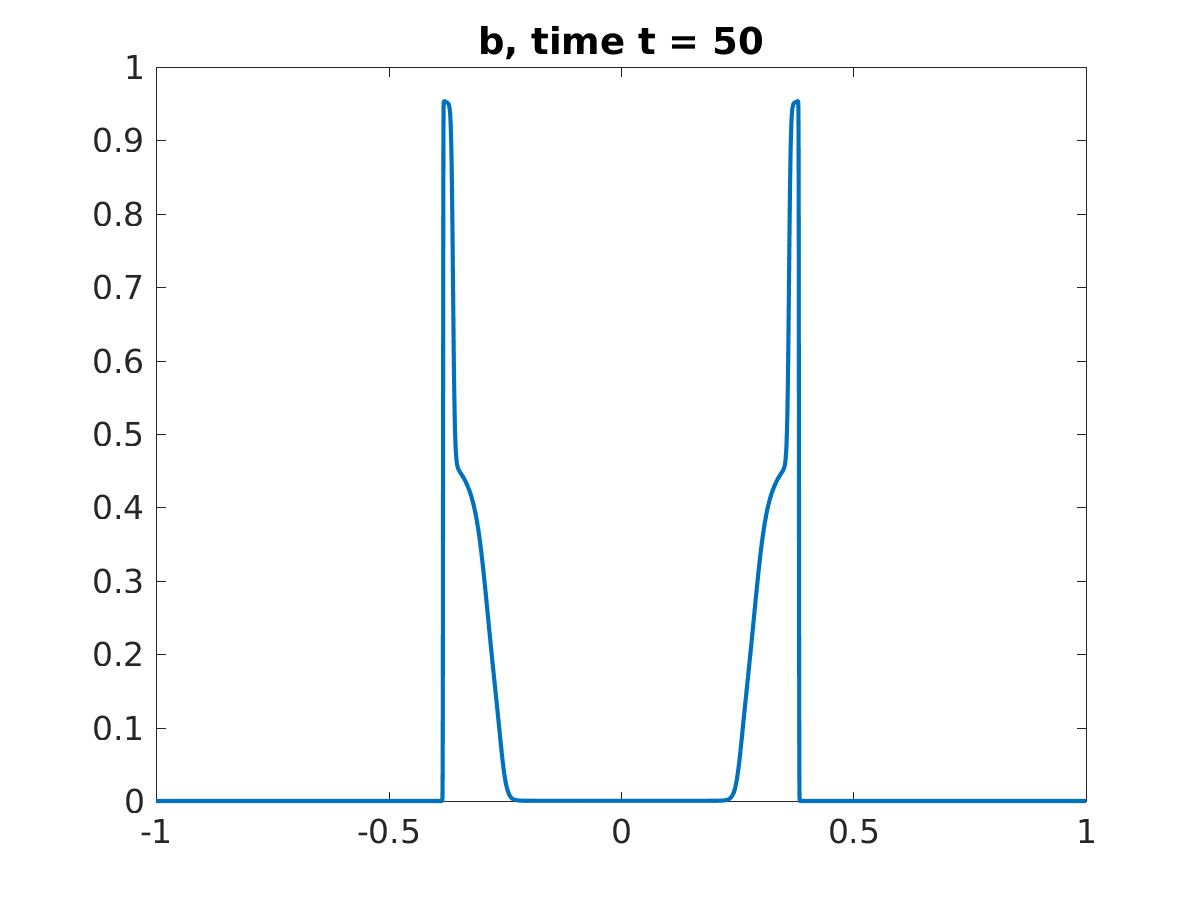}
        \caption{t = 50}
    \end{subfigure}%
    \begin{subfigure}[t]{0.24\textwidth}
        \centering
        \includegraphics[width=\textwidth]{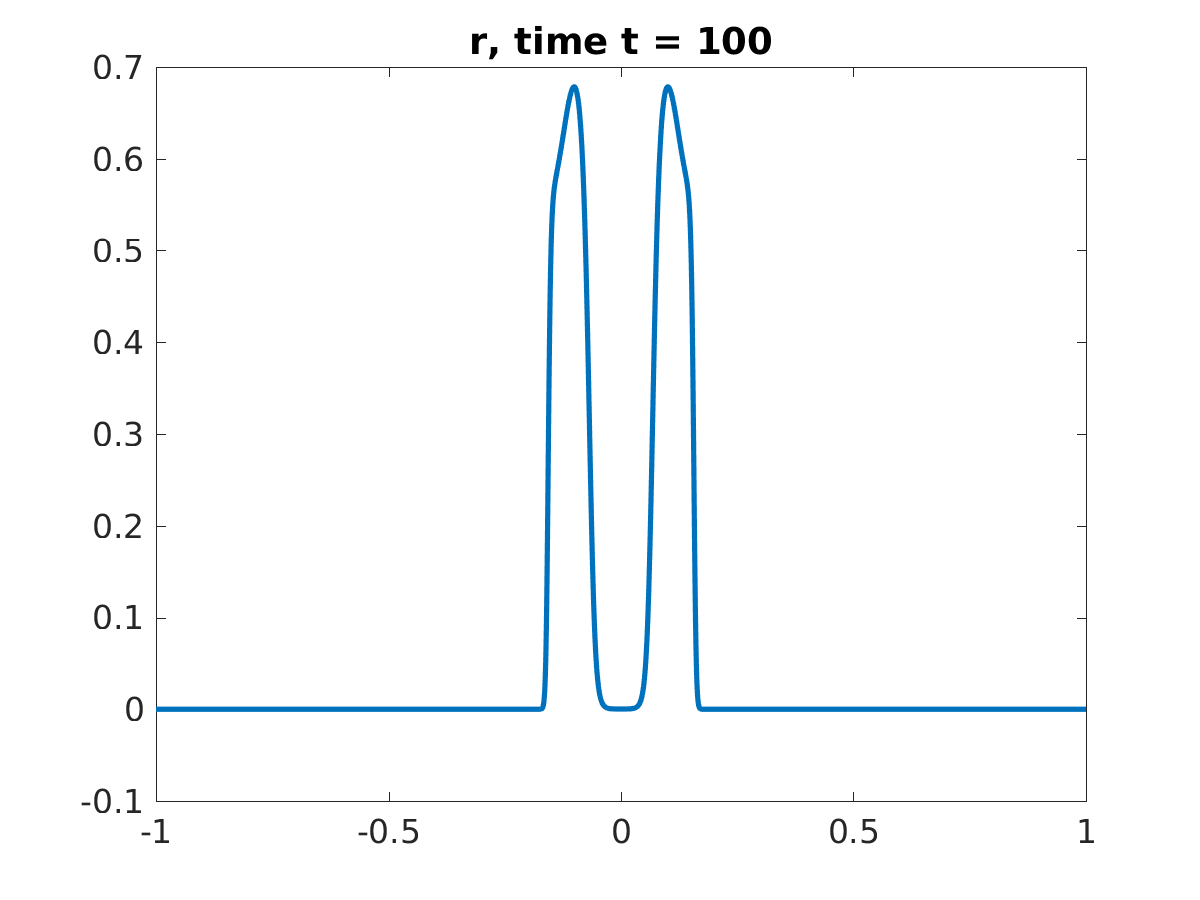}\\
        \includegraphics[width=\textwidth]{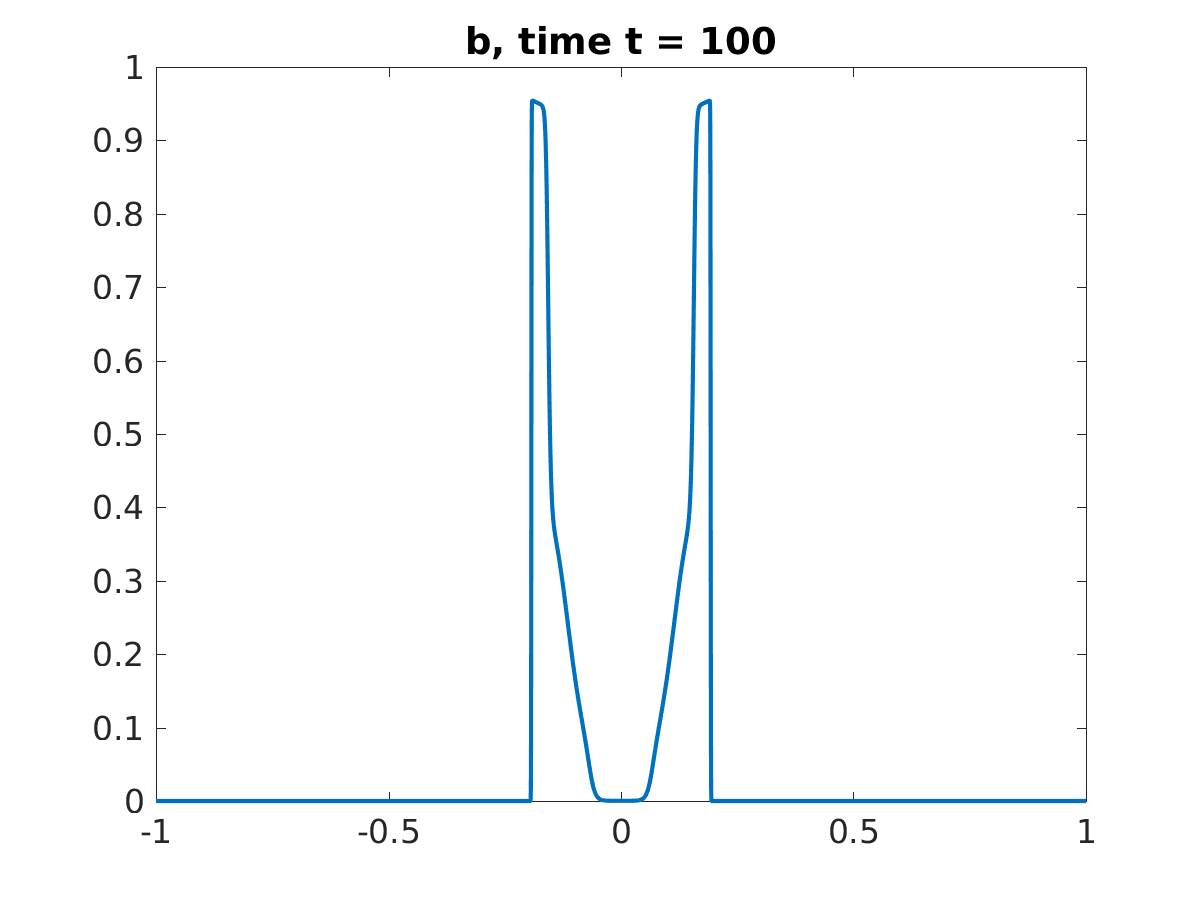}
        \caption{t = 100}
    \end{subfigure}%
 \end{center}
 \caption{The case $c_{11}=-1$ and $c_{22}=-0.5$. A video can be found at \cite{VidURL}}
 \label{fig:mixmeet2}
 \end{figure} 


%

\subsection{Examples in two spatial dimensions}
In two dimensions, we only present one example showing that for large $t$, the solutions of the system seem in fact to converge to the minimizers obtained by the ADMM scheme. To this end, we chose again the case $c_{11}=-1$ and $c_{22}=-1/2$ and $\eps = 0.02$ and used the solution of the PDE at time $t=650$. The results are shown in figure \ref{fig:pdeadmm2d}. Note that our present results only ensure the existence of solution for arbitrary large yet finite time. However, we expect that these can be extended to the case $t\to\infty$ and that we can in fact show convergence to stationary solutions by means of relative entropy methods.
 \begin{figure}
 \begin{center}
  \begin{subfigure}[t]{0.33\textwidth}
        \centering
        \includegraphics[width=\textwidth]{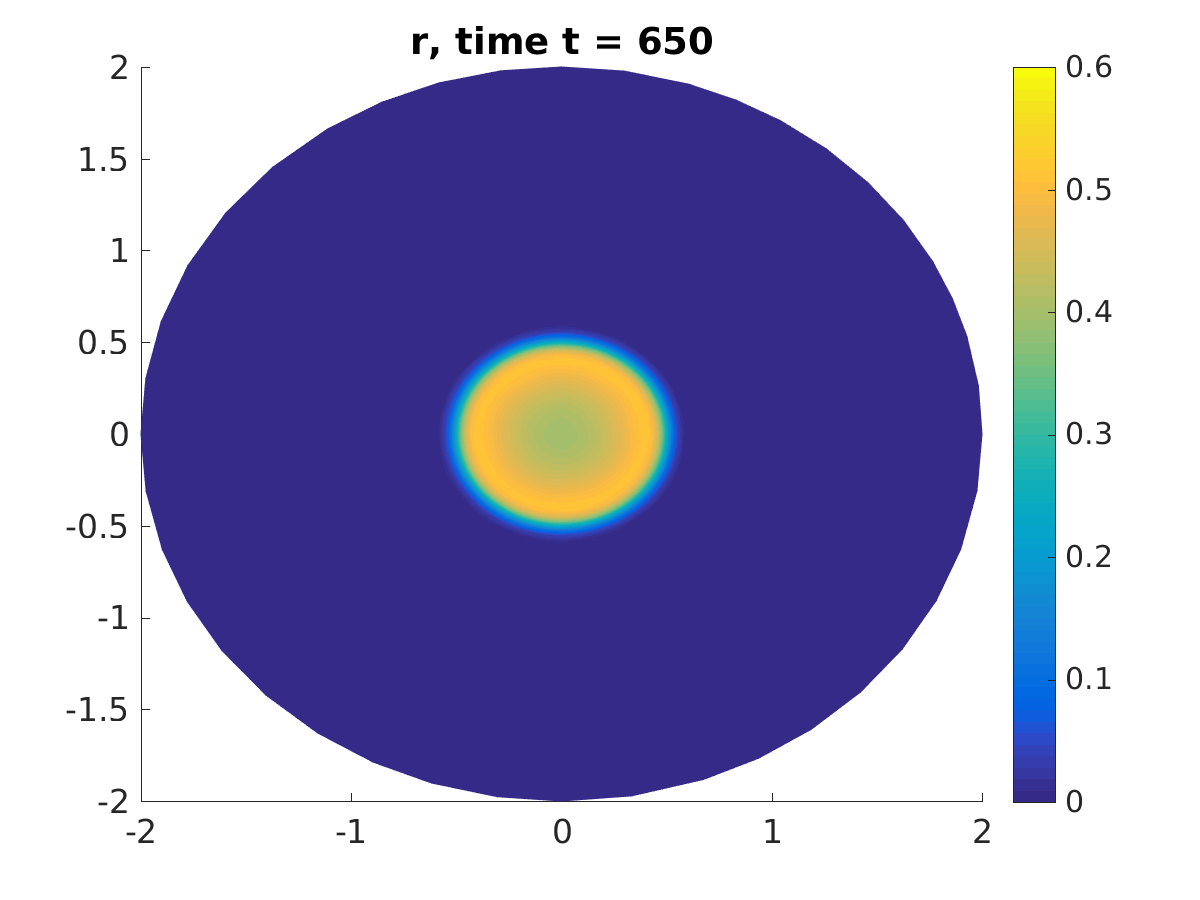}\\
        \includegraphics[width=\textwidth]{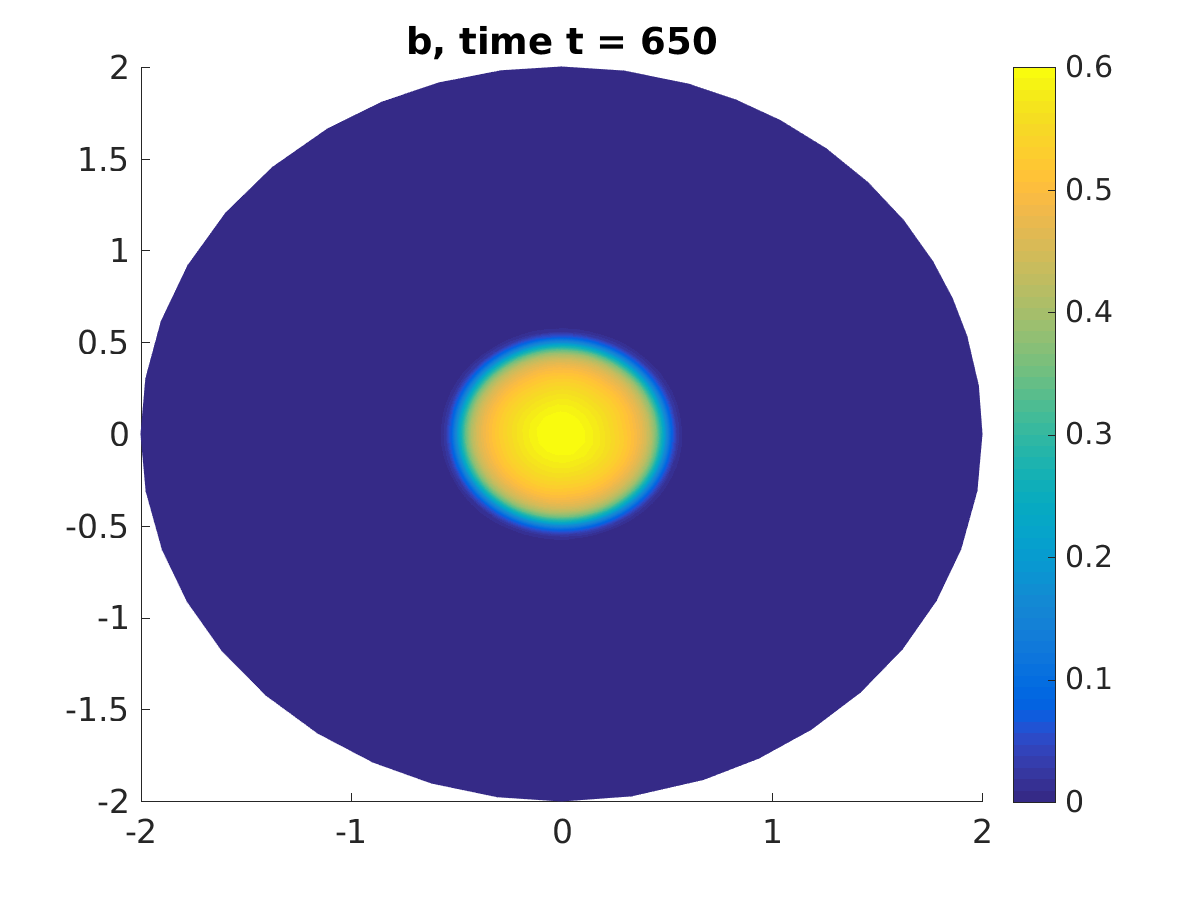}
        \caption{PDE, t = 650}
    \end{subfigure}%
    \begin{subfigure}[t]{0.33\textwidth}
        \centering
        \includegraphics[width=\textwidth]{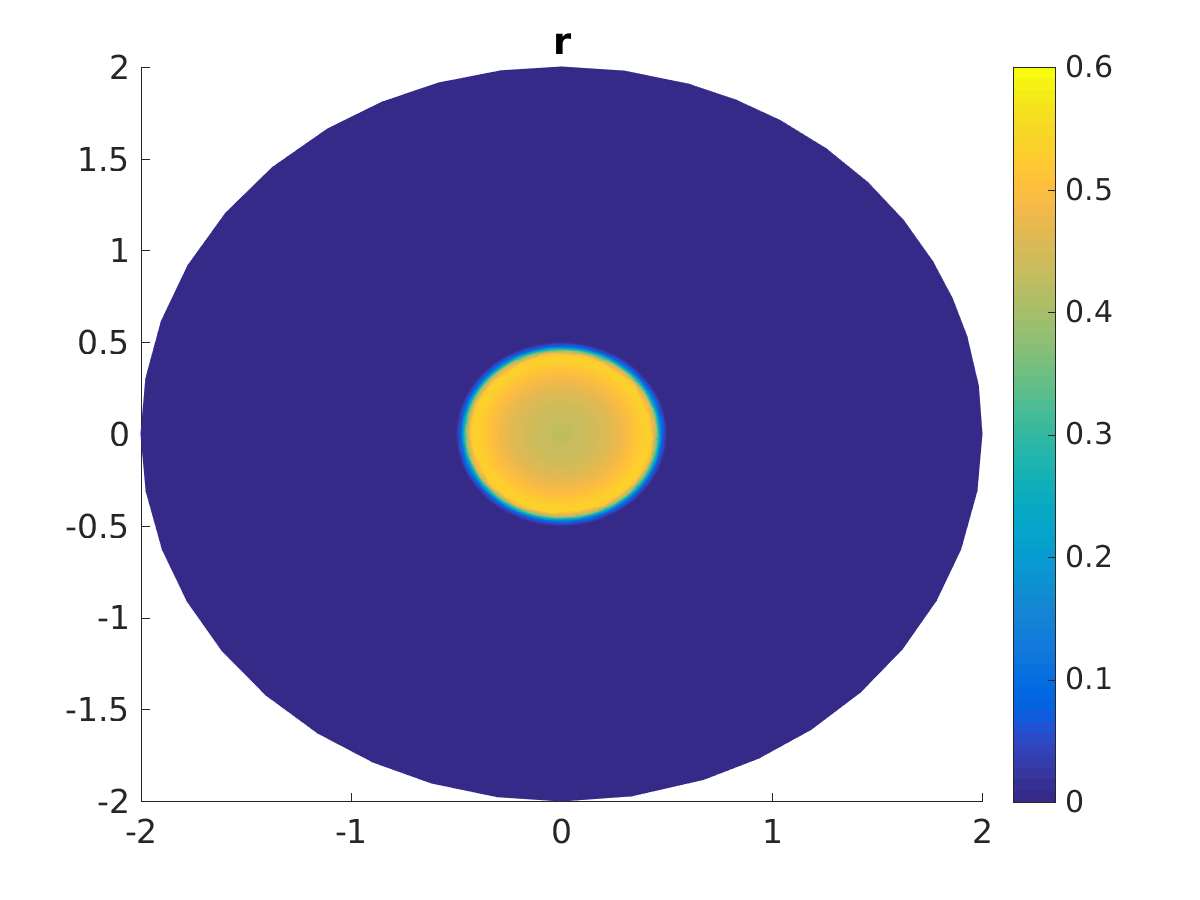}\\
        \includegraphics[width=\textwidth]{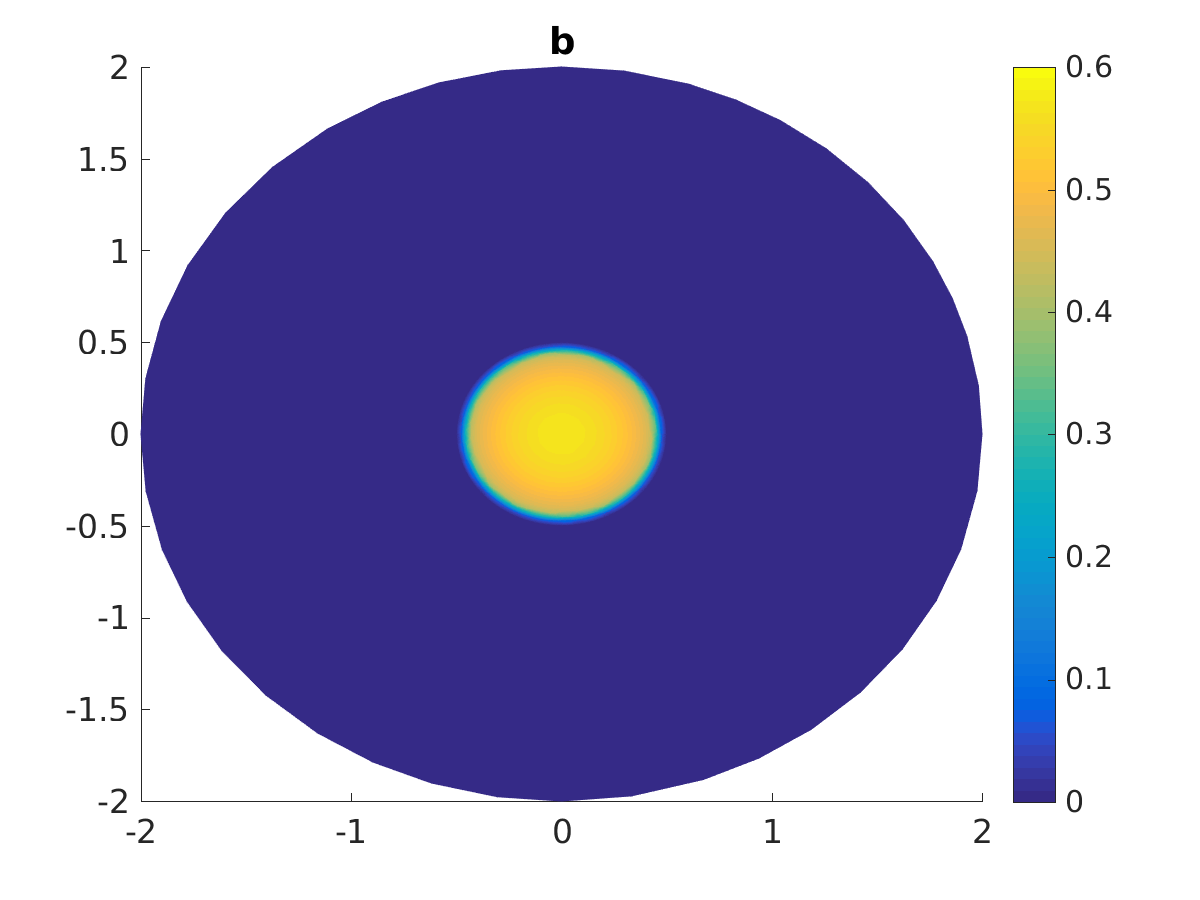}
        \caption{ADMM}
    \end{subfigure}%
    \begin{subfigure}[t]{0.32\textwidth}
        \centering
        \includegraphics[width=\textwidth]{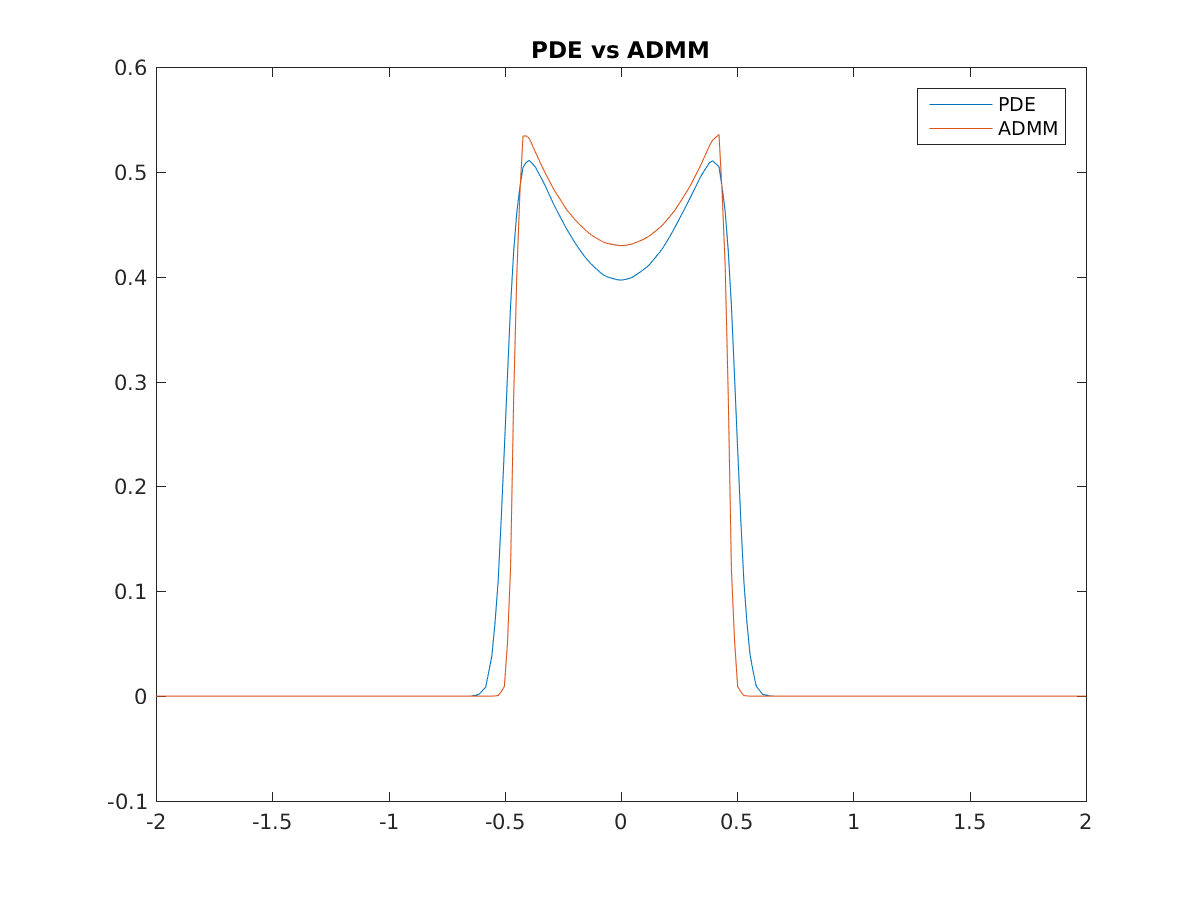}\\
        \includegraphics[width=\textwidth]{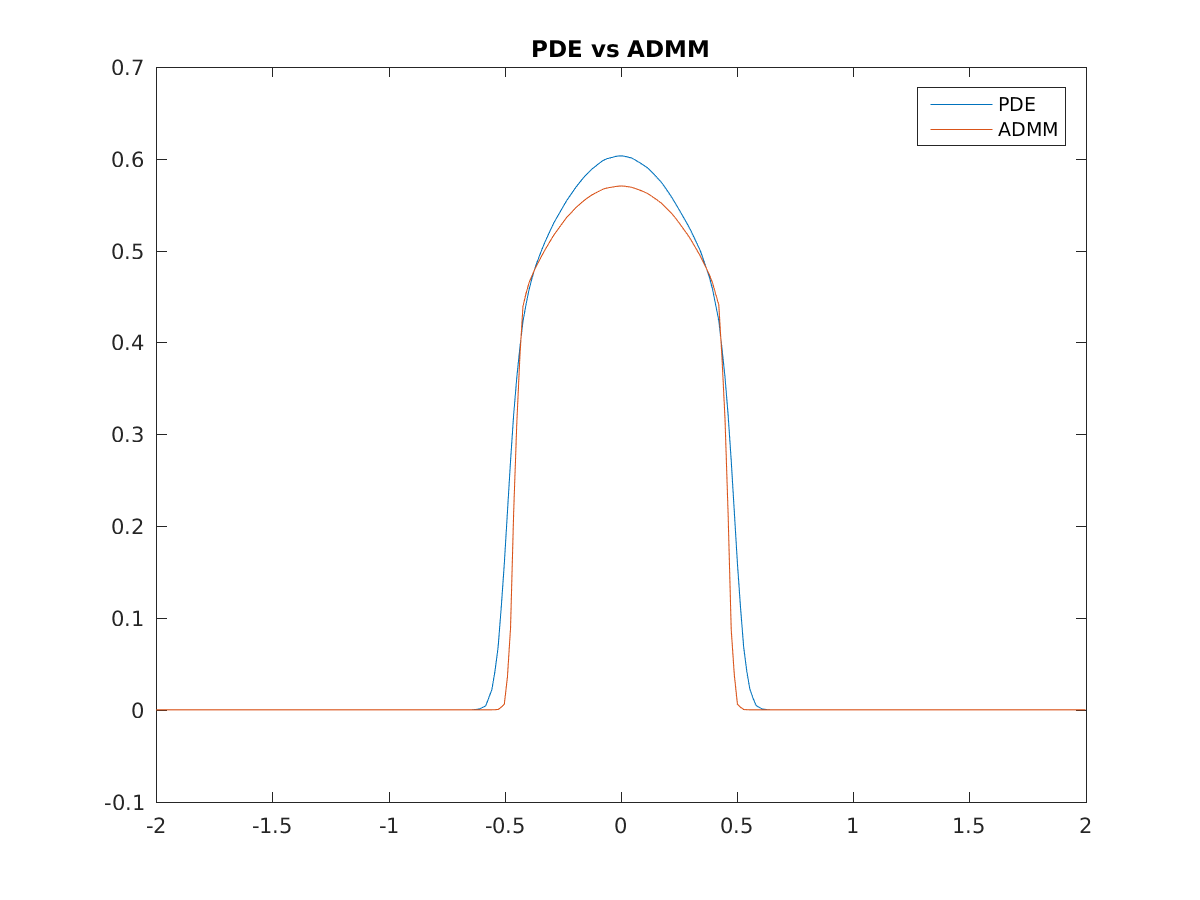}
        \caption{Cut at $y=0$}
    \end{subfigure}%
 \end{center}
 \caption{Comparison between solutions of the PDE at time $t=650$ and minimizers generated by the ADMM scheme for $\eps = 0.02$, $c_{11}=-1$ and $c_{22}=-1/2$. On the right a cut through the axis $y=0$ is shown.}
 \label{fig:pdeadmm2d}
 \end{figure}

\section{Outlook: Coarsening Dynamics}\label{sec:coarsening}
In the following we further investigate the coarsening dynamics of the system \eqref{eq:crossdiffusion1}, \eqref{eq:crossdiffusion2} as $\eps \rightarrow 0$. Let us first of all discuss the case $\eps =0$, which is characterized by a very large set of stationary solutions. Indeed, every pair $(\overline{r},\overline{b})$ with 
\begin{equation}
	\overline{r} (1- \overline{r} - \overline{b}) = \overline{b} (1- \overline{r} - \overline{b}) = 0 
\end{equation}
is a stationary solution, but not necessarily stable under the dynamics. In the case for a single species, stability was characterized from an entropy condition in \cite{burger2008asymptotic}, noticing that the model for $\eps = 0$ consists of nonlinear hyperbolic conservation laws with nonlocal terms. A generalization of the entropy condition to the system case seems out of reach, hence we revisit the single species case ($b=0$ in our notation) and rederive the stability condition from an optimization argument. Intuitively a stationary state is stable if it is a local minimizer of the energy, and using a Lagrange multiplier $\lambda_r$ for the mass constraint we easily arrive at the conditions
\begin{align*}
c_{11} K*r(x) + \lambda_r \geq 0 & \quad \text{if } r(x)=0 \\
c_{11} K*r(x) + \lambda_r \leq 0 & \quad \text{if } r(x)=1. 
\end{align*}
Thus, at an interface between a region with $r=1$ inside and $r=0$ outside the normal derivative of $S=c_{11} K*r$ is nonpositive, which is exactly the entropy condition derived in \cite{dolak2005keller,burger2008asymptotic}. 

In the system case we find by analogous arguments with Lagrange parameters $\lambda_r$ and $\lambda_b$ for the masses 
\begin{align*}
c_{11} K*r(x) - K*b + \lambda_r \geq 0 & \quad \text{if } r(x)=0, \\
c_{22} K*b(x) - K*r + \lambda_b \geq 0 & \quad \text{if } b(x)=0, \\
c_{11} K*r(x) - K*b + \lambda_r \leq 0 & \quad \text{if } \rho(x)=1, \\
c_{22} K*b(x) - K*r + \lambda_b \leq 0 & \quad \text{if } \rho(x)=1.  
\end{align*}
Hence we can effectively derive conditions on the sign of the normal derivatives of 
\begin{align*}
S_r = c_{11}K*r - K*b \text{ and } S_b=c_{22}K*b - K*r
\end{align*}
 on interfaces between regions with $r=1$, $b=1$, or $\rho=0$ (voids). On interfaces between red or blue and void the same analysis as in \cite{burger2008asymptotic} applies, on interfaces between red and blue sign conditions on the derivatives of $S_b$ and $S_r$ need to hold simultaneously.  
Such stable configurations will lead to similar stationary solutions or metastable dynamics in the case of small positive $\eps$.

We finally mention that for the coarsening of interfaces in the metastable case analogous laws as in \cite{dolak2005keller,burger2008asymptotic} can be derived by multiscale asymptotic expansions, whose details we leave to future research. In numerical experiments the coarsening dynamics is well observed, see Figure \ref{fig:mixmeet}, where one obtains local unmixing followed by coarsening. In the case of small self-attraction as shown in Figure \ref{fig:mixmeet2} a different kind of coarsening dynamics is obtained, which is not characterized by $r$ and $b$ attaining values close to $1$, but a mixed phase is coarsening versus the void regions. In this case a further analysis is quite open and an interesting issue for the future.

\section*{Acknowledgements}
JB and MB acknowledge support by ERC via Grant EU FP 7 - ERC Consolidator Grant 615216 LifeInverse. The work of JFP was supported by DFG via Grant 1073/1-2. The authors thank Alessio Brancolini, Benedikt Wirth and Caterina Ida Zeppieri (all WWU M\"unster) and Harald Garcke (University Regensburg) for useful discussions and links to literature. Furthermore, we thank the anonymous referees for many suggestions that substantially improved the paper.

\bibliographystyle{plain}
\bibliography{nonlocal}

\begin{thebibliography}{10}

\bibitem{bakhtacontrol}
Athmane Bakhta, Eric Cances, Virginie Ehrlacher, and Thomas Lelievre.
\newblock Control of atom fluxes for the production of solar cell devices.

\bibitem{Balague2013}
D.~Balagué, J.A. Carrillo, T.~Laurent, and G.~Raoul.
\newblock Nonlocal interactions by repulsive–attractive potentials: Radial
  ins/stability.
\newblock {\em Physica D: Nonlinear Phenomena}, 260:5 -- 25, 2013.
\newblock Emergent Behaviour in Multi-particle Systems with Non-local
  Interactions.

\bibitem{baldo1990minimal}
Sisto Baldo.
\newblock Minimal interface criterion for phase transitions in mixtures of
  {C}ahn-{H}illiard fluids.
\newblock In {\em Annales de l'IHP Analyse non lin{\'e}aire}, volume~7, pages
  67--90, 1990.

\bibitem{barrett2001fully}
John~W Barrett, James~F Blowey, and Harald Garcke.
\newblock On fully practical finite element approximations of degenerate
  {C}ahn-{H}illiard systems.
\newblock {\em ESAIM: Mathematical Modelling and Numerical Analysis},
  35(4):713--748, 2001.

\bibitem{bedrossian2011global}
Jacob Bedrossian.
\newblock Global minimizers for free energies of subcritical aggregation
  equations with degenerate diffusion.
\newblock {\em Applied Mathematics Letters}, 24(11):1927--1932, 2011.

\bibitem{bedrossian2011local}
Jacob Bedrossian, Nancy Rodr{\'\i}guez, and Andrea~L. Bertozzi.
\newblock Local and global well-posedness for aggregation equations and
  {P}atlak--{K}eller--{S}egel models with degenerate diffusion.
\newblock {\em Nonlinearity}, 24(6):1683, 2011.

\bibitem{VidURL}
Judith Berendsen, Martin Burger, and Jan-Frederik Pietschmann.
\newblock Videos of the one-dimensional dynamics.
\newblock http://www.jfpietschmann.eu/nonlocal.

\bibitem{bertozzi2011lp}
Andrea~L. Bertozzi, Thomas Laurent, and Jes{\'u}s Rosado.
\newblock {L}$^p$ theory for the multidimensional aggregation equation.
\newblock {\em Communications on Pure and Applied Mathematics}, 64(1):45--83,
  2011.

\bibitem{bertozzi2009existence}
Andrea~L Bertozzi and Dejan Slepcev.
\newblock Existence and uniqueness of solutions to an aggregation equation with
  degenerate diffusion.
\newblock {\em Communications on Pure and Applied Analysis}, 9(6):1617, 2009.

\bibitem{boi2000modelling}
Silvia Boi, Vincenzo Capasso, and Daniela Morale.
\newblock Modeling the aggregative behavior of ants of the species polyergus
  rufescens.
\newblock {\em Nonlinear Analysis: Real World Applications}, 1(1):163--176,
  2000.

\bibitem{Braides2002}
Andrea Braides.
\newblock {\em $\Gamma$-convergence for beginners}, volume~22 of {\em Oxford
  Lecture Series in Mathematics and its Applications}.
\newblock Oxford University Press, Oxford, 2002.

\bibitem{bronsard1998multi}
Lia Bronsard, Harald Garcke, and Barbara Stoth.
\newblock A multi-phase {M}ullins--{S}ekerka system: Matched asymptotic
  expansions and an implicit time discretisation for the geometric evolution
  problem.
\newblock {\em Proceedings of the Royal Society of Edinburgh: Section A
  Mathematics}, 128(03):481--506, 1998.

\bibitem{bruna2012diffusion}
Maria Bruna and S.~Jonathan Chapman.
\newblock Diffusion of multiple species with excluded-volume effects.
\newblock {\em The Journal of chemical physics}, 137(20):204116, 2012.

\bibitem{bruna2012excluded}
Maria Bruna and S.~Jonathan Chapman.
\newblock Excluded-volume effects in the diffusion of hard spheres.
\newblock {\em Physical Review E}, 85(1):011103, 2012.

\bibitem{Burger2007}
Martin Burger, Vincenzo Capasso, and Daniela Morale.
\newblock On an aggregation model with long and short range interactions.
\newblock {\em Nonlinear Anal. Real World Appl.}, 8(3):939--958, 2007.

\bibitem{Burger2008}
Martin Burger and Marco Di~Francesco.
\newblock Large time behavior of nonlocal aggregation models with nonlinear
  diffusion.
\newblock {\em Netw. Heterog. Media}, 3(4):749--785, 2008.

\bibitem{burger2006}
Martin Burger, Marco Di~Francesco, and Yasmin Dolak-Struss.
\newblock The {K}eller-{S}egel model for chemotaxis with prevention of
  overcrowding: linear vs.\ nonlinear diffusion.
\newblock {\em SIAM J. Math. Anal.}, 38(4):1288--1315 (electronic), 2006.

\bibitem{Burger2016b}
Martin Burger, Marco Di~Francesco, Simone Fagioli, and Angela Stevens.
\newblock Segregated stationary solutions of a cross-diffusion system with
  nonlocal interaction.
\newblock {\em Preprint}, 2016.

\bibitem{Burger2013}
Martin Burger, Marco Di~Francesco, and Marzena Franek.
\newblock Stationary states of quadratic diffusion equations with long-range
  attraction.
\newblock {\em Commun. Math. Sci.}, 11(3):709--738, 2013.

\bibitem{Burger2010}
Martin Burger, Marco Di~Francesco, Jan-Frederik Pietschmann, and B{\"a}rbel
  Schlake.
\newblock Nonlinear cross-diffusion with size exclusion.
\newblock {\em SIAM J. Math. Anal.}, 42(6):2842--2871, 2010.

\bibitem{burger2008asymptotic}
Martin Burger, Yasmin Dolak-Struss, Christian Schmeiser, et~al.
\newblock Asymptotic analysis of an advection-dominated chemotaxis model in
  multiple spatial dimensions.
\newblock {\em Communications in Mathematical Sciences}, 6(1):1--28, 2008.

\bibitem{Burger2014}
Martin Burger, Razvan Fetecau, and Yanghong Huang.
\newblock Stationary states and asymptotic behavior of aggregation models with
  nonlinear local repulsion.
\newblock {\em SIAM Journal on Applied Dynamical Systems}, 13(1):397--424,
  2014.

\bibitem{Burger2016}
Martin Burger, Sabine Hittmeir, Helene Ranetbauer, and Marie-Therese Wolfram.
\newblock Lane formation by side-stepping.
\newblock {\em SIAM Journal on Mathematical Analysis}, 48(2):981--1005, 2016.

\bibitem{burger2012nonlinear}
Martin Burger, B\"arbel Schlake, and Marie-Therese Wolfram.
\newblock Nonlinear {P}oisson--{N}ernst--{P}lanck equations for ion flux
  through confined geometries.
\newblock {\em Nonlinearity}, 25(4):961, 2012.

\bibitem{calvez2006volume}
Vincent Calvez and Jos{\'e}~A Carrillo.
\newblock Volume effects in the {K}eller--{S}egel model: energy estimates
  preventing blow-up.
\newblock {\em Journal de math{\'e}matiques pures et appliqu{\'e}es},
  86(2):155--175, 2006.

\bibitem{Canizo2015}
Jos{\'e}~A. Ca{\~n}izo, Jos{\'e}~A. Carrillo, and Francesco~S. Patacchini.
\newblock Existence of compactly supported global minimisers for the
  interaction energy.
\newblock {\em Arch. Ration. Mech. Anal.}, 217(3):1197--1217, 2015.

\bibitem{Carrillo2016}
Jos\'e~A. Carrillo, Matias~G. Delgadino, and Antoine Mellet.
\newblock Regularity of local minimizers of the interaction energy via obstacle
  problems.
\newblock {\em Communications in Mathematical Physics}, 343(3):747--781, 2016.

\bibitem{Carrillo2011}
Jos\'e~A. Carrillo, M.~{D}i Francesco, A.~Figalli, T.~Laurent, and
  D.~Slep{\v{c}}ev.
\newblock Global-in-time weak measure solutions and finite-time aggregation for
  nonlocal interaction equations.
\newblock {\em Duke Math. J.}, 156(2):229--271, 2011.

\bibitem{carrillo2016nonlinear}
Jos\'e~A. Carrillo, Sabine Hittmeir, Bruno Volzone, and Yao Yao.
\newblock Nonlinear aggregation-diffusion equations: Radial symmetry and long
  time asymptotics.
\newblock {\em arXiv preprint arXiv:1603.07767}, 2016.

\bibitem{chayes2013aggregation}
Lincoln Chayes, Inwon Kim, and Yao Yao.
\newblock An aggregation equation with degenerate diffusion: qualitative
  property of solutions.
\newblock {\em SIAM Journal on Mathematical Analysis}, 45(5):2995--3018, 2013.

\bibitem{chen2014minimal}
Yuxin Chen and Theodore Kolokolnikov.
\newblock A minimal model of predator--swarm interactions.
\newblock {\em Journal of The Royal Society Interface}, 11(94):20131208, 2014.

\bibitem{Choksi2015}
Rustum Choksi, Razvan~C. Fetecau, and Ihsan Topaloglu.
\newblock On minimizers of interaction functionals with competing attractive
  and repulsive potentials.
\newblock {\em Ann. Inst. H. Poincar\'e Anal. Non Lin\'eaire},
  32(6):1283--1305, 2015.

\bibitem{Cicalese2015}
Marco Cicalese, Lucia~De Luca, Matteo Novaga, and Marcello Ponsiglione.
\newblock Ground states of a two phase model with cross and self attractive
  interactions, 2015.

\bibitem{dolak2005keller}
Yasmin Dolak and Christian Schmeiser.
\newblock The {K}eller--{S}egel model with logistic sensitivity function and
  small diffusivity.
\newblock {\em SIAM Journal on Applied Mathematics}, 66(1):286--308, 2005.

\bibitem{Dreher2012}
Michael Dreher and Ansgar J{\"u}ngel.
\newblock Compact families of piecewise constant functions in.
\newblock {\em Nonlinear Analysis: Theory, Methods {\&} Applications},
  75(6):3072 -- 3077, 2012.

\bibitem{dyson2010existence}
Janet Dyson, Stephen~A Gourley, Rosanna Villella-Bressan, and Glenn~F Webb.
\newblock Existence and asymptotic properties of solutions of a nonlocal
  evolution equation modeling cell-cell adhesion.
\newblock {\em SIAM Journal on Mathematical Analysis}, 42(4):1784--1804, 2010.

\bibitem{elliott1997diffusional}
Charles~M Elliott and Harald Garcke.
\newblock Diffusional phase transitions in multicomponent systems with a
  concentration dependent mobility matrix.
\newblock {\em Physica D: Nonlinear Phenomena}, 109(3):242--256, 1997.

\bibitem{Evans1998}
L.~C. Evans.
\newblock {\em Partial Differential Equations}, volume~19 of {\em Graduate
  Studies in Mathematics}.
\newblock American Mathematical Society, 1998.

\bibitem{Fellner2010}
Klemens Fellner and Gaël Raoul.
\newblock Stable stationary states of non-local interaction equations.
\newblock {\em Mathematical Models and Methods in Applied Sciences},
  20(12):2267--2291, 2010.

\bibitem{garcke1998anisotropic}
Harald Garcke, Britta Nestler, and Barbara Stoth.
\newblock On anisotropic order parameter models for multi-phase systems and
  their sharp interface limits.
\newblock {\em Physica D: Nonlinear Phenomena}, 115(1):87--108, 1998.

\bibitem{garcke2000singular}
Harald Garcke, Amy Novick-Cohen, et~al.
\newblock A singular limit for a system of degenerate {C}ahn-{H}illiard
  equations.
\newblock {\em Advances in Differential Equations}, 5(4-6):401--434, 2000.

\bibitem{Gilbarg1983}
David Gilbarg and Neill Trudinger.
\newblock {\em Elliptic Partial Differential Equations of Second Order}.
\newblock Springer-Verlag, 1983.

\bibitem{Osher2016}
Roland Glowinski, Stanley~J. Osher, and Wotao~(Eds.) Yin.
\newblock {\em Splitting Methods in Communication and Imaging, Science, and
  Engineering}.
\newblock Scientific Computation. Springer International Publishing, 2016.

\bibitem{Griepentrog2004}
Jens~A. Griepentrog.
\newblock On the unique solvability of a nonlocal phase separation problem for
  multicomponent systems.
\newblock {\em Banach Center Publications, WIAS preprint 898}, 66:153--164,
  2004.

\bibitem{herz2016including}
Matthias Herz and Peter Knabner.
\newblock Including van der {W}aals forces in diffusion-convection
  equations-modeling, analysis, and numerical simulations.
\newblock {\em arXiv preprint arXiv:1608.08431}, 2016.

\bibitem{Hestenes1969}
Magnus~R. Hestenes.
\newblock Multiplier and gradient methods.
\newblock {\em J. Optimization Theory Appl.}, 4:303--320, 1969.

\bibitem{hillen2009user}
Thomas Hillen and Kevin~J. Painter.
\newblock A user’s guide to pde models for chemotaxis.
\newblock {\em Journal of mathematical biology}, 58(1-2):183--217, 2009.

\bibitem{Hinze2009}
M.~Hinze, R.~Pinnau, M.~Ulbrich, and S.~Ulbrich.
\newblock {\em Optimization with {PDE} constraints}, volume~23 of {\em
  Mathematical Modelling: Theory and Applications}.
\newblock Springer, New York, 2009.

\bibitem{Juengel2014}
Ansgar J{\"u}ngel.
\newblock The boundedness-by-entropy method for cross-diffusion systems.
\newblock {\em Nonlinearity}, 28(6):1963, 2015.

\bibitem{kaib2016stationary}
Gunnar Kaib.
\newblock Stationary states of an aggregation equation with degenerate
  diffusion and bounded attractive potential.
\newblock {\em arXiv preprint arXiv:1604.07298}, 2016.

\bibitem{Lieberman1996}
Gary~M Lieberman.
\newblock {\em Second order parabolic differential equations}.
\newblock World scientific, 1996.

\bibitem{mogilner1999non}
Alexander Mogilner and Leah Edelstein-Keshet.
\newblock A non-local model for a swarm.
\newblock {\em Journal of Mathematical Biology}, 38(6):534--570, 1999.

\bibitem{nurnberg2009numerical}
Robert N{\"u}rnberg.
\newblock Numerical simulations of immiscible fluid clusters.
\newblock {\em Applied Numerical Mathematics}, 59(7):1612--1628, 2009.

\bibitem{painter2009continuous}
Kevin~J. Painter.
\newblock Continuous models for cell migration in tissues and applications to
  cell sorting via differential chemotaxis.
\newblock {\em Bulletin of Mathematical Biology}, 71(5):1117--1147, 2009.

\bibitem{painter2002volume}
Kevin~J. Painter and Thomas Hillen.
\newblock Volume-filling and quorum-sensing in models for chemosensitive
  movement.
\newblock {\em Can. Appl. Math. Quart}, 10(4):501--543, 2002.

\bibitem{Powel69}
M.~J.~D. Powell.
\newblock {A method for nonlinear constraints in minimization problems}.
\newblock In R.~Fletcher, editor, {\em Optimization}, pages 283--298. Academic
  Press, New York, 1969.

\bibitem{Schlake2011}
B\"arbel Schlake.
\newblock {\em Mathematical Models for Particle Transport: Crowded Motion}.
\newblock PhD thesis, WWU M\"unster, 2011.

\bibitem{Simione2014}
Robert Simione.
\newblock {\em Properties of {M}inimizers of {N}onlocal {I}nteraction
  {E}nergy}.
\newblock ProQuest LLC, Ann Arbor, MI, 2014.
\newblock Thesis (Ph.D.)--Carnegie Mellon University.

\bibitem{Simione2015}
Robert Simione, Dejan Slep{\v{c}}ev, and Ihsan Topaloglu.
\newblock Existence of ground states of nonlocal-interaction energies.
\newblock {\em J. Stat. Phys.}, 159(4):972--986, 2015.

\bibitem{simpson2009multi}
Matthew~J Simpson, Kerry~A Landman, and Barry~D Hughes.
\newblock Multi-species simple exclusion processes.
\newblock {\em Physica A: Statistical Mechanics and its Applications},
  388(4):399--406, 2009.

\bibitem{slepcev2008coarsening}
Dejan Slepcev.
\newblock Coarsening in nonlocal interfacial systems.
\newblock {\em SIAM Journal on Mathematical Analysis}, 40(3):1029--1048, 2008.

\bibitem{stanczy2011evolution}
Robert Sta{\'n}czy.
\newblock On an evolution system describing self-gravitating particles in
  microcanonical setting.
\newblock {\em Monatshefte f{\"u}r Mathematik}, 162(2):197--224, 2011.

\bibitem{Wurst2015}
Jan-Eric Wurst.
\newblock {\em Hp-Finite Elements for PDE-Constrained Optimization}.
\newblock PhD thesis, 2015.

\bibitem{Zamponi2015}
Nicola Zamponi and Ansgar J{\"u}ngel.
\newblock Analysis of degenerate cross-diffusion population models with volume
  filling.
\newblock In {\em Annales de l'Institut Henri Poincare (C) Non Linear
  Analysis}. Elsevier, 2015.

\bibitem{zm2014}
Jonathan Zinsl and Daniel Matthes.
\newblock Transport distances and geodesic convexity for systems of degenerate
  diffusion equations.
\newblock {\em Calc. Var. Partial Differential Equations}, 54(4):3397--3438,
  2015.

\end{thebibliography}

\end{document}